\documentclass[reqno,10pt,centertags]{amsart}
\usepackage{amsmath,amsthm,amscd,amssymb,latexsym,esint,upref,enumerate,
color,verbatim,yfonts}
\date{\today}
\usepackage{hyperref}

\newcommand*{\mailto}[1]{\href{mailto:#1}{\nolinkurl{#1}}}
\newcommand{\arxiv}[1]{\href{http://arxiv.org/abs/#1}{arXiv: #1}}

\makeatletter
\def\theequation{\@arabic\c@equation}

\newcommand{\oT}{H}

\newcommand{\bbN}{{\mathbb{N}}}
\newcommand{\bbR}{{\mathbb{R}}}

\newcommand{\bbC}{{\mathbb{C}}}

\newcommand{\cB}{{\mathcal B}}

\newcommand{\cD}{{\mathcal D}}

\newcommand{\cH}{{\mathcal H}}
\newcommand{\cI}{{\mathcal I}}
\newcommand{\cJ}{{\mathcal J}}
\newcommand{\cK}{{\mathcal K}}

\newcommand{\cN}{{\mathcal N}}

\newcommand{\cS}{{\mathcal S}}

\newcommand{\cX}{{\mathcal X}}

\newcommand{\no}{\nonumber}
\newcommand{\lb}{\label}
\newcommand{\f}{\frac}

\newcommand{\ol}{\overline}
\newcommand{\bs}{\backslash}

\newcommand{\wti}{\widetilde}

\newcommand{\la}{\lambda}
\newcommand{\al}{\alpha}

\newcommand{\Oh}{O}

\newcommand{\loc}{\text{\rm{loc}}}

\newcommand{\ran}{\operatorname{ran}}

\newcommand{\dom}{\operatorname{dom}}
\newcommand{\supp}{\operatorname{supp}}
\newcommand{\linspan}{\operatorname{lin.span}}

\renewcommand{\Re}{\operatorname{Re}}
\renewcommand{\Im}{\operatorname{Im}}

\newcommand{\bi}{\bibitem}
\newcommand{\hatt}{\widehat}

\newcommand{\essran}{\text{\rm ess.ran}}

\newcommand{\high}[1]{{\raisebox{1mm}{$#1$}}}

\DeclareMathOperator*{\slim}{s-lim}

\numberwithin{equation}{section}

\newtheorem{theorem}{Theorem}[section]

\newtheorem{lemma}[theorem]{Lemma}
\newtheorem{corollary}[theorem]{Corollary}
\newtheorem{hypothesis}[theorem]{Hypothesis}

\theoremstyle{definition}
\newtheorem{definition}[theorem]{Definition}
\newtheorem{remark}[theorem]{Remark}

\begin{document}

\title[Donoghue-Type $m$-Functions]{Donoghue-Type $m$-Functions for
Schr\"odinger Operators with Operator-Valued Potentials}

\author[F.\ Gesztesy]{Fritz Gesztesy}
\address{Department of Mathematics,
University of Missouri, Columbia, MO 65211, USA}
\email{\mailto{gesztesyf@missouri.edu}}
%\email{gesztesyf@missouri.edu}
\urladdr{\url{http://faculty.missouri.edu/~gesztesyf}}
%\urladdr{http://faculty.missouri.edu/~gesztesyf}

\author[S.\ N.\ Naboko]{Sergey N.\ Naboko}
\address{Department of Mathematical Physics, St.~Petersburg  State University, Ulianovskaia 1, NIIF, St.~Peterhof, St.~Petersburg,
Russian Federation, 198504}
\email{\mailto{sergey.naboko@gmail.com}}
%\email{sergey.naboko@gmail.com}

\author[R.\ Weikard]{Rudi Weikard}
\address{Department of Mathematics, University of
Alabama at Birmingham, Birmingham, AL 35294, USA}
\email{\mailto{rudi@math.uab.edu}}
%\email{rudi@math.uab.edu}
\urladdr{\url{http://www.math.uab.edu/~rudi/}}
%\urladdr{http://www.math.uab.edu/~rudi/}

\author[M.\ Zinchenko]{Maxim Zinchenko}
\address{Department of Mathematics and Statistics,
University of New Mexico, Albuquerque, NM 87131, USA}
\email{\mailto{maxim@math.unm.edu}}
%\email{maxim@math.unm.edu}
\urladdr{\url{http://www.math.unm.edu/~maxim/}}
%\urladdr{http://www.math.unm.edu/~maxim/}

%\dedicatory{}
%\date{\today}
\date{\today}
\thanks{S.N. is supported by grants NCN 2013/09/BST1/04319, RFBR 12-01-00215-a, and Marie Curie grant PIIF-GA-2011-299919; Research of M.Z. is supported in part by a Simons Foundation grant CGM--281971.}

%%%%%%%%%%%%%%%%%%%%%%%%%%%%%%%%%%%%%%%%
\begin{abstract}
Given a complex, separable Hilbert space $\cH$, we consider differential expressions of the 
type $\tau =  - (d^2/dx^2) I_{\cH} + V(x)$, with $x \in (x_0,\infty)$ for some $x_0 \in \bbR$, 
or $x \in \bbR$ (assuming the limit-point property of $\tau$ at $\pm \infty$). Here $V$ denotes a bounded operator-valued potential $V(\cdot) \in \cB(\cH)$ such that $V(\cdot)$ is weakly measurable, the operator norm $\|V(\cdot)\|_{\cB(\cH)}$ is locally integrable, and $V(x) = V(x)^*$ a.e.\ on $x \in [x_0,\infty)$ or $x \in \bbR$. 
We focus on two major cases. First, on $m$-function theory for self-adjoint half-line $L^2$-realizations $H_{+,\alpha}$ in $L^2((x_0,\infty); dx; \cH)$ (with $x_0$ a regular endpoint for $\tau$, associated with the self-adjoint boundary condition $\sin(\alpha)u'(x_0) + \cos(\alpha)u(x_0)=0$,
indexed by the self-adjoint operator $\alpha = \alpha^* \in \cB(\cH)$), and second, on 
$m$-function theory for self-adjoint full-line $L^2$-realizations $H$ of $\tau$ in $L^2(\bbR; dx; \cH)$. 

In a nutshell, a Donoghue-type $m$-function $M_{A,\cN_i}^{Do}(\cdot)$ associated with self-adjoint extensions $A$ of a closed, symmetric operator $\dot A$ in $\cH$ with deficiency spaces
$\cN_z =  \ker \big(\dot A\high{^*} - z I_{\cH}\big)$ and corresponding orthogonal projections $P_{\cN_z}$ onto $\cN_z$ is given by
\begin{align*}
M_{A,\cN_i}^{Do}(z)&=P_{\cN_i} (zA + I_\cH)(A - z I_{\cH})^{-1}
P_{\cN_i}\big\vert_{\cN_i}      \\
&=zI_{\cN_i} + (z^2+1) P_{\cN_i} (A - z I_{\cH})^{-1}
P_{\cN_i}\big\vert_{\cN_i} \, , \quad  z\in \bbC\backslash \bbR.
\end{align*}
In the concrete case of half-line and full-line Schr\"odinger operators, the role of $\dot A$ is 
played by a suitably defined minimal Schr\"odinger operator $H_{+,\min}$ in 
$L^2((x_0,\infty); dx; \cH)$ and $H_{\min}$ in $L^2(\bbR; dx; \cH)$, both of which will be proven 
to be completely non-self-adjoint. The latter property is used to prove that if $H_{+,\alpha}$ 
in $L^2((x_0,\infty); dx; \cH)$, respectively, $H$ in $L^2(\bbR; dx; \cH)$, are self-adjoint 
extensions of $H_{+,\min}$, respectively,  $H_{\min}$, then the corresponding operator-valued measures in the Herglotz--Nevanlinna representations of the Donoghue-type 
$m$-functions $M_{H_{+,\alpha}, \cN_{+,i}}^{Do}(\cdot)$ and $M_{H, \cN_i}^{Do}(\cdot)$ encode the 
entire spectral information of $H_{+,\alpha}$, respectively, $H$. 
\end{abstract}
%%%%%%%%%%%%%%%%%%%%%%%%%%%%%%%%%%%%%%%%

\maketitle

\newpage 

{\scriptsize \tableofcontents}

%%%%%%%%%%%%%%%%%%%%%%%%%%%%%%%%%%%%%%%%
%%%%%%%%%%%%%%%%%%%%%%%%%%%%%%%%%%%%%%%%
\section{Introduction} \lb{s1}
%%%%%%%%%%%%%%%%%%%%%%%%%%%%%%%%%%%%%%%%
%%%%%%%%%%%%%%%%%%%%%%%%%%%%%%%%%%%%%%%%

The principal topic of this paper centers around basic spectral theory for self-adjoint 
Schr\"odinger operators with bounded operator-valued potentials on a half-line as well as 
on the full real line, focusing on Donoghue-type $m$-function theory, eigenfunction expansions, 
and a version of the spectral theorem. More precisely, given a complex,
separable Hilbert space $\cH$, we consider differential expressions $\tau$ of the type
\begin{equation}
\tau =  - (d^2/dx^2) I_{\cH} + V(x),    \lb{1.1}
\end{equation}
with $x \in (x_0,\infty)$ or $x \in \bbR$ ($x_0 \in \bbR$ a reference point), and $V$ a bounded
operator-valued potential $V(\cdot) \in \cB(\cH)$ such that $V(\cdot)$ is weakly measurable, 
the operator norm $\|V(\cdot)\|_{\cB(\cH)}$ is locally integrable, and $V(x) = V(x)^*$ a.e.\ on 
$x \in [x_0,\infty)$ or $x \in \bbR$. The self-adjoint operators
in question are then half-line $L^2$-realizations of $\tau$ in $L^2((x_0,\infty); dx; \cH)$,
with $x_0$ assumed to be a regular endpoint for $\tau$, and hence with appropriate boundary
conditions at $x_0$ (cf.\ \eqref{1.26}) on one hand, and full-line $L^2$-realizations of $\tau$ in
$L^2(\bbR; dx; \cH)$ on the other.

The case of Schr\"odinger operators with operator-valued potentials under various
continuity or smoothness hypotheses on $V(\cdot)$, and under various self-adjoint
boundary conditions on bounded and unbounded open intervals, received considerable attention
in the past. In the special case where $\dim(\cH)<\infty$, that is, in the case of Schr\"odinger
operators with matrix-valued potentials, the literature is so voluminous that we cannot
possibly describe individual references and hence we primarily refer to the monographs
\cite{AM63}, \cite{RK05}, and the references cited therein. We note that the
finite-dimensional case, $\dim(\cH) < \infty$, as discussed in \cite{BL00}, is of
considerable interest as it represents an important ingredient in some proofs of
Lieb--Thirring inequalities (cf.\ \cite{LW00}). For the particular case of Schr\"odinger-type 
operators corresponding to the differential expression
$\tau = - (d^2/dx^2) I_{\cH} + A + V(x)$ on a bounded interval $(a,b) \subset \bbR$ with either
$A=0$ or $A$ a self-adjoint operator satisfying $A\geq c I_{\cH}$ for some $c>0$, we refer to 
the list of references in \cite{GWZ13b}. For earlier results on various aspects of boundary value problems, spectral theory, and scattering theory in the half-line case $(a,b) =(0,\infty)$, we refer, 
for instance, to \cite{Al06a}, \cite{AM10}, \cite{De08}, \cite{Go68}--\cite{GG69}, 
\cite[Chs.~3,4]{GG91}, \cite{GM76},
\cite{HMM13}, \cite{KL67}, \cite{Mo07}, \cite{Mo10}, \cite{Ro60}, \cite{Sa71}, \cite{Tr00}
(the case of the real line is discussed in \cite{VG70}). Our treatment of spectral 
theory for half-line and full-line Schr\"odinger operators in $L^2((x_0,\infty); dx; \cH)$ and in 
$L^2(\bbR; dx; \cH)$, respectively, in \cite{GWZ13}, \cite{GWZ13b} represents the most general 
one to date. 

Next, we briefly turn to Donoghue-type $m$-functions which abstractly can be introduced 
as follows (cf.\ \cite{GKMT01}, \cite{GMT98}). 
Given a self-adjoint extension $A$ of a densely defined, closed, symmetric operator $\dot A$ 
in $\cK$ (a complex, separable Hilbert space) and the deficiency subspace $\cN_i$ of $\dot A$ 
in $\cK$, with 
\begin{equation}
 \cN_i = \ker \big(\dot A\high{^*} - i I_{\cK}\big), \quad   
\dim \, (\cN_i)=k\in \bbN \cup \{\infty\},     \lb{1.2}
\end{equation} 
the Donoghue-type $m$-operator $M_{A,\cN_i}^{Do} (z) \in\cB(\cN_i)$ associated with the pair
 $(A,\cN_i)$  is given by
\begin{align}
\begin{split}
M_{A,\cN_i}^{Do}(z)&=P_{\cN_i} (zA + I_\cK)(A - z I_{\cK})^{-1}
P_{\cN_i}\big\vert_{\cN_i}      \\
&=zI_{\cN_i} + (z^2+1) P_{\cN_i} (A - z I_{\cK})^{-1}
P_{\cN_i} \big\vert_{\cN_i}\,, \quad  z\in \bbC\backslash \bbR,     \lb{1.3}
\end{split}
\end{align}
with $I_{\cN_i}$ the identity operator in $\cN_i$, and $P_{\cN_i}$ the orthogonal projection in 
$\cK$ onto $\cN_i$. Then $M_{A,\cN_i}^{Do}(\cdot)$ is a $\cB(\cN_i)$-valued 
Nevanlinna--Herglotz function that admits the representation 
\begin{equation}
M_{A,\cN_i}^{Do}(z) = \int_\bbR
d\Omega_{A,\cN_i}^{Do}(\lambda) \bigg[\f{1}{\lambda-z} -
\f{\lambda}{\lambda^2 + 1}\bigg], \quad z\in\bbC\backslash\bbR,    \lb{1.4}
\end{equation}
where the $\cB(\cN_i)$-valued measure $\Omega_{A,\cN_i}^{Do}(\cdot)$ satisfies 
\eqref{5.6}--\eqref{5.8}. 

In the concrete case of regular half-line Schr\"odinger operators in $L^2((x_0,\infty); dx)$ with a scalar potential, Donoghue \cite{Do65} introduced the analog of \eqref{1.3} and used it to settle certain inverse spectral problems. 

As has been shown in detail in \cite{GKMT01}, \cite{GMT98}, \cite{GT00}, Donoghue-type 
$m$-functions naturally lead to Krein-type resolvent formulas as well as linear fractional transformations relating two different self-adjoint extensions of $\dot A$. However, in this 
paper we are particularly interested in the question under which conditions on $\dot A$, 
the spectral information on its self-adjoint extension $A$, contained in its family of spectral projections $\{E_A(\lambda)\}_{\lambda \in \bbR}$, is already encoded in the 
$\cB(\cN_i)$-valued measure $\Omega_{A,\cN_i}^{Do}(\cdot)$. As shown in 
Corollary \ref{c5.8}, this is the case if and only if $\dot A$ is completely non-self-adjoint 
in $\cK$ and we will apply this to half-line and full-line Schr\"odinger operators with 
$\cB(\cH)$-valued potentials.  

In the general case of $\cB(\cH)$-valued potentials on the right half-line 
$(x_0,\infty)$, assuming Hypothesis \ref{h6.1}\,$(i)$, we introduce minimal and maximal, 
operators $H_{+,\min}$ and $H_{+,\max}$ 
in $L^2((x_0,\infty); dx; \cH)$ associated to $\tau$, and self-adjoint extensions 
$H_{+,\alpha}$ of $H_{+,\min}$ (cf.\ \eqref{3.2}, \eqref{3.4}, \eqref{3.9}) and given the generating 
property of the deficiency spaces $\cN_{+,z} = \ker(H_{+,\min} - z I)$, 
$z\in \bbC \backslash \bbR$, proven in Theorem \ref{t6.2}, conclude that $H_{+,\min}$ is 
completely non-self-adjoint (i.e., it has no nontrivial invariant subspace in 
$L^2((x_0,\infty); dx; \cH)$ on which it is self-adjoint). 

According to \eqref{1.3}, the right half-line Donoghue-type $m$-function corresponding to 
$H_{+,\alpha}$ and $\cN_{+,i}$ is given by 
\begin{align}
\begin{split}
M_{H_{+,\alpha}, \cN_{+,i}}^{Do} (z,x_0) &= P_{\cN_{+,i}} (z H_{+,\alpha} + I)
(H_{+,\alpha} - z I)^{-1} P_{\cN_{+,i}} \big|_{\cN_{+,i}}     \lb{1.5} \\
&= \int_\bbR d\Omega_{H_{+,\alpha},\cN_{+,i}}^{Do}(\lambda,x_0) \bigg[\f{1}{\lambda-z} -
\f{\lambda}{\lambda^2 + 1}\bigg], \quad z\in\bbC\backslash\bbR,
\end{split}
\end{align}
where $\Omega_{H_{+,\alpha},\cN_{+,i}}^{Do}(\, \cdot\, , x_0)$ satisfies the analogs of 
\eqref{5.6}--\eqref{5.8}.

Combining Corollary \ref{c5.8} with the complete non-self-adjointness of $H_{+,\min}$ proves that the entire spectral information for $H_{+,\alpha}$, contained in the corresponding family of spectral projections $\{E_{H_{+,\alpha}}(\lambda)\}_{\lambda \in \bbR}$ in $L^2((x_0,\infty); dx; \cH)$, is
already encoded in the $\cB(\cN_{+,i})$-valued measure 
$\Omega_{H_{+,\alpha},\cN_{+,i}}^{Do}(\, \cdot \,,x_0)$ (including multiplicity properties of the spectrum of $H_{+,\alpha}$). 

An explicit computation of $M_{H_{+,\alpha}, \cN_{+,i}}^{Do} (z,x_0)$ then yields 
\begin{align}
& M_{H_{+,\alpha}, \cN_{\pm,i}}^{Do} (z,x_0)
= \pm \sum_{j,k \in \cJ} \big(e_j, m_{+,\alpha}^{Do}(z,x_0) e_k\big)_{\cH}    \no \\
& \quad \times (\psi_{+,\alpha}(i, \, \cdot \, ,x_0)
[\Im(m_{+,\alpha}(i,x_0))]^{-1/2}
e_k, \, \cdot \, )_{L^2((x_0,\infty); dx; \cH))}     \no \\
& \quad \times \psi_{+,\alpha}(i, \, \cdot \, ,x_0)
[\Im(m_{+,\alpha}(i,x_0))]^{-1/2} e_j \big|_{\cN_{+,i}},
\quad z \in \bbC \backslash \bbR,     \lb{1.6}
\end{align}
where $\{e_j\}_{j \in \cJ}$ is an orthonormal basis in $\cH$ ($\cJ \subseteq \bbN$ an
appropriate index set) and the $\cB(\cH)$-valued Nevanlinna--Herglotz functions
$m_{+,\alpha}^{Do}(\, \cdot \, , x_0)$ are given by
\begin{align}
m_{+,\alpha}^{Do}(z,x_0) &= [\Im(m_{+,\alpha}(i,x_0))]^{-1/2}
[m_{+,\alpha}(z,x_0) - \Re(m_{+,\alpha}(i,x_0))]      \no \\
& \quad \times [\Im(m_{+,\alpha}(i,x_0))]^{-1/2}     \lb{1.7} \\
&= d_{+, \alpha} + \int_\bbR d\omega_{+,\alpha}^{Do}(\lambda,x_0)
\bigg[\f{1}{\lambda-z} -
\f{\lambda}{\lambda^2 + 1}\bigg], \quad z\in\bbC\backslash\bbR.     \lb{1.8}
\end{align}
Here $d_{+,\alpha} = \Re(m_{+,\alpha}^{Do}(i,x_0)) \in \cB(\cH)$, and
\begin{equation}
\omega_{+,\alpha}^{Do}(\, \cdot \,,x_0) = [\Im(m_{+,\alpha}(i,x_0))]^{-1/2}
\rho_{+,\alpha}(\,\cdot\,,x_0) [\Im(m_{+,\alpha}(i,x_0))]^{-1/2}   \lb{1.9} 
\end{equation}
satisfies the analogs of \eqref{A.42a}, \eqref{A.42b}. In addition, $\psi_{+,\alpha}(\, \cdot \,,x,x_0)$ 
is the right half-line Weyl--Titchmarsh solution \eqref{3.58A}, and $m_{+,\alpha}(\, \cdot \,,x_0)$ 
represents the standard $\cB(\cH)$-valued right half-line Weyl--Titchmarsh $m$-function in \eqref{3.58A} with $\cB(\cH)$-valued measure $\rho_{+,\alpha}(\, \cdot \,,x_0)$ in its Nevanlinna--Herglotz 
representation \eqref{2.25}--\eqref{2.26}. 

This result shows that the entire spectral information for $H_{+,\alpha}$ is also contained in
the $\cB(\cH)$-valued measure $\omega_{+,\alpha}^{Do}(\, \cdot \,,x_0)$ (again, including multiplicity properties of the spectrum of $H_{+,\alpha}$). Naturally, the same facts apply to the 
left half-line $(- \infty,x_0)$. 

Turning to the full-line case assuming Hypotheis \ref{h2.8}, and denoting by $H$ the self-adjoint 
realization of $\tau$ in $L^2(\bbR; dx; \cH)$, we now decompose 
\begin{equation}
L^2(\bbR; dx; \cH) = L^2((-\infty,x_0); dx; \cH) \oplus L^2((x_0, \infty); dx; \cH),  \lb{1.10}
\end{equation}
and introduce the orthogonal projections $P_{\pm,x_0}$ of $L^2(\bbR; dx; \cH)$ onto
the left/right subspaces $L^2((x_0,\pm\infty); dx; \cH)$. Thus, we introduce the $2 \times 2$ block operator 
representation,
\begin{equation}
(H - z I)^{-1} = \begin{pmatrix} P_{-,x_0} (H - zI)^{-1} P_{-,x_0}
& P_{-,x_0} (H - z I)^{-1} P_{+,x_0} \\
P_{+,x_0} (H - z I)^{-1} P_{-,x_0} & P_{+,x_0} (H - z I)^{-1} P_{+,x_0}
\end{pmatrix},     \lb{1.11}
\end{equation}
and introduce with respect to the decomposition \eqref{1.10}, the minimal operator $H_{\min}$ 
in $L^2(\bbR; dx; \cH)$ via
\begin{align}
H_{\min} &:= H_{-,\min} \oplus H_{+,\min}, \quad
H_{\min}^* = H_{-,\min}^* \oplus H_{+,\min}^*,     \lb{1.12} \\
\cN_z &\, = \ker\big(H_{\min}^* - z I \big) =  \ker\big(H_{-,\min}^* - z I \big) \oplus
 \ker\big(H_{+,\min}^* - z I \big)     \no \\
&\, = \cN_{-,z} \oplus \cN_{+,z}, \quad z \in \bbC \backslash \bbR,    \lb{1.13} 
\end{align}
(see the additional comments concerning our choice of minimal operator in Section \ref{s6}, 
following \eqref{6.23a}).

According to \eqref{1.3}, the full-line Donoghue-type $m$-function is given by 
\begin{align}
\begin{split}
M_{H, \cN_i}^{Do} (z) &= P_{\cN_i} (z H + I)
(H - z I)^{-1} P_{\cN_i} \big|_{\cN_i},     \lb{1.14} \\
&= \int_\bbR d\Omega_{H,\cN_i}^{Do}(\lambda) \bigg[\f{1}{\lambda-z} -
\f{\lambda}{\lambda^2 + 1}\bigg], \quad z\in\bbC\backslash\bbR,
\end{split}
\end{align}
where $\Omega_{H,\cN_i}^{Do}(\cdot)$ satisfies the analogs of \eqref{5.6}--\eqref{5.8}
(resp., \eqref{A.42A}--\eqref{A.42b}). 

Combining Corollary \ref{c5.8} with the complete non-self-adjointness of $H_{\min}$ proves that the entire spectral information for $H$, contained in the corresponding family of spectral projections 
$\{E_{H}(\lambda)\}_{\lambda \in \bbR}$ in $L^2(\bbR; dx; \cH)$, is
already encoded in the $\cB(\cN_i)$-valued measure $\Omega_{H,\cN_i}^{Do}(\cdot)$ 
(including multiplicity properties of the spectrum of $H$). 

With respect to the decomposition \eqref{1.10}, one
can represent $M_{H, \cN_i}^{Do} (\cdot)$ as the $2 \times 2$ block operator, 
\begin{align}
& M_{H, \cN_i}^{Do} (\cdot) = \big(M_{H, \cN_i,\ell,\ell'}^{Do} (\cdot)\big)_{0 \leq \ell, \ell' \leq 1}
\no \\
& \quad = z \left(\begin{smallmatrix} P_{\cN_-,i} & 0 \\
0 & P_{\cN_+,i} \end{smallmatrix}\right)    \lb{1.15} \\
& \qquad + (z^2 + 1) \left(\begin{smallmatrix} 
P_{\cN_-,i} P_{-,x_0} (H - zI)^{-1} P_{-,x_0} P_{\cN_-,i} & \;\;\; 
P_{\cN_-,i} P_{-,x_0} (H - zI)^{-1} P_{+,x_0} P_{\cN_+,i}   \\ 
P_{\cN_+,i} P_{+,x_0} (H - zI)^{-1} P_{-,x_0} P_{\cN_-,i} & \;\;\; 
P_{\cN_+,i} P_{+,x_0} (H - zI)^{-1} P_{+,x_0} P_{\cN_+,i}  
\end{smallmatrix}\right),     \no
\end{align}
and utilizing the fact that 
\begin{align}
\begin{split}
&\big\{\hatt \Psi_{-,\alpha,j}(z,\, \cdot \, , x_0) =
P_{-,x_0} \psi_{-,\alpha}(z, \, \cdot \, ,x_0)[- (\Im(z)^{-1} m_{-,\alpha}(z,x_0)]^{-1/2} e_j,  \\
& \;\;\, \hatt \Psi_{+,\alpha,j}(z,\, \cdot \, , x_0) =
P_{+,x_0} \psi_{+,\alpha}(z, \, \cdot \, ,x_0)[(\Im(z)^{-1} m_{+,\alpha}(z,x_0)]^{-1/2}
e_j\big\}_{j \in \cJ}    \lb{1.16}
\end{split}
\end{align}
is an orthonormal basis for $\cN_{z} = \ker\big(H_{\min}^* - z I \big)$,
$z \in \bbC \backslash \bbR$, with $\{e_j\}_{j \in \cJ}$ an orthonormal basis for $\cH$, one 
eventually computes explicitly, 
\begin{align}
& M_{H, \cN_i,0,0}^{Do} (z) = \sum_{j,k \in \cJ}
(e_j, M_{\alpha,0,0}^{Do}(z,x_0) e_k)_{\cH}  \no \\
& \hspace*{3.1cm} \times
\big(\hatt \Psi_{-,\alpha,k}(i,\, \cdot \,,x_0), \, \cdot \, \big)_{L^2(\bbR; dx; \cH)}
\hatt \Psi_{-,\alpha,j}(i,\, \cdot \,,x_0),     \lb{1.17} \\
& M_{H, \cN_i,0,1}^{Do} (z) = \sum_{j,k \in \cJ}
(e_j, M_{\alpha,0,1}^{Do}(z,x_0) e_k)_{\cH}   \no \\
& \hspace*{3.1cm} \times
\big(\hatt \Psi_{+,\alpha,k}(i,\, \cdot \,,x_0), \, \cdot \, \big)_{L^2(\bbR; dx; \cH)}
\hatt \Psi_{-,\alpha,j}(i,\, \cdot \,,x_0),     \lb{1.18} \\
& M_{H, \cN_i,1,0}^{Do} (z) = \sum_{j,k \in \cJ}
(e_j, M_{\alpha,1,0}^{Do}(z,x_0) e_k)_{\cH}    \no \\
& \hspace*{3.1cm} \times
\big(\hatt \Psi_{-,\alpha,k}(i,\, \cdot \,,x_0), \, \cdot \, \big)_{L^2(\bbR; dx; \cH)}
\hatt \Psi_{+,\alpha,j}(i,\, \cdot \,,x_0),     \lb{1.19} \\
& M_{H, \cN_i,1,1}^{Do} (z) = \sum_{j,k \in \cJ}
(e_j, M_{\alpha,1,1}^{Do}(z,x_0) e_k)_{\cH}    \no \\
& \hspace*{3.1cm} \times
\big(\hatt \Psi_{+,\alpha,k}(i,\, \cdot \,,x_0), \, \cdot \, \big)_{L^2(\bbR; dx; \cH)}
\hatt \Psi_{+,\alpha,j}(i,\, \cdot \,,x_0),    \lb{1.20} \\
& \hspace*{8.95cm} z\in\bbC\backslash\bbR,     \no
\end{align}
with $M_{\alpha}^{Do}(\, \cdot \,,x_0)$ given by 
\begin{align}
\begin{split} 
M_{\alpha}^{Do} (z,x_0) &= T_{\alpha}^* M_{\alpha}(z,x_0) T_{\alpha}
+ E_{\alpha}     \lb{1.21} \\
&= D_{\alpha} + \int_\bbR d\Omega_{\alpha}^{Do}(\lambda,x_0) \bigg[\f{1}{\lambda-z} -
\f{\lambda}{\lambda^2 + 1}\bigg], \quad z\in\bbC\backslash\bbR,    
\end{split} 
\end{align}
Here $D_{\alpha} = \Re(M_{\alpha}^{Do}(i,x_0)) \in \cB\big(\cH^2\big)$, and
\begin{equation}
\Omega_{\alpha}^{Do}(\, \cdot \,,x_0) = T_{\alpha}^*
\Omega_{\alpha}(\,\cdot\,,x_0) T_{\alpha}   \lb{1.23} 
\end{equation}
satisfies the analogs of \eqref{A.42a}, \eqref{A.42b}. In addition, the $2 \times 2$ block operators $T_{\alpha} \in \cB\big(\cH^2\big)$ \big(with $T_{\alpha}^{-1} \in \cB\big(\cH^2\big)$\big) and
$E_{\alpha} \in \cB\big(\cH^2\big)$ are defined in \eqref{6.46} and \eqref{6.47}, and 
$M_{\alpha}(\, \cdot \,,x_0)$ is the standard $\cB\big(\cH^2\big)$-valued Weyl--Titchmarsh 
$2 \times 2$ block operator Weyl--Titchmarsh function \eqref{2.71}--\eqref{2.71a} with 
$\Omega_{\alpha}(\, \cdot \, x_0)$ the $\cB\big(\cH^2\big)$-valued measure in its 
Nevanlinna--Herglotz representation \eqref{2.71b}--\eqref{2.71d}. 

This result shows that the entire spectral information for $H$ is also contained in
the $\cB\big(\cH^2\big)$-valued measure $\Omega_{\alpha}^{Do}(\, \cdot \,,x_0)$ (again, 
including multiplicity properties of the spectrum of $H$). 

%%%%%%%%
\begin{remark} \lb{r1.1}
As the first equality in \eqref{1.21} shows, $M_{\alpha}^{Do} (z,x_0)$ recovers the traditional 
Weyl--Titchmarsh operator $M_{\alpha}(z,x_0)$ apart from the boundedly invertible 
$2 \times 2$ block operators $T_{\alpha}$. The latter is built from the half-line Weyl--Titchmarsh 
operators $m_{\pm,\alpha}(z,x_0)$ in a familiar, yet somewhat intriguing, manner 
(cf.\ \eqref{2.71}--\eqref{2.71a}), 
\begin{align}
\begin{split} 
& M_{\alpha}(z,x_0)      \\ 
& \quad = \left(\begin{smallmatrix} 
W(z)^{-1}  & \;\;\; 2^{-1} W(z)^{-1} [m_{-,\alpha}(z,x_0) + m_{-,\alpha}(z,x_0)]   \\ 
2^{-1} [m_{-,\alpha}(z,x_0) + m_{-,\alpha}(z,x_0)] W(z)^{-1} & \;\;\;
m_{\pm,\alpha}(z,x_0) W(z)^{-1} m_{\mp,\alpha}(z,x_0) \end{smallmatrix}\right),   
\lb{1.24} \\ 
& \hspace*{9.35cm} z \in \bbC \backslash \sigma(H),    
\end{split} 
\end{align} 
abbreviating $W(z) = [m_{-,\alpha}(z,x_0) - m_{+,\alpha}(z,x_0)]$, 
$z \in \bbC \backslash \sigma(H)$. 
In contrast to this construction, combining the Donoghue $m$-function $M_{H, \cN_i}^{Do}(\cdot)$ 
with the left/right half-line decomposition \eqref{1.10}, via equation \eqref{1.15}, directly leads to 
\eqref{1.17}--\eqref{1.20}, and hence to \eqref{1.21}, and thus to the $\cB\big(\cH^2\big)$-valued 
measure $\Omega_{\alpha}^{Do}(\, \cdot \,,x_0)$ in the Nevanlinna--Herglotz representation of  
$M_{\alpha}^{Do} (\, \cdot \,,x_0)$, encoding the entire spectral information of $H$ 
contained in it's family of spectral projections $E_H(\cdot)$. 

Of course, $\Omega_{\alpha}^{Do}(\, \cdot \,,x_0)$ is directly related to the 
$\cB\big(\cH^2\big)$-valued Weyl--Titchmarsh measure measure 
$\Omega_{\alpha}(\, \cdot \,,x_0)$ in the Nevanlinna--Herglotz representation of  
$M_{\alpha}(\, \cdot \,,x_0)$ via relation \eqref{1.23}, but our point is that the simple left/right 
half-line decomposition \eqref{1.10} combined with the Donoghue-type $m$ function \eqref{1.14} 
naturally leads to $\Omega_{\alpha}^{Do}(\, \cdot \,,x_0)$, without employing \eqref{1.24}. This 
offers interesting possibilities in the PDE context where $\bbR^n$, $n \in \bbN$, $n \geq 2$, can 
now be decomposed in various manners, for instance, into the interior and exterior of a given 
(bounded or unbounded) domain $D \subset \bbR^n$, a left/right (upper/lower) half-space, etc. 
In this context we should add that this paper concludes the first part of our program, the 
treatment of half-line and full-line Schr\"odinger operators with bounded operator-valued potentials. Part two will aim at certain classes of unbounded operator-valued potentials $V$, applicable to
multi-dimensional Schr\"odinger operators in $L^2(\bbR^n; d^n x)$, $n \in \bbN$,
$n \geq 2$, generated by differential expressions of the type $- \Delta + V(\cdot)$. In
fact, it was precisely the connection between multi-dimensional Schr\"odinger
operators and one-dimensional Schr\"odinger operators with unbounded operator-valued potentials which originally motivated our interest in this program. We will return to this circle of ideas  elsewhere.  \hfill $\diamond$
\end{remark}
%%%%%%%%

At this point we turn to the content of each section: Section \ref{s2} recalls our basic results in
\cite{GWZ13} on the initial value problem associated with Schr\"odinger operators with bounded
operator-valued potentials. We use this section to introduce some of the basic notation employed
subsequently and note that our conditions on $V(\cdot)$ (cf.\ Hypothesis \ref{h2.7}) are the most
general to date with respect to the local behavior of the potential $V(\cdot)$. Following our detailed treatment in \cite{GWZ13}, Section \ref{s3} introduces maximal and minimal operators associated with the differential expression $\tau = - (d^2/dx^2) I_{\cH} + V(\cdot)$ on the interval 
$(a,b) \subset \bbR$
(eventually aiming at the case of a half-line $(a,\infty)$), and assuming that the left endpoint
$a$ is regular for $\tau$ and that $\tau$ is in the limit-point case at the endpoint $b$ we discuss
the family of self-adjoint extensions $H_{\alpha}$ in $L^2((a,b); dx; \cH)$ corresponding to boundary
conditions of the type
\begin{equation}
\sin(\alpha)u'(a) + \cos(\alpha)u(a)=0,     \lb{1.26}
\end{equation}
indexed by the self-adjoint operator $\alpha = \alpha^* \in \cB(\cH)$.
In addition, we recall elements of Weyl--Titchmarsh theory, the introduction of the
operator-valued Weyl--Titchmarsh function $m_{\alpha}(\cdot) \in \cB(\cH)$ and the Green's
function $G_{\alpha}(z,\, \cdot \,, \,\cdot \,) \in \cB(\cH)$ of $H_{\alpha}$. In particular, we prove bounded invertibility of $\Im(m_{\alpha}(\cdot))$ in $\cB(\cH)$ in Theorem \ref{t3.3}. In Section \ref{s4} we recall the analogous results
for full-line Schr\"odinger operators $H$ in $L^2(\bbR; dx; \cH)$, employing a $2 \times 2$ block
operator representation of the associated Weyl--Titchmarsh $M_{\alpha} (\, \cdot \,,x_0)$-matrix 
and its $\cB\big(\cH^2\big)$-valued spectral measure $d\Omega_{\alpha}(\, \cdot \,,x_0)$, 
decomposing $\bbR$ into a left and right half-line with respect to the reference point 
$x_0 \in \bbR$, $(-\infty, x_0] \cup [x_0, \infty)$. Various 
basic facts on deficiency subspaces, abstract Donoghue-type $m$-functions and the bounded 
invertibility of their imaginary parts, and the notion of completely non-self-adjoint symmetric operators are provided in Section \ref{s5}. This section also discusses the possibility 
of a reduction of the spectral family $E_A(\cdot)$ of the self-adjoint operator $A$ 
in $\cH$ to the measure $\Sigma_A(\cdot) = P_{\cN} E_A(\cdot)P_{\cN}\big|_{\cN}$ in $\cN$ 
(with $P_{\cN}$ the orthogonal projection onto a closed linear subspace $\cN$ of $\cH$) to the 
effect that $A$ is unitarily equivalent to the operator of multiplication by the independent
variable $\la$ in the space $L^2(\bbR; d\Sigma_A(\la);\cN)$, yielding a diagonalization 
of $A$ (see Theorem \ref{t5.6}). Our final and principal Section \ref{s6}, establishes complete 
non-self-adjointness of the minimal operators $H_{\pm,\min}$ in $L^2((x_0, \pm \infty); dx;\cH)$ 
(cf.\ Theorem \ref{t6.2}), and analyzes in detail the half-line Donoghue-type $m$-functions 
$M_{H_{\pm,\alpha}, \cN_{\pm,i}}^{Do} (\, \cdot \,,x_0)$ in $\cN_{\pm,i}$. In addition, it 
introduces the derived quantities $m_{\pm,\alpha}^{Do}(\, \cdot \,,x_0)$ in $\cH$ and 
subsequently, turns to the full-line Donoghue-type operators $M_{H, \cN_i}^{Do} (\cdot)$ in 
$\cN_i$ and $M_{\alpha}^{Do} (\,\cdot\,,x_0)$ in $\cH^2$. It is then proved that the entire spectral information for $H_{\pm}$ and $H$ (including multiplicity issues) are encoded in 
$M_{H_{\pm,\alpha}, \cN_{\pm,i}}^{Do} (\,\cdot \,,x_0)$ (equivalently, in 
$m_{\pm,\alpha}^{Do}(\,\cdot\,,x_0)$) and in $M_{H, \cN_i}^{Do} (\cdot)$ (equivalently, 
in $M_{\alpha}^{Do} (\, \cdot \,,x_0)$), respectively. Appendix \ref{sA} collects basic
facts on operator-valued Nevanlinna--Herglotz functions. We introduced the background material in Sections \ref{s2}--\ref{s4} to make this paper reasonably self-contained.

Finally, we briefly comment on the notation used in this paper: Throughout, $\cH$
denotes a separable, complex Hilbert space with inner product and norm
denoted by $(\, \cdot \,, \, \cdot \,)_{\cH}$ (linear in the second argument) and
$\|\cdot \|_{\cH}$, respectively. The identity operator in $\cH$ is written as
$I_{\cH}$. We denote by
$\cB(\cH)$ (resp., $\cB_{\infty}(\cH)$) the Banach space of linear bounded (resp., compact)
operators in $\cH$. The domain, range, kernel (null space), resolvent set, and spectrum
of a linear operator will be denoted by $\dom(\cdot)$,
$\ran(\cdot)$, $\ker(\cdot)$, $\rho(\cdot)$, and $\sigma(\cdot)$, respectively. The closure
of a closable operator $S$ in $\cH$ is denoted by $\ol S$. 
By $\mathfrak{B}(\bbR)$ we denote the collection of Borel subsets of $\bbR$.

%%%%%%%%%%%%%%%%%%%%%%%%%%%%%%%%%%%%%%%%
%%%%%%%%%%%%%%%%%%%%%%%%%%%%%%%%%%%%%%%%
\section{Basics on the Initial Value For Schr\"odinger Operators With Operator-Valued Potentials} \label{s2}
%%%%%%%%%%%%%%%%%%%%%%%%%%%%%%%%%%%%%%%%
%%%%%%%%%%%%%%%%%%%%%%%%%%%%%%%%%%%%%%%%

In this section we recall the basic results on initial value problems for second-order differential equations of the form $-y''+Qy=f$ on an arbitrary open interval
$(a,b) \subseteq \bbR$ with a bounded operator-valued coefficient $Q$, that is, when $Q(x)$ is a bounded operator on a separable, complex Hilbert space $\cH$ for a.e.\ $x\in(a,b)$. We are concerned with two types of situations: in the first one $f(x)$ is an element of the Hilbert space $\cH$ for a.e.\ $x\in (a,b)$, and the solution sought is to take values in $\cH$. In the second situation, $f(x)$ is a bounded operator on $\cH$ for a.e.\ $x\in(a,b)$, as is the proposed solution $y$.

All results recalled in this section were proved in detail in \cite{GWZ13}.

We start with some necessary preliminaries:
Let $(a,b) \subseteq \bbR$ be a finite or infinite interval and $\cX$ a Banach space.
Unless explicitly stated otherwise (such as in the context of operator-valued measures in
Nevanlinna--Herglotz representations, cf.\ Appendix \ref{sA}), integration of $\cX$-valued functions on $(a,b)$ will
always be understood in the sense of Bochner (cf., e.g., \cite[p.\ 6--21]{ABHN01},
\cite[p.\ 44--50]{DU77}, \cite[p.\ 71--86]{HP85}, \cite[Ch.\ III]{Mi78}, \cite[Sect.\ V.5]{Yo80} for
details). In particular, if $p\ge 1$, the symbol $L^p((a,b);dx;\cX)$ denotes the set of equivalence classes of strongly measurable $\cX$-valued functions which differ at most on sets of Lebesgue measure zero, such that $\|f(\cdot)\|_{\cX}^p \in L^1((a,b);dx)$. The
corresponding norm in $L^p((a,b);dx;\cX)$ is given by
$\|f\|_{L^p((a,b);dx;\cX)} = \big(\int_{(a,b)} dx\, \|f(x)\|_{\cX}^p \big)^{1/p}$,
rendering $L^p((a,b);dx;\cX)$ a Banach space. If $\cH$ is a separable Hilbert space, then so is $L^2((a,b);dx;\cH)$ (see, e.g.,
\cite[Subsects.\ 4.3.1, 4.3.2]{BW83}, \cite[Sect.\ 7.1]{BS87}). One recalls
that by a result of Pettis \cite{Pe38}, if $\cX$ is separable, weak
measurability of $\cX$-valued functions implies their strong measurability.

Sobolev spaces $W^{n,p}((a,b); dx; \cX)$ for $n\in\bbN$ and $p\geq 1$ are defined as follows: $W^{1,p}((a,b);dx;\cX)$ is the set of all
$f\in L^p((a,b);dx;\cX)$ such that there exists a $g\in L^p((a,b);dx;\cX)$ and an
$x_0\in(a,b)$ such that
\begin{equation}
f(x)=f(x_0)+\int_{x_0}^x dx' \, g(x') \, \text{ for a.e.\ $x \in (a,b)$.}
\end{equation}
In this case $g$ is the strong derivative of $f$, $g=f'$. Similarly,
$W^{n,p}((a,b);dx;\cX)$ is the set of all $f\in L^p((a,b);dx;\cX)$ so that the first $n$ strong
derivatives of $f$ are in $L^p((a,b);dx;\cX)$. For simplicity of notation one also introduces
$W^{0,p}((a,b);dx;\cX)=L^p((a,b);dx;\cX)$. Finally, $W^{n,p}_{\rm loc}((a,b);dx;\cX)$ is
the set of $\cX$-valued functions defined on $(a,b)$ for which the restrictions to any
compact interval $[\alpha,\beta]\subset(a,b)$ are in $W^{n,p}((\alpha,\beta);dx;\cX)$.
In particular, this applies to the case $n=0$ and thus defines $L^p_{\rm loc}((a,b);dx;\cX)$.
If $a$ is finite we may allow $[\alpha,\beta]$ to be a subset of $[a,b)$ and denote the
resulting space by $W^{n,p}_{\rm loc}([a,b);dx;\cX)$ (and again this applies to the case
$n=0$).

Following a frequent practice (cf., e.g., the discussion in \cite[Sect.\ III.1.2]{Am95}), we
will call elements of $W^{1,1} ([c,d];dx;\cX)$, $[c,d] \subset (a,b)$ (resp.,
$W^{1,1}_{\rm loc}((a,b);dx;\cX)$), strongly absolutely continuous $\cX$-valued functions
on $[c,d]$ (resp., strongly locally absolutely continuous $\cX$-valued functions
on $(a,b)$), but caution the reader that unless $\cX$ possesses the Radon--Nikodym
(RN) property, this notion differs from the classical definition
of $\cX$-valued absolutely continuous functions (we refer the interested reader
to \cite[Sect.\ VII.6]{DU77} for an extensive list of conditions equivalent to $\cX$ having the
RN property). Here we just mention that reflexivity of $\cX$ implies the RN property.

In the special case where $\cX = \bbC$, we omit $\cX$ and just write
$L^p_{(\loc)}((a,b);dx)$, as usual.

We emphasize that a strongly continuous operator-valued function $F(x)$,
$x \in (a,b)$, always means continuity of $F(\cdot) h$ in $\cH$ for all $h \in\cH$ (i.e.,
pointwise continuity of $F(\cdot)$ in $\cH$). The same pointwise conventions will apply to 
the notions of strongly differentiable and strongly measurable operator-valued functions
 throughout this manuscript. In particular, and unless explicitly stated otherwise, for 
 operator-valued functions $Y$, the symbol $Y'$ will be understood in the strong sense; 
 similarly,  $y'$ will denote the strong derivative for vector-valued functions $y$.

%%%%%%%%%%%%%
\begin{definition} \lb{d2.2}
Let $(a,b)\subseteq\bbR$ be a finite or infinite interval and
$Q:(a,b)\to\cB(\cH)$ a weakly measurable operator-valued function with
$\|Q(\cdot)\|_{\cB(\cH)}\in L^1_\loc((a,b);dx)$, and suppose that
$f\in L^1_{\loc}((a,b);dx;\cH)$. Then the $\cH$-valued function
$y: (a,b)\to \cH$ is called a (strong) solution of
\begin{equation}
- y'' + Q y = f   \lb{2.15A}
\end{equation}
if $y \in W^{2,1}_\loc((a,b);dx;\cH)$ and \eqref{2.15A} holds a.e.\ on $(a,b)$.
\end{definition}
%%%%%%%%%%%%%

One verifies that $Q:(a,b)\to\cB(\cH)$ satisfies the conditions in
Definition \ref{d2.2} if and only if $Q^*$ does (a fact that will play a role later on, cf.\ the paragraph following \eqref{2.33A}).

%%%%%%%%%%%%%
\begin{theorem} \lb{t2.3}
Let $(a,b)\subseteq\bbR$ be a finite or infinite interval and
$V:(a,b)\to\cB(\cH)$ a weakly measurable operator-valued function with
$\|V(\cdot)\|_{\cB(\cH)}\in L^1_\loc((a,b);dx)$. Suppose that
$x_0\in(a,b)$, $z\in\bbC$, $h_0,h_1\in\cH$, and $f\in
L^1_{\loc}((a,b);dx;\cH)$. Then there is a unique $\cH$-valued
solution $y(z,\, \cdot \,,x_0)\in W^{2,1}_\loc((a,b);dx;\cH)$ of the initial value problem
\begin{equation}
\begin{cases}
- y'' + (V - z) y = f \, \text{ on } \, (a,b)\bs E,  \\
\, y(x_0) = h_0, \; y'(x_0) = h_1,
\end{cases}     \lb{2.1}
\end{equation}
where the exceptional set $E$ is of Lebesgue measure zero and depends only on the representatives chosen for $V$ and $f$ but is independent of $z$.

Moreover, the following properties hold:
\begin{enumerate}[$(i)$]
\item For fixed $x_0,x\in(a,b)$ and $z\in\bbC$, $y(z,x,x_0)$ depends jointly continuously on $h_0,h_1\in\cH$, and $f\in L^1_{\loc}((a,b);dx;\cH)$ in the sense that
\begin{align}
\begin{split}
& \big\|y\big(z,x,x_0;h_0,h_1,f\big) - y\big(z,x,x_0;\wti h_0,\wti h_1,\wti f\big)\big\|_{\cH}    \\
& \quad \leq C(z,V)
\big[\big\|h_0 - \wti h_0\big\|_{\cH} + \big\|h_1 - \wti h_1\big\|_{\cH}
+ \big\|f - \wti f\big\|_{L^1([x_0,x];dx;\cH)}\big],    \lb{2.1A}
\end{split}
\end{align}
where $C(z,V)>0$ is a constant, and the dependence of
$y$ on the initial data $h_0, h_1$ and the inhomogeneity $f$ is displayed
in \eqref{2.1A}.
\item For fixed $x_0\in(a,b)$ and $z\in\bbC$, $y(z,x,x_0)$ is strongly continuously differentiable with respect to $x$ on $(a,b)$.
\item For fixed $x_0\in(a,b)$ and $z\in\bbC$, $y'(z,x,x_0)$ is strongly differentiable with respect to $x$ on $(a,b)\bs E$.
\item For fixed $x_0,x \in (a,b)$, $y(z,x,x_0)$ and $y'(z,x,x_0)$
are entire with respect to $z$.
\end{enumerate}
\end{theorem}
%%%%%%%%%%%%%

For classical references on initial value problems we refer, for instance, to
\cite[Chs.\ III, VII]{DK74} and \cite[Ch.\ 10]{Di60}, but we emphasize again that our approach minimizes the smoothness hypotheses on $V$ and $f$.

%%%%%%%%%%%%%
\begin{definition} \lb{d2.4}
Let $(a,b)\subseteq\bbR$ be a finite or infinite interval and assume that
$F,\,Q:(a,b)\to\cB(\cH)$ are two weakly measurable operator-valued functions such
that $\|F(\cdot)\|_{\cB(\cH)},\,\|Q(\cdot)\|_{\cB(\cH)}\in L^1_\loc((a,b);dx)$. Then the
$\cB(\cH)$-valued function $Y:(a,b)\to\cB(\cH)$ is called a solution of
\begin{equation}
- Y'' + Q Y = F   \lb{2.26A}
\end{equation}
if $Y(\cdot)h\in W^{2,1}_\loc((a,b);dx;\cH)$ for every $h\in\cH$ and $-Y''h+QYh=Fh$ holds
a.e.\ on $(a,b)$.
\end{definition}
%%%%%%%%%%%%%

%%%%%%%%%%%%%
\begin{corollary} \lb{c2.5}
Let $(a,b)\subseteq\bbR$ be a finite or infinite interval, $x_0\in(a,b)$, $z\in\bbC$, $Y_0,\,Y_1\in\cB(\cH)$, and suppose $F,\,V:(a,b)\to\cB(\cH)$ are two weakly measurable operator-valued functions with
$\|V(\cdot)\|_{\cB(\cH)},\,\|F(\cdot)\|_{\cB(\cH)}\in L^1_\loc((a,b);dx)$. Then there is a
unique $\cB(\cH)$-valued solution $Y(z,\, \cdot \,,x_0):(a,b)\to\cB(\cH)$ of the initial value
problem
\begin{equation}
\begin{cases}
- Y'' + (V - z)Y = F \, \text{ on } \, (a,b)\bs E,  \\
\, Y(x_0) = Y_0, \; Y'(x_0) = Y_1.
\end{cases} \lb{2.3}
\end{equation}
where the exceptional set $E$ is of Lebesgue measure zero and depends only on the representatives chosen for $V$ and $F$ but is independent of $z$. Moreover, the following properties hold:
\begin{enumerate}[$(i)$]
\item For fixed $x_0 \in (a,b)$ and $z \in \bbC$, $Y(z,x,x_0)$ is continuously
differentiable with respect to $x$ on $(a,b)$ in the $\cB(\cH)$-norm.
\item For fixed $x_0 \in (a,b)$ and $z \in \bbC$, $Y'(z,x,x_0)$ is strongly differentiable with respect to $x$ on $(a,b)\bs E$.
\item For fixed $x_0, x \in (a,b)$, $Y(z,x,x_0)$ and $Y'(z,x,x_0)$ are entire in $z$ in
the $\cB(\cH)$-norm.
\end{enumerate}
\end{corollary}
%%%%%%%%%%%%%

Various versions of Theorem \ref{t2.3} and Corollary \ref{c2.5} exist in the literature under varying assumptions on $V$ and $f, F$ (cf.\ the discussion in
\cite{GWZ13} which uses the most general hypotheses to date).

%%%%%%%%%%%%
\begin{definition} \lb{d2.6}
Pick $c \in (a,b)$.
The endpoint $a$ (resp., $b$) of the interval $(a,b)$ is called {\it regular} for the operator-valued differential expression $- (d^2/dx^2) + Q(\cdot)$ if it is finite and if $Q$ is weakly measurable and $\|Q(\cdot)\|_{\cB(\cH)}\in  L^1([a,c];dx)$ (resp.,
$\|Q(\cdot)\|_{\cB(\cH)}\in  L^1([c,b];dx)$) for some $c\in (a,b)$. Similarly,
$- (d^2/dx^2) + Q(\cdot)$ is called {\it regular at $a$} (resp., {\it regular at $b$}) if
$a$ (resp., $b$) is a regular endpoint for $- (d^2/dx^2) + Q(\cdot)$.
\end{definition}
%%%%%%%%%%%%

We note that if $a$ (resp., $b$) is regular for $- (d^2/dx^2) + Q(x)$, one may allow for
$x_0$ to be equal to $a$ (resp., $b$) in the existence and uniqueness Theorem \ref{t2.3}.

If $f_1, f_2$ are strongly continuously differentiable $\cH$-valued functions, we define the Wronskian of $f_1$ and $f_2$ by
\begin{equation}
W_{*}(f_1,f_2)(x)=(f_1(x),f'_2(x))_\cH - (f'_1(x),f_2(x))_\cH,    \lb{2.31A}
\quad x \in (a,b).
\end{equation}
If $f_2$ is an $\cH$-valued solution of $-y''+Qy=0$ and $f_1$ is an $\cH$-valued
solution of $-y''+Q^*y=0$, their Wronskian $W_{*}(f_1,f_2)(x)$ is $x$-independent, that is,
\begin{equation}
\f{d}{dx} W_{*}(f_1,f_2)(x) = 0 \, \text{ for a.e.\ $x \in (a,b)$}   \lb{2.32A}
\end{equation}
(in fact, by \eqref{2.52A}, the right-hand side of \eqref{2.32A} actually
vanishes for all $x \in (a,b)$).

We decided to use the symbol $W_{*}(\, \cdot \,,\, \cdot \,)$ in \eqref{2.31A} to
indicate its conjugate linear behavior with respect to its first entry.

Similarly, if $F_1,F_2$ are strongly continuously differentiable $\cB(\cH)$-valued
functions, their Wronskian is defined by
\begin{equation}
W(F_1,F_2)(x) = F_1(x) F'_2(x) - F'_1(x) F_2(x), \quad x \in (a,b).    \lb{2.33A}
\end{equation}
Again, if $F_2$ is a $\cB(\cH)$-valued solution of  $-Y''+QY = 0$ and $F_1$ is a
$\cB(\cH)$-valued solution of $-Y'' + Y Q = 0$ (the latter is equivalent to
$- {(Y^{*})}^{\prime\prime} + Q^* Y^* = 0$ and hence can be handled in complete analogy
via Theorem \ref{t2.3} and Corollary \ref{c2.5}, replacing $Q$ by $Q^*$) their Wronskian will be $x$-independent,
\begin{equation}
\f{d}{dx} W(F_1,F_2)(x) = 0 \, \text{ for a.e.\ $x \in (a,b)$.}
\end{equation}

Our main interest lies in the case where $V(\cdot)=V(\cdot)^* \in \cB(\cH)$ is self-adjoint. Thus, we now introduce the following basic assumption:

%%%%%%%%%%
\begin{hypothesis} \lb{h2.7}
Let $(a,b)\subseteq\bbR$, suppose that $V:(a,b)\to\cB(\cH)$ is a weakly
measurable operator-valued function with $\|V(\cdot)\|_{\cB(\cH)}\in L^1_\loc((a,b);dx)$,
and assume that $V(x) = V(x)^*$ for a.e.\ $x \in (a,b)$.
\end{hypothesis}
%%%%%%%%%%

Moreover, for the remainder of this paper we assume
\begin{equation}
\alpha = \alpha^* \in \cB(\cH).      \lb{2.4A}
\end{equation}

Assuming Hypothesis \ref{h2.7} and \eqref{2.4A}, we introduce the standard fundamental systems of operator-valued solutions of $\tau y=zy$ as follows: Since $\alpha$ is a bounded self-adjoint operator, one may define the self-adjoint operators $A=\sin(\alpha)$ and $B=\cos(\alpha)$ via the spectral theorem. Given such an operator $\alpha$ and a point $x_0\in(a,b)$ or a regular endpoint for $\tau$, we now
define $\theta_\alpha(z,\, \cdot \,, x_0), \, \phi_\alpha(z,\, \cdot \,,x_0)$ as those $\cB(\cH)$-valued
solutions of $\tau Y=z Y$ (in the sense of Definition \ref{d2.4}) which satisfy the initial
conditions
\begin{equation}
\theta_\alpha(z,x_0,x_0)=\phi'_\alpha(z,x_0,x_0)=\cos(\alpha), \quad
-\phi_\alpha(z,x_0,x_0)=\theta'_\alpha(z,x_0,x_0)=\sin(\alpha).    \lb{2.5}
\end{equation}

By Corollary \ref{c2.5}\,$(iii)$, for any fixed $x, x_0\in(a,b)$, the functions
$\theta_{\alpha}(z,x,x_0)$, $\phi_{\alpha}(z,x,x_0)$,
$\theta_{\alpha}(\ol{z},x,x_0)^*$, and $\phi_{\alpha}(\ol{z},x,x_0)^*$, as well as their strong $x$-derivatives are entire with respect to $z$ in the $\cB(\cH)$-norm.

Since $\theta_{\alpha}(\bar z,\, \cdot \,,x_0)^*$ and
$\phi_{\alpha}(\bar z,\, \cdot \,,x_0)^*$ satisfy
the adjoint equation $-Y''+YV=z Y$ and the same initial conditions as
$\theta_\alpha$ and $\phi_\alpha$, respectively, one can show the following identities (cf.\ \cite{GWZ13}):
\begin{align}
\theta_{\alpha}' (\bar z,x,x_0)^*\theta_{\alpha} (z,x,x_0)-
\theta_{\alpha} (\bar z,x,x_0)^*\theta_{\alpha}' (z,x,x_0)&=0, \label{2.7f}
\\
\phi_{\alpha}' (\bar z,x,x_0)^*\phi_{\alpha} (z,x,x_0)-
\phi_{\alpha} (\bar z,x,x_0)^*\phi_{\alpha}' (z,x,x_0)&=0, \label{2.7g}
\\
\phi_{\alpha}' (\bar z,x,x_0)^*\theta_{\alpha} (z,x,x_0)-
\phi_{\alpha} (\bar z,x,x_0)^*\theta_{\alpha}' (z,x,x_0)&=I_{\cH}, \label{2.7h}
\\
\theta_{\alpha} (\bar z,x,x_0)^*\phi_{\alpha}' (z,x,x_0)
- \theta_{\alpha}' (\bar z,x,x_0)^*\phi_{\alpha} (z,x,x_0)&=I_{\cH}, \label{2.7i}
\end{align}
as well as,
\begin{align}
\phi_{\alpha} (z,x,x_0)\theta_{\alpha} (\bar z,x,x_0)^*-
\theta_{\alpha} (z,x,x_0)\phi_{\alpha} (\bar z,x,x_0)^*&=0, \label{2.7j}
\\
\phi_{\alpha}' (z,x,x_0)\theta_{\alpha}' (\bar z,x,x_0)^*-
\theta_{\alpha}' (z,x,x_0)\phi_{\alpha}' (\bar z,x,x_0)^*&=0, \label{2.7k}
\\
\phi_{\alpha}' (z,x,x_0)\theta_{\alpha} (\bar z,x,x_0)^*-
\theta_{\alpha}' (z,x,x_0)\phi_{\alpha} (\bar z,x,x_0)^*&=I_{\cH}, \label{2.7l}
\\
\theta_{\alpha} (z,x,x_0)\phi_{\alpha}' (\bar z,x,x_0)^*-
\phi_{\alpha} (z,x,x_0)\theta_{\alpha}' (\bar z,x,x_0)^*&=I_{\cH}. \label{2.7m}
\end{align}

Finally, we recall two versions of Green's formula (resp., Lagrange's identity).

%%%%%%%%%%%%
\begin{lemma} \label{l2.9}
Let $(a,b)\subseteq\bbR$ be a finite or infinite interval and $[x_1,x_2]\subset(a,b)$. \\
$(i)$ Assume that $f,g\in W^{2,1}_{\rm loc}((a,b);dx;\cH)$. Then
\begin{equation}
\int_{x_1}^{x_2} dx \, [((\tau f)(x),g(x))_\cH-(f(x),(\tau g)(x))_\cH]
= W_{*}(f,g)(x_2)-W_{*}(f,g)(x_1).     \lb{2.52A}
\end{equation}
$(ii)$ Assume that $F,\,G:(a,b)\to\cB(\cH)$ are absolutely continuous operator-valued functions such that $F',\,G'$ are again differentiable and that $F''$, $G''$ are weakly measurable. In addition, suppose that $\|F''\|_\cH,\, \|G''\|_\cH \in L^1_\loc((a,b);dx)$. Then
\begin{equation}
\int_{x_1}^{x_2} dx \, [(\tau F^*)(x)^*G(x) - F(x) (\tau G)(x)] = W(F,G)(x_2) - W(F,G)(x_1).
\lb{2.53A}
\end{equation}
\end{lemma}
%%%%%%%%%%%%

%%%%%%%%%%%%%%%%%%%%%%%%%%%%%%%%%%%%%%%%
%%%%%%%%%%%%%%%%%%%%%%%%%%%%%%%%%%%%%%%%
\section{Half-Line Weyl--Titchmarsh and Spectral Theory for Schr\"odinger Operators 
with Operator-Valued Potentials} \label{s3}
%%%%%%%%%%%%%%%%%%%%%%%%%%%%%%%%%%%%%%%%
%%%%%%%%%%%%%%%%%%%%%%%%%%%%%%%%%%%%%%%%

In this section we recall the basics of Weyl--Titchmarsh and spectral theory for self-adjoint half-line Schr\"odinger operators $H_{\alpha}$ in
$L^2((a,b); dx; \cH)$ associated with the operator-valued differential
expression $\tau =-(d^2/dx^2) I_{\cH} + V(\cdot)$, assuming regularity of the
left endpoint $a$ and the limit-point case at the right endpoint $b$ (see
Definition \ref{d3.6}). These results were proved in \cite{GWZ13} and
\cite{GWZ13b} and we refer to these sources for details and an extensive bibliography on this topic.

As before, $\cH$ denotes a separable Hilbert space and $(a,b)$ denotes
a finite or infinite interval. One recalls that $L^2((a,b);dx;\cH)$ is separable
(since $\cH$ is) and that
\begin{equation}
(f,g)_{L^2((a,b);dx;\cH)} =\int_a^b dx \, (f(x),g(x))_\cH, \quad f,g\in L^2((a,b);dx;\cH).
\end{equation}

Assuming Hypothesis \ref{h2.7} throughout this section, we discuss self-adjoint operators in $L^2((a,b);dx;\cH)$ associated with the operator-valued differential expression 
$\tau =-(d^2/dx^2) I_{\cH} + V(\cdot)$ as suitable restrictions of the {\it maximal} operator 
$\oT_{\max}$ in $L^2((a,b);dx;\cH)$ defined by
\begin{align}
& \oT_{\max} f = \tau f,   \no \\
& f\in \dom(\oT_{\max})=\big\{g\in L^2((a,b);dx;\cH) \,\big|\,
g\in W^{2,1}_{\rm loc}((a,b);dx;\cH);      \lb{3.2} \\
& \hspace*{6.6cm} \tau g\in L^2((a,b);dx;\cH)\big\}.     \no
\end{align}
We also introduce the operator $\dot \oT_{\min}$ in $L^2((a,b);dx;\cH)$
\begin{equation}
\dom(\dot \oT_{\min})=\{g\in\dom(\oT_{\max})\,|\,\supp (g) \,
\text{is compact in} \, (a,b)\},
\end{equation}
and the {\it minimal} operator $\oT_{\min}$ in $L^2((a,b);dx;\cH)$ associated with $\tau$,
\begin{equation}
\oT_{\min} = \ol{\dot \oT_{\min}}.    \lb{3.4}
\end{equation}

One obtains,
\begin{equation}
\oT_{\max} = (\dot \oT_{\min})^*, \quad
\oT_{\max}^* = \ol{\dot \oT_{\min}} = \oT_{\min}.    \lb{3.13a}
\end{equation}

Moreover, Green's formula holds, that is, if $u$ and $v$ are in
$\dom(\oT_{\max})$, then
\begin{equation}\label{3.17A}
(\oT_{\max}u,v)_{L^2((a,b);dx;\cH)}-(u,\oT_{\max}v)_{L^2((a,b);dx;\cH)}
= W_*(u,v)(b) - W_*(u,v)(a).
\end{equation}

%%%%%%%%%%%%%%%
\begin{definition} \lb{d3.6}
Assume Hypothesis \ref{h2.7}.
Then the endpoint $a$ (resp., $b$) is said to be of {\it limit-point-type for $\tau$} if
$W_*(u,v)(a)=0$ (resp., $W_*(u,v)(b)=0$) for all $u,v\in\dom(\oT_{\max})$.
\end{definition}
%%%%%%%%%%%%%%%

Next, we introduce the subspaces
\begin{equation}
\cD_{z}=\{u\in\dom(\oT_{\max}) \,|\, \oT_{\max}u=z u\}, \quad z \in \bbC.
\end{equation}
For $z\in\bbC\backslash\bbR$, $\cD_{z}$ represent the deficiency subspaces of
$\oT_{\min}$. Von Neumann's theory of extensions of symmetric operators implies that
\begin{equation} \label{3.20A}
\dom(\oT_{\max})=\dom(\oT_{\min}) \dotplus \cD_i \dotplus \cD_{-i}, 
\end{equation}
where $\dotplus$ indicates the direct (but not necessarily orthogonal direct) sum in 
the underlying Hilbert space $L^2((a,b);dx;\cH)$.

For the remainder of this section we now make the following asumptions:

%%%%%%%%%%%%%
\begin{hypothesis} \lb{h3.9}
In addition to Hypothesis \ref{h2.7} suppose that $a$ is a regular
endpoint for $\tau$ and $b$ is of limit-point-type for $\tau$.
\end{hypothesis}
%%%%%%%%%%%%%

Given Hypothesis \ref{h3.9}, it has been shown in \cite{GWZ13} that all self-adjoint restrictions, $H_{\alpha}$, of $H_{\max}$, equivalently, all self-adjoint extensions of $H_{\min}$, are parametrized by $\alpha = \alpha^* \in \cB(\cH)$,  with domains given by
\begin{equation}
\dom(H_{\alpha})=\{u\in\dom(\oT_{\max}) \,|\, \sin(\alpha)u'(a)+\cos(\alpha)u(a)=0\}.   \lb{3.9} 
\end{equation}

Next, we recall that (normalized) $\cB(\cH)$-valued and square integrable solutions of 
$\tau Y=zY$, denoted by $\psi_{\alpha}(z,\, \cdot \,,a)$,
$z\in\bbC\backslash\sigma(H_{\alpha})$, and traditionally called
{\it Weyl--Titchmarsh} solutions of $\tau Y = z Y$, and the $\cB(\cH)$-valued
Weyl--Titchmarsh functions $m_{\alpha}(z,a)$, have been constructed in \cite{GWZ13} 
to the effect that
\begin{equation}
\psi_{\alpha}(z,x,a)=\theta_{\alpha}(z,x,a)
+ \phi_{\alpha}(z,x,a) m_{\alpha}(z,a),
\quad z \in \bbC\backslash\sigma(H_{\alpha}), \; x \in [a,b).     \lb{3.58A}
\end{equation}
Then $\psi_{\alpha}(\, \cdot \,,x,a)$ is analytic in $z$ on
$\bbC\backslash\bbR$ for fixed $x \in [a,b)$, and
\begin{equation}
\int_a^b dx \, \|\psi_{\alpha}(z,x,a) h\|_{\cH}^2 < \infty, \quad
h \in \cH, \; z \in \bbC\backslash\sigma(H_{\alpha}),
\end{equation}
in particular, 
\begin{equation}
\psi_{\alpha}(z, \, \cdot \, ,a) h \in L^2((a,b); dx; \cH), \quad h \in \cH, \; 
z \in \bbC \backslash \sigma(H_{\alpha}), 
\end{equation}
and 
\begin{equation}
\ker(H_{\max} - z I_{L^2((a,b); dx; \cH)}) = \{\psi_{\alpha}(z, \, \cdot \, ,a) h \,|\, h \in \cH\}. 
\quad z \in \bbC \backslash \bbR. 
\end{equation}
In addition, $m_{\alpha}(z,a)$ is a $\cB(\cH)$-valued Nevanlinna--Herglotz function (cf.\ Definition \ref{dA.4}), and
\begin{equation}
m_{\alpha}(z,a)=m_{\alpha}(\ol z, a)^*, \quad
z \in \bbC\backslash\sigma(H_{\alpha}).    \lb{3.59A}
\end{equation}
Given $u \in \cD_{z}$, the operator $m_{0}(z,a)$ assigns Neumann boundary
data $u'(a)$ to the Dirichlet boundary data $u(a)$, that is, $m_{0}(z,a)$ is
the ($z$-dependent) Dirichlet-to-Neumann map.

With the help of Weyl--Titchmarsh solutions one can now describe the resolvent
of $H_{\alpha}$ as follows,
\begin{align}
\begin{split}
\big((H_{\alpha} - z I_{L^2((a,b);dx;\cH)})^{-1} u\big)(x)
= \int_a^b dx' \, G_{\alpha}(z,x,x')u(x'),& \\
u \in L^2((a,b);dx;\cH), \; z \in \rho(H_{\alpha}), \; x \in [a,b),&
\end{split}
\end{align}
with the $\cB(\cH)$-valued Green's function $G_{\alpha}(z,\,\cdot \,, \, \cdot \,)$ given by
\begin{equation} \label{3.63A}
G_{\alpha}(z,x,x') = \begin{cases}
\phi_{\alpha}(z,x,a) \psi_{\alpha}(\ol{z},x',a)^*, & a\leq x \leq x'<b, \\
\psi_{\alpha}(z,x,a) \phi_{\alpha}(\ol{z},x',a)^*, & a\leq x' \leq x<b,
\end{cases}  \quad z\in\bbC\backslash\bbR.
\end{equation}

Next, we replace the interval $(a,b)$ by the right half-line $(x_0,\infty)$ and indicate this
change with the additional subscript $+$ in $H_{+,\min}$, $H_{+,\max}$, $H_{+,\alpha}$,
$\psi_{+,\alpha}(z,\, \cdot \,,x_0)$, $m_{+,\alpha}(\, \cdot \,,x_0)$,
$d\rho_{+,\alpha}(\, \cdot \,,x_0)$, $G_{+,\alpha}(z,\, \cdot \, , \, \cdot \,) $, etc., to distinguish
these quantities from the analogous objects on the left half-line $(-\infty, x_0)$ (later indicated
with the subscript $-$), which are needed in our subsequent full-line Section \ref{s4}.

Our aim is to relate the family of spectral projections,
$\{E_{H_{+,\alpha}}(\lambda)\}_{\lambda\in\bbR}$, of the self-adjoint
operator $H_{+,\alpha}$ and the $\cB(\cH)$-valued spectral function
$\rho_{+,\alpha}(\lambda,x_0)$, $\lambda\in\bbR$, which generates the
operator-valued measure $d\rho_{+,\alpha}(\, \cdot \, , x_0)$ in the Nevanlinna--Herglotz representation
\eqref{2.25} of $m_{+,\alpha}(\, \cdot \, , x_0)$:
\begin{equation}
m_{+,\alpha}(z,x_0) = c_{+,\alpha}
+ \int_{\bbR}  d\rho_{+,\alpha}(\lambda,x_0) \Big[
\frac{1}{\lambda-z} -\frac{\lambda}{\lambda^2 + 1}\Big],  \quad
z\in\bbC\backslash \sigma(H_{+,\alpha}),   \lb{2.25}
\end{equation}
where
\begin{equation}
c_{+,\alpha} = \Re(m_{+,\alpha}(i,x_0)) \in \cB(\cH),    \lb{2.25a}
\end{equation}
and $d\rho_{+,\alpha}(\, \cdot \, , x_0)$ is a $\cB(\cH)$-valued measure satisfying
\begin{equation}
\int_{\bbR} d(e,\rho_{+,\alpha}(\lambda, x_0)e)_{\cH} \,
(\lambda^2 + 1)^{-1} < \infty, \quad e\in\cH.    \lb{2.26}
\end{equation}
In addition, the Stieltjes inversion formula for the
nonnegative $\cB(\cH)$-valued measure $d\rho_{+,\alpha}(\, \cdot \,,x_0)$ reads
\begin{equation}
\rho_{+,\alpha}((\lambda_1,\lambda_2],x_0)
=\f1\pi \lim_{\delta\downarrow 0}
\lim_{\varepsilon\downarrow 0} \int^{\lambda_2+\delta}_{\lambda_1+\delta}
d\lambda \, \Im(m_{+,\alpha}(\lambda +i\varepsilon,x_0)), \quad \lambda_1,
\lambda_2 \in\bbR, \; \lambda_1<\lambda_2 \lb{2.27}
\end{equation}
(cf.\ Appendix \ref{sA} for details on Nevanlinna--Herglotz functions). We also note that
$m_{+,\alpha}(\, \cdot\,,x_0)$ and $m_{+,\beta}(\,\cdot\,,x_0)$ are related by the
following linear fractional transformation,
\begin{equation}\label{3.67A}
m_{+,\beta}(\,\cdot\,,x_0) = (C+Dm_{+,\alpha}(\,\cdot\,,x_0))(A+Bm_{+\alpha}(\,\cdot\,,x_0))^{-1},
\end{equation}
where
\begin{equation}
\begin{pmatrix}A&B\\ C&D\end{pmatrix}
= \begin{pmatrix}\cos(\beta) & \sin(\beta) \\ -\sin(\beta) & \cos(\beta) \end{pmatrix}
\begin{pmatrix}\cos(\alpha) & -\sin(\alpha) \\ \sin(\alpha) & \cos(\alpha) \end{pmatrix}.
\end{equation}

An important consequence of \eqref{3.67A} and the fact that the $m$-functions take values in $\cB(\cH)$ is the following invertibility result.

%%%%%%%%
\begin{theorem}\lb{t3.3}
Assume Hypothesis \ref{h3.9}, then $[\Im(m_{+,\al}(z,x_0))]^{-1} \in \cB(\cH)$ for all 
$z\in\bbC\bs\bbR$ and $\al=\al^*\in\cB(\cH)$.
\end{theorem}
%%%%%%%%
\begin{proof}
Let $z\in\bbC\bs\bbR$ be fixed. We first show that $[\Im(m_{+,0}(z,x_0))]^{-1} \in \cB(\cH)$. By \eqref{3.67A},
\begin{align}\lb{3.21}
m_{+,\beta}(z,x_0) = [\cos(\beta)m_{+,0}(z,x_0)-\sin(\beta)][\sin(\beta)m_{+,0}(z,x_0)+\cos(\beta)]^{-1},
\end{align}
hence using $\sin^2(\beta)+\cos^2(\beta)=I_\cH$ and commutativity of $\sin(\beta)$ and 
$\cos(\beta)$, one gets
\begin{align}
\cos(\beta)-\sin(\beta)m_{+,\beta}(z,x_0) = [\sin(\beta)m_{+,0}(z,x_0)+\cos(\beta)]^{-1}.
\end{align}
Taking $\beta = \beta(z) = {\rm arccot}(-\Re(m_{+,0}(z,x_0))) \in\cB(\cH)$ yields
\begin{align}
\cos(\beta)-\sin(\beta)m_{+,\beta}(z,x_0) = [\sin(\beta)i\Im(m_{+,0}(z,x_0))]^{-1}, 
\end{align}
and since the left-hand side is in $\cB(\cH)$, also $[\Im(m_{+,0}(z,x_0))]^{-1} \in \cB(\cH)$.

Next, we show that for any $\al=\al^*\in\cB(\cH)$, $[\Im(m_{+,\al}(z,x_0))]^{-1} \in \cB(\cH)$. Replacing $\beta$ by $\al$ in \eqref{3.21} and noting that both $\sin(\al)$ and $\cos(\al)$ are self-adjoint, one obtains
\begin{align}
m_{+,\al}(z,x_0) &= [\cos(\al)m_{+,0}(z,x_0)-\sin(\al)][\sin(\al)m_{+,0}(z,x_0)+\cos(\al)]^{-1},  \no \\
m_{+,\al}(z,x_0)^* &= [m_{+,0}(z,x_0)^*\sin(\al)+\cos(\al)]^{-1}[m_{+,0}(z,x_0)^*\cos(\al)-\sin(\al)],
\end{align}
and consequently
\begin{align}
2i\Im(m_{+,\al}(z,x_0)) &= m_{+,\al}(z,x_0) - m_{+,\al}(z,x_0)^*   \no \\
& = [m_{+,0}(z,x_0)^*\sin(\al)+\cos(\al)]^{-1} [2i\Im(m_{+,0}(z,x_0))]   \no \\
& \quad \times [\sin(\al)m_{+,0}(z,x_0) + \cos(\al)]^{-1}.  
\end{align}
Since $[\Im(m_{+,0}(z,x_0))]^{-1} \in \cB(\cH)$, it follows that 
$[\Im(m_{+,\al}(z,x_0))]^{-1} \in \cB(\cH)$.
\end{proof}
%%%%%%%%%

In the following, $C_0^\infty((c,d); \cH)$, $-\infty \leq c<d\leq \infty$,
denotes the usual space of infinitely differentiable $\cH$-valued functions of
compact support contained in $(c,d)$.

%%%%%%%%%%%%%%%%%%%%
\begin{theorem} \lb{t2.5}
Assume Hypothesis \ref{h3.9} and let  $f,g \in C^\infty_0((x_0,\infty); \cH)$,
$F\in C(\bbR)$, and $\lambda_1, \lambda_2 \in\bbR$,
$\lambda_1<\lambda_2$. Then,
\begin{align}
\begin{split}
& \big(f,F(H_{+,\alpha})E_{H_{+,\alpha}}((\lambda_1,\lambda_2])g \big)_{L^2((x_0,\infty);dx;\cH)}
\\
& \quad =  \big(\hatt f_{+,\alpha},M_FM_{\chi_{(\lambda_1,\lambda_2]}} \hatt
g_{+,\alpha}\big)_{L^2(\bbR;d\rho_{+,\alpha}(\, \cdot \,,x_0);\cH)},    \lb{2.28}
\end{split}
\end{align}
where we introduced the notation
\begin{equation}
\hatt u_{+,\alpha}(\lambda)=\int_{x_0}^\infty dx \,
\phi_\alpha(\lambda,x,x_0)^* u(x), \quad \lambda \in\bbR, \;
u \in C^\infty_0((x_0,\infty); \cH), \lb{2.29}
\end{equation}
and $M_G$ denotes the maximally defined operator of multiplication by the
function $G \in C(\bbR)$ in the Hilbert space $L^2(\bbR;d\rho_{+,\alpha};\cH)$,
\begin{align}
& \big(M_G\hatt u\big)(\lambda)=G(\lambda)\hatt u(\lambda)
\, \text{ for $\rho_{+,\alpha}$-a.e.\ $\lambda\in\bbR$},    \lb{2.30} \\
& \, \hatt u \in\dom(M_G)=\big\{\hatt v \in L^2(\bbR;d\rho_{+,\alpha}(\,\cdot\,,x_0);\cH) \,\big|\,
G\hatt v \in L^2(\bbR;d\rho_{+,\alpha}(\,\cdot\,,x_0);\cH)\big\}.   \no
\end{align}
Here $\rho_{+,\alpha}(\, \cdot \, , x_0)$ generates the operator-valued measure in the
Nevanlinna--Herglotz representation of the $\cB(\cH)$-valued Weyl--Titchmarsh function
$m_{+,\alpha}(\, \cdot \, , x_0) \in \cB(\cH)$ $($cf.\ \eqref{2.25}$)$.
\end{theorem}
%%%%%%%%%%%%%%%%%%%

For a discussion of the model Hilbert space $L^2(\bbR;d\Sigma;\cK)$ for
operator-valued measures $\Sigma$ we refer to \cite{GKMT01},
\cite{GWZ13a} and \cite[App.~B]{GWZ13b}.

In the context of operator-valued potential coefficients of half-line
Schr\"odinger operators we also refer to M.\ L.\ Gorbachuk \cite{Go68}, Sait{\= o} \cite{Sa71}, 
and Trooshin \cite{Tr00}.

The proof of Theorem \ref{t2.5} in \cite{GWZ13b} relies on a version of
Stone's formula in the weak sense (cf., e.g., \cite[p.\ 1203]{DS88}):

%%%%%%%%%%%%%%%%%%%%%
\begin{lemma} \lb{l2.4a}
Let $T$ be a self-adjoint operator in a complex separable Hilbert space
$\cH$ $($with scalar product denoted by $(\, \cdot \,,\, \cdot \,)_\cH$, linear in the
second factor$)$ and denote by $\{E_T(\lambda)\}_{\lambda\in\bbR}$ the
family of self-adjoint right-continuous spectral projections associated
with $T$, that is, $E_T(\lambda)=\chi_{(-\infty,\lambda]}(T)$,
$\lambda\in\bbR$. Moreover, let $f,g \in\cH$, $\lambda_1,\lambda_2\in\bbR$,
$\lambda_1<\lambda_2$, and $F\in C(\bbR)$. Then,
\begin{align}
&(f,F(T)E_{T}((\lambda_1,\lambda_2])g)_{\cH} \no \\
& \quad = \lim_{\delta\downarrow 0}\lim_{\varepsilon\downarrow 0}
\frac{1}{2\pi i}
\int_{\lambda_1+\delta}^{\lambda_2+\delta} d\lambda \, F(\lambda)
\big[\big(f,(T-(\lambda+i\varepsilon) I_{\cH})^{-1}g\big)_{\cH}  \no \\
& \hspace*{4.9cm} - \big(f,(T-(\lambda-i\varepsilon)I_{\cH})^{-1}
g\big)_{\cH}\big]. \lb{2.26a}
\end{align}
\end{lemma}
%%%%%%%%%%%%%%%%%%%%%

One can remove the compact support restrictions on $f$ and $g$
in Theorem \ref{t2.5} in the usual way by introducing the map
\begin{equation}
\widetilde U_{+,\alpha} : \begin{cases} C_0^\infty((x_0,\infty); \cH)\to
L^2(\bbR;d\rho_{+,\alpha}(\,\cdot\,,x_0);\cH) \\[1mm]
u \mapsto \hatt u_{+,\alpha}(\cdot)=
\int_{x_0}^\infty dx\, \phi_\alpha(\, \cdot \,,x,x_0)^* u(x). \end{cases} \lb{2.39}
\end{equation}
Taking $f=g$, $F=1$, $\lambda_1\downarrow -\infty$, and
$\lambda_2\uparrow \infty$ in \eqref{2.28} then shows that
$\widetilde U_{+,\alpha}$ is a densely defined isometry in
$L^2((x_0,\infty);dx;\cH)$, which extends by continuity to an isometry on
$L^2((x_0,\infty);dx;\cH)$. The latter is denoted by $U_{+,\alpha}$ and given by
\begin{equation}
U_{+,\alpha} : \begin{cases}L^2((x_0,\infty);dx;\cH)\to
L^2(\bbR;d\rho_{+,\alpha}(\,\cdot\,,x_0);\cH)    \\[1mm]
u \mapsto \hatt u_{+,\alpha}(\cdot)=
\slim_{b\uparrow\infty}\int_{x_0}^b dx\, \phi_\alpha(\, \cdot \,,x,x_0)^* u(x),
\end{cases}  \lb{2.40}
\end{equation}
where $\slim$ refers to the $L^2(\bbR;d\rho_{+,\alpha}(\,\cdot\,,x_0);\cH)$-limit.
In addition, one can show
that the map $U_{+,\alpha}$ in \eqref{2.40} is onto and hence that
$U_{+,\alpha}$ is unitary (i.e., $U_{+,\alpha}$ and
$U_{+,\alpha}^{-1}$ are isometric isomorphisms between the Hilbert spaces
$L^2((x_0,\infty);dx;\cH)$ and $L^2(\bbR;d\rho_{+,\alpha}(\,\cdot\,,x_0);\cH)$) with
\begin{equation}
U_{+,\alpha}^{-1} : \begin{cases} L^2(\bbR;d\rho_{+,\alpha};\cH) \to
L^2((x_0,\infty);dx;\cH)  \\[1mm]
\hatt u \mapsto \slim_{\mu_1\downarrow -\infty, \mu_2\uparrow\infty}
\int_{\mu_1}^{\mu_2}
\phi_\alpha(\lambda,\, \cdot \,,x_0) \, d\rho_{+,\alpha}(\lambda,x_0)\, \hatt u(\lambda).
\end{cases} \lb{2.45}
\end{equation}
Here $\slim$ refers to the $L^2((x_0,\infty); dx; \cH)$-limit.

We recall that the essential range of $F$ with respect to a scalar measure
$\mu$ is defined by
\begin{equation}
\essran_{\mu}(F)=\{z\in\bbC\,|\, \text{for all
$\varepsilon>0$,} \, \mu(\{\lambda\in\bbR \,|\,
|F(\lambda)-z|<\varepsilon\})>0\},  \lb{2.46c}
\end{equation}
and that $\essran_{\rho_{+,\alpha}}(F)$ for $F\in C(\bbR)$ is then defined to be
$\essran_{\nu_{+,\alpha}}(F)$ for any control measure $d\nu_{+,\alpha}$
of the operator-valued measure $d\rho_{+,\alpha}$. Given a complete orthonormal system
$\{e_n\}_{n \in \cI}$ in $\cH$ ($\cI \subseteq \bbN$ an appropriate index set), a convenient
control measure for $d\rho_{+,\alpha}$ is given by
\begin{equation}
\mu_{+,\alpha}(B)=\sum_{n\in\cI}2^{-n}(e_n, \rho_{+,\alpha}(B,x_0)e_n)_\cH, \quad
B\in\mathfrak{B}(\bbR).        \lb{2.46d}
\end{equation}

These considerations lead to a variant of the
spectral theorem for $H_{+,\alpha}$:

%%%%%%%%%%%%%%%%%%%%
\begin{theorem} \lb{t2.6}
Assume Hypothesis \ref{h3.9} and suppose $F\in C(\bbR)$. Then,
\begin{equation}
U_{+,\alpha} F(H_{+,\alpha})U_{+,\alpha}^{-1} = M_F I_{\cH}    \lb{2.46}
\end{equation}
in $L^2(\bbR;d\rho_{+,\alpha}(\,\cdot\,,x_0);\cH)$ $($cf.\ \eqref{2.30}$)$. Moreover,
\begin{align}
& \sigma(F(H_{+,\alpha}))= \essran_{\rho_{+,\alpha}}(F), \lb{2.46a} \\
& \sigma(H_{+,\alpha})=\supp(d\rho_{+,\alpha}(\,\cdot\,,x_0)),  \lb{2.46b}
\end{align}
and the multiplicity of the spectrum of $H_{+,\alpha}$ is at most equal to $\dim (\cH)$.
\end{theorem}
%%%%%%%%%%%%%%%%%%%%

%%%%%%%%%%%%%%%%%%%%%%%%%%%%%%%%%%%%%%%%
%%%%%%%%%%%%%%%%%%%%%%%%%%%%%%%%%%%%%%%%
\section{Weyl--Titchmarsh and Spectral Theory of Schr\"odinger Operators
with Operator-Valued Potentials on the Real Line} \lb{s4}
%%%%%%%%%%%%%%%%%%%%%%%%%%%%%%%%%%%%%%%%
%%%%%%%%%%%%%%%%%%%%%%%%%%%%%%%%%%%%%%%%

In this section we briefly recall the basic spectral theory for full-line Schr\"odinger operators $H$ in $L^2(\bbR; dx; \cH)$, employing a $2 \times 2$ block
operator representation of the associated Weyl--Titchmarsh matrix and its
$\cB\big(\cH^2\big)$-valued spectral measure, decomposing $\bbR$ into a left and right half-line with reference point $x_0 \in \bbR$,
$(-\infty, x_0] \cup [x_0, \infty)$.

We make the following basic assumption throughout this section.

%%%%%%%%%%%%%%%%%%%%%%%%%%%%%%%%%%%%%%%%
\begin{hypothesis} \lb{h2.8}
$(i)$ Assume that
\begin{equation}
V\in L^1_{\loc} (\bbR;dx;\cH), \quad V(x)=V(x)^* \, \text{ for a.e. } x\in\bbR
\lb{2.51}
\end{equation}
$(ii)$ Introducing the differential expression $\tau$ given by
\begin{equation}
\tau=-\f{d^2}{dx^2} I_{\cH} + V(x), \quad x\in\bbR, \lb{2.52}
\end{equation}
we assume $\tau$ to be in the limit-point case at $+\infty$ and at $-\infty$.
\end{hypothesis}
%%%%%%%%%%%%%%%%%%%%%%%%%%%%%%%%%%%%%%%%

Associated with the differential expression $\tau$ one introduces the self-adjoint Schr\"odinger operator $H$ in $L^2(\bbR;dx;\cH)$ by
\begin{align}
&Hf=\tau f,   \lb{2.53}
\\ \no
&f\in \dom(H)= \big\{g\in L^2(\bbR;dx;\cH) \, \big| \, g, g' \in
W^{2,1}_{\loc}(\bbR;dx;\cH); \, \tau g\in L^2(\bbR;dx;\cH)\big\}.
\end{align}

As in the half-line context we introduce the $\cB(\cH)$-valued fundamental
system of solutions $\phi_\alpha(z,\, \cdot \,,x_0)$ and
$\theta_\alpha(z,\, \cdot \,,x_0)$, $z\in\bbC$, of
\begin{equation}
(\tau \psi)(z,x) = z \psi(z,x), \quad x\in \bbR, \lb{2.54}
\end{equation}
with respect to a fixed reference point $x_0\in\bbR$, satisfying the
initial conditions at the point $x=x_0$,
\begin{align}
\begin{split}
\phi_\alpha(z,x_0,x_0)&=-\theta'_\alpha(z,x_0,x_0)=-\sin(\alpha), \\
\phi'_\alpha(z,x_0,x_0)&=\theta_\alpha(z,x_0,x_0)=\cos(\alpha), \quad
\alpha=\alpha^*\in\cB(\cH). \lb{2.55}
\end{split}
\end{align}
Again we note that by Corollary \ref{c2.5}\,$(iii)$, for any fixed $x, x_0\in\bbR$, the functions $\theta_{\alpha}(z,x,x_0)$, $\phi_{\alpha}(z,x,x_0)$,
$\theta_{\alpha}(\ol{z},x,x_0)^*$, and $\phi_{\alpha}(\ol{z},x,x_0)^*$ as well as their strong $x$-derivatives are entire with respect to $z$ in the $\cB(\cH)$-norm. Moreover, by \eqref{2.7i},
\begin{equation}
W(\theta_\alpha(\ol{z},\, \cdot \,,x_0)^*,\phi_\alpha(z,\, \cdot \,,x_0))(x)=I_\cH, \quad
z\in\bbC.  \lb{2.56}
\end{equation}

Particularly important solutions of \eqref{2.54} are the
{\it Weyl--Titchmarsh solutions} $\psi_{\pm,\alpha}(z,\, \cdot \,,x_0)$,
$z\in\bbC\backslash\bbR$, uniquely characterized by
\begin{align}
\begin{split}
&\psi_{\pm,\alpha}(z,\, \cdot \,,x_0)h \in L^2((x_0,\pm\infty);dx;\cH), \quad h\in\cH,
\\
&\sin(\alpha)\psi'_{\pm,\alpha}(z,x_0,x_0)
+\cos(\alpha)\psi_{\pm,\alpha}(z,x_0,x_0)=I_\cH, \quad
z\in\bbC\backslash\sigma(H_{\pm,\alpha}). \lb{2.57}
\end{split}
\end{align}
The crucial condition in \eqref{2.57} is again the $L^2$-property which
uniquely determines $\psi_{\pm,\alpha}(z,\, \cdot \,,x_0)$ up to constant
multiples by the limit-point hypothesis of $\tau$ at $\pm\infty$. In
particular, for
$\alpha = \alpha^*, \beta = \beta^* \in \cB(\cH)$,
\begin{align}
\psi_{\pm,\alpha}(z,\, \cdot \,,x_0) = \psi_{\pm,\beta}(z,\, \cdot \,,x_0)C_\pm(z,\alpha,\beta,x_0)
\lb{2.58}
\end{align}
for some coefficients $C_\pm (z,\alpha,\beta,x_0)\in\cB(\cH)$. The normalization in \eqref{2.57} shows that
$\psi_{\pm,\alpha}(z,\, \cdot \,,x_0)$ are of the type
\begin{equation}
\psi_{\pm,\alpha}(z,x,x_0)=\theta_{\alpha}(z,x,x_0)
+ \phi_{\alpha}(z,x,x_0) m_{\pm,\alpha}(z,x_0),
\quad z\in\bbC\backslash\sigma(H_{\pm,\alpha}), \; x\in\bbR, \lb{2.59}
\end{equation}
for some coefficients $m_{\pm,\alpha}(z,x_0)\in\cB(\cH)$, the
{\it Weyl--Titchmarsh $m$-functions} associated with $\tau$, $\alpha$,
and $x_0$. In addition, we note that
(with $z, z_1, z_2  \in \bbC\backslash\sigma(H_{\pm,\alpha})$)
\begin{align}
&W(\psi_{\pm,\al}(\ol{z_1},x_0,x_0)^*,\psi_{\pm,\al}(z_2,x_0,x_0)) = m_{\pm,\alpha}(z_2,x_0)- m_{\pm,\alpha}(z_1,x_0),    \lb{2.59a} \\
& \f{d}{dx}W(\psi_{\pm,\al}(\ol{z_1},x,x_0)^*,\psi_{\pm,\al}(z_2,x,x_0)) = (z_1-z_2)\psi_{\pm,\al}(\ol{z_1},x,x_0)^*\psi_{\pm,\al}(z_2,x,x_0),   \lb{2.59b} \\
& (z_2-z_1)\int_{x_0}^{\pm\infty} dx\, \psi_{\pm,\alpha}(\ol{z_{1}},x,x_0)^* 
\psi_{\pm,\alpha}(z_{2},x,x_0) = m_{\pm,\alpha}(z_2,x_0)- m_{\pm,\alpha}(z_1,x_0),   \lb{2.60} \\
& m_{\pm,\alpha}(z,x_0) = m_{\pm,\alpha}(\ol z,x_0)^*,   \lb{2.61} \\
& \Im[m_{\pm,\alpha}(z,x_0)] = 
\Im(z)\int_{x_0}^{\pm\infty} dx\, \psi_{\pm,\alpha}(z,x,x_0)^* \psi_{\pm,\alpha}(z,x,x_0).    \lb{2.62}
\end{align}
In particular, $\pm m_{\pm,\alpha}(\, \cdot \,,x_0)$ are operator-valued Nevanlinna--Herglotz functions.

In the following we abbreviate the Wronskian of $\psi_{+,\al}(\ol{z},x,x_0)^*$
and $\psi_{-,\al}(z,x,x_0)$ by $W(z)$ (thus, $W(z) = m_{-,\al}(z,x_0) - m_{+,\al}(z,x_0)$, 
$z\in\bbC\backslash\sigma(H)$). The Green's function $G(z,x,x')$ of the Schr\"odinger 
operator $H$ then reads
\begin{align}
G(z,x,x') = \psi_{\mp,\alpha}(z,x,x_0) W(z)^{-1} \psi_{\pm,\alpha}(\ol{z},x',x_0)^*,
\quad x \lesseqgtr x', \; z\in\bbC\backslash\sigma(H).
\lb{2.63}
\end{align}
Thus,
\begin{align}
\begin{split}
((H-zI_{L^2(\bbR; dx;\cH)})^{-1}f)(x)
=\int_{\bbR} dx' \, G(z,x,x')f(x'), \quad z\in\bbC\backslash\sigma(H),&   \\
x\in\bbR, \; f\in L^2(\bbR;dx;\cH).&    \lb{2.65}
\end{split}
\end{align}

Next, we introduce the $2\times 2$ block operator-valued Weyl--Titchmarsh
$m$-function,
$M_\alpha(z,x_0)\in\cB\big(\cH^2\big)$,
\begin{align}
M_\alpha(z,x_0)&=\big(M_{\alpha,j,j'}(z,x_0)\big)_{j,j'=0,1}, \quad
z\in\bbC\backslash\sigma(H),     \lb{2.71}
\\
M_{\alpha,0,0}(z,x_0) &= W(z)^{-1},
\\
M_{\alpha,0,1}(z,x_0) &= 2^{-1} W(z)^{-1} \big[m_{-,\alpha}(z,x_0)+m_{+,\alpha}(z,x_0)\big],
\\
M_{\alpha,1,0}(z,x_0) &= 2^{-1} \big[m_{-,\alpha}(z,x_0)+m_{+,\alpha}(z,x_0)\big] W(z)^{-1},
\\
M_{\alpha,1,1}(z,x_0) &= m_{+,\alpha}(z,x_0) W(z)^{-1} m_{-,\alpha}(z,x_0)
\no
\\
&= m_{-,\alpha}(z,x_0) W(z)^{-1} m_{+,\alpha}(z,x_0).   \lb{2.71a}
\end{align}
$M_\alpha(z,x_0)$ is a $\cB\big(\cH^2\big)$-valued Nevanlinna--Herglotz function with representation
\begin{equation}
M_\alpha(z,x_0)=C_\alpha(x_0)+\int_{\bbR}
d\Omega_\alpha (\lambda,x_0)\bigg[\frac{1}{\lambda -z}-\frac{\lambda}
{\lambda^2 + 1}\bigg], \quad z\in\bbC\backslash\sigma(H),      \lb{2.71b}
\end{equation}
where
\begin{equation}
C_\alpha(x_0)=\Re(M_\alpha(i,x_0)) \in \cB\big(\cH^2\big),    \lb{2.71c}
\end{equation}
and $d\Omega_{\alpha}(\, \cdot \,,x_0)$ is a $\cB\big(\cH^2\big)$-valued measure
satisfying
\begin{equation}
\int_{\bbR} \big(e,d\Omega_{\alpha}(\lambda,x_0)e\big)_{\cH^2} \, (\lambda^2 + 1)^{-1} < \infty,
\quad e\in\cH^2.    \lb{2.71d}
\end{equation}
In addition, the Stieltjes inversion formula for the nonnegative $\cB \big(\cH^2\big)$-valued measure $d\Omega_\alpha(\, \cdot \,,x_0)$ reads
\begin{equation}
\Omega_\alpha((\lambda_1,\lambda_2],x_0)
=\f1\pi \lim_{\delta\downarrow 0}
\lim_{\varepsilon\downarrow 0} \int^{\lambda_2+\delta}_{\lambda_1+\delta}
d\lambda \, \Im(M_\alpha(\lambda +i\varepsilon,x_0)), \quad \lambda_1,
\lambda_2 \in\bbR, \; \lambda_1<\lambda_2. \lb{2.71e}
\end{equation}
In particular, $d\Omega_\alpha(\, \cdot \,,x_0)$ is a $2\times 2$ block operator-valued measure
with $\cB(\cH)$-valued entries $d\Omega_{\al,\ell,\ell'}(\, \cdot \,,x_0)$, $\ell,\ell'=0,1$.

\medskip

Relating the family of spectral projections,
$\{E_H(\lambda)\}_{\lambda\in\bbR}$, of the self-adjoint
operator $H$ and the $2\times 2$ operator-valued increasing spectral
function $\Omega_{\alpha}(\lambda,x_0)$, $\lambda\in\bbR$, which
generates the $\cB\big(\cH^2\big)$-valued measure $d\Omega_\alpha(\, \cdot \,,x_0)$ in the Nevanlinna--Herglotz representation \eqref{2.71b} of $M_\alpha(z,x_0)$, one obtains the following result:

%%%%%%%%%%%%%%%%%%%
\begin{theorem} \lb{t2.9}
Let $\alpha = \alpha^* \in \cB(\cH)$, $f,g \in C^\infty_0(\bbR;\cH)$,
$F\in C(\bbR)$, $x_0\in\bbR$, and $\lambda_1, \lambda_2 \in\bbR$,
$\lambda_1<\lambda_2$. Then,
\begin{align} \lb{2.73}
&\big(f,F(H)E_H((\lambda_1,\lambda_2])g\big)_{L^2(\bbR;dx;\cH)}
\no \\
&\quad =
\big(\hatt
f_{\alpha}(\, \cdot \,,x_0),M_FM_{\chi_{(\lambda_1,\lambda_2]}} \hatt
g_{\alpha}(\, \cdot \,,x_0)\big)_{L^2(\bbR;d\Omega_{\alpha}(\, \cdot \,,x_0);\cH^2)}
\end{align}
where we introduced the notation
\begin{align} \lb{2.74}
&\hatt u_{\alpha,0}(\lambda,x_0) = \int_\bbR dx \,
\theta_\alpha(\lambda,x,x_0)^* u(x),  \quad
\hatt u_{\alpha,1}(\lambda,x_0) = \int_\bbR dx \,
\phi_\alpha(\lambda,x,x_0)^* u(x),
\no
\\
&\hatt u_{\alpha}(\lambda,x_0) = \big(\,\hatt u_{\alpha,0}(\lambda,x_0),
\hatt u_{\alpha,1}(\lambda,x_0)\big)^\top,  \quad
\lambda \in\bbR, \;  u \in C^\infty_0(\bbR;\cH),
\end{align}
and $M_G$ denotes the maximally defined operator of multiplication
by the function $G \in C(\bbR)$ in the Hilbert space
$L^2\big(\bbR;d\Omega_{\alpha}(\, \cdot \,,x_0);\cH^2\big)$,
\begin{align}
\begin{split}
& \big(M_G\hatt u\big)(\lambda)=G(\lambda)\hatt u(\lambda)
=\big(G(\lambda) \hatt u_0(\lambda), G(\lambda) \hatt u_1(\lambda)\big)^\top
\, \text{ for \ $\Omega_{\alpha}(\, \cdot \,,x_0)$-a.e.\ $\lambda\in\bbR$}, \lb{2.75} \\
& \, \hatt u \in \dom(M_G)=\big\{\hatt v \in
L^2\big(\bbR;d\Omega_{\alpha}(\, \cdot \,,x_0);\cH^2\big) \,\big|\,
G\hatt v \in L^2(\bbR;d\Omega_{\alpha}\big(\, \cdot \,,x_0);\cH^2\big)\big\}.
\end{split}
\end{align}
\end{theorem}
%%%%%%%%%%%%%%%%%

As in the half-line case, one can remove the compact support restrictions
on $f$ and $g$ in the usual way by considering the map
\begin{align}
&\widetilde U_{\alpha}(x_0) : \begin{cases} C_0^\infty(\bbR)\to
L^2\big(\bbR;d\Omega_{\alpha}(\, \cdot \,,x_0);\cH^2\big)
\\[1mm]
u \mapsto \hatt u_{\alpha}(\, \cdot \,,x_0)
=\big(\,\hatt u_{\alpha,0}(\lambda,x_0),
\hatt u_{\alpha,1}(\lambda,x_0)\big)^\top, \end{cases} \lb{2.95}
\\
&  \hatt u_{\alpha,0}(\lambda,x_0)=\int_\bbR dx \,
\theta_\alpha(\lambda,x,x_0)^* u(x), \quad \hatt u_{\alpha,1}(\lambda,x_0)=\int_\bbR dx \, \phi_\alpha(\lambda,x,x_0)^* u(x).  \no
\end{align}
Taking $f=g$, $F=1$, $\lambda_1\downarrow -\infty$, and
$\lambda_2\uparrow \infty$ in \eqref{2.73} then shows that $\widetilde
U_{\alpha}(x_0)$ is a densely defined isometry in $L^2(\bbR;dx;\cH)$,
which extends by continuity to an isometry on $L^2(\bbR;dx;\cH)$. The latter
is denoted by $U_{\alpha}(x_0)$ and given by
\begin{align}
&U_{\alpha}(x_0) : \begin{cases} L^2 (\bbR;dx;\cH)\to
L^2\big(\bbR;d\Omega_{\alpha}(\, \cdot \,,x_0);\cH^2\big) \\[1mm]
u \mapsto \hatt u_{\alpha}(\, \cdot \,,x_0)
= \big(\,\hatt u_{\alpha,0}(\, \cdot \,,x_0),
\hatt u_{\alpha,1}(\, \cdot \,,x_0)\big)^\top, \end{cases} \lb{2.96} \\
& \hatt u_\alpha(\, \cdot \,,x_0)=\begin{pmatrix}
\hatt u_{\alpha,0}(\, \cdot \,,x_0) \\
\hatt u_{\alpha,1}(\, \cdot \,,x_0) \end{pmatrix}=
\slim_{a\downarrow -\infty, b \uparrow\infty} \begin{pmatrix}
\int_{a}^b dx \, \theta_\alpha(\,\cdot \,,x,x_0)^* u(x) \\
\int_{a}^b dx \, \phi_\alpha(\, \cdot \,,x,x_0)^* u(x) \end{pmatrix}, \no
\end{align}
where $\slim$ refers to the
$L^2\big(\bbR;d\Omega_{\alpha}(\, \cdot \,,x_0);\cH^2\big)$-limit.

In addition, one can show
that the map $U_{\alpha}(x_0)$ in \eqref{2.96} is onto and hence that
$U_{\alpha}(x_0)$ is unitary with
\begin{align}
&U_{\alpha}(x_0)^{-1} : \begin{cases}
L^2\big(\bbR;d\Omega_{\alpha}(\, \cdot \,,x_0);\cH^2\big) \to  L^2(\bbR;dx;\cH)  \\[1mm]
\hatt u \mapsto u_\alpha, \end{cases} \lb{2.101} \\
& u_\alpha(\cdot)= \slim_{\mu_1\downarrow -\infty, \mu_2\uparrow
\infty} \int_{\mu_1}^{\mu_2}(\theta_\alpha(\lambda,\, \cdot \,,x_0),
\phi_\alpha(\lambda,\, \cdot \,,x_0)) \, d\Omega_{\alpha}(\lambda,x_0)\,
\hatt u(\lambda). \no
\end{align}
Here $\slim$ refers to the $L^2(\bbR;dx;\cH)$-limit.

Again, these considerations lead to a variant of the spectral theorem for $H$:

%%%%%%%%%%%%%%%%%%%%%%%%%%%%%%%%%%%%
\begin{theorem} \lb{t2.10}
Let $F\in C(\bbR)$ and $x_0\in\bbR$. Then,
\begin{equation}
U_{\alpha}(x_0) F(H)U_{\alpha}(x_0)^{-1} = M_F \lb{2.102}
\end{equation}
in $L^2\big(\bbR;d\Omega_{\alpha}(\, \cdot \,,x_0);\cH^2\big)$ $($cf.\ \eqref{2.75}$)$.
Moreover,
\begin{align}
& \sigma(F(H))= \essran_{\Omega_{\alpha}}(F),  \lb{2.103a} \\
& \sigma(H)=\supp (d\Omega_\alpha(\, \cdot \,,x_0)),    \lb{2.103b}
\end{align}
and the multiplicity of the spectrum of $H$ is at most equal to $2\dim(\cH)$.
\end{theorem}
%%%%%%%%%%%%%%%%%%%%%%%%%%%%%%%%%%%%%

%%%%%%%%%%%%%%%%%%%%%%%%%%%%%%%%%%%%%%%%
%%%%%%%%%%%%%%%%%%%%%%%%%%%%%%%%%%%%%%%%
\section{Some Facts on Deficiency Subspaces and Abstract \\ Donoghue-type
$m$-Functions} \lb{s5}
%%%%%%%%%%%%%%%%%%%%%%%%%%%%%%%%%%%%%%%%
%%%%%%%%%%%%%%%%%%%%%%%%%%%%%%%%%%%%%%%%

Throughout this preparatory section we make the following assumptions:

%%%%%%%%
\begin{hypothesis} \lb{h5.1}
Let $\cK$ be a separable, complex Hilbert space, and $\dot A$ a densely defined, closed, 
symmetric operator in $\cK$, with equal deficiency indices $(k,k)$, $k\in\bbN\cup\{\infty\}$. 
\end{hypothesis}
%%%%%%%%

Self-adjoint extensions of $\dot A$ in $\cK$ will be denoted by $A$ $($or by $A_{\alpha}$, 
with $\alpha$ an appropriate operator parameter\,$)$.

Given Hypothesis \ref{h5.1}, we will 
study properties of deficiency spaces of $\dot A$, and introduce
operator-valued Donoghue-type  $m$-functions corresponding to $A$, closely following
the treatment in \cite{GKMT01}. These results will be applied to Schr\"odinger operators
in the following section.

In the special case $k=1$, detailed investigation of this type were undertaken by Donoghue \cite{Do65}. The case $k\in \bbN$ was discussed in depth in \cite{GT00} (we also refer to
\cite{HKS98} for another comprehensive treatment of this subject).~Here we treat the general situation $k\in \bbN\cup\{\infty\}$, utilizing results in \cite{GKMT01}, \cite{GMT98}.

The deficiency subspaces $\cN_{z_0}$ of $\dot A$, $z_0 \in \bbC \backslash \bbR$, are given by
\begin{equation}
 \cN_{z_0} = \ker \big((\dot A)\high{^*} - z_0 I_{\cK}\big), \quad
\dim \, (\cN_{z_0})=k,      \lb{5.1}
\end{equation}
and for any self-adjoint extension $A$ of $\dot A$ in $\cK$, one has (see also 
\cite[p.~80--81]{Kr71})
\begin{equation}
(A - z_0 I_{\cK})(A - z I_{\cK})^{-1} \cN_{z_0} = \cN_{z}, \quad z, z_0 \in \bbC\backslash\bbR.
\lb{5.2}
\end{equation}

We also note the following result on deficiency spaces.

%%%%%%%%
\begin{lemma} \lb{l5.2} Assume Hypothesis \ref{h5.1}. Suppose $z_0 \in \bbC \backslash \bbR$, 
$h \in \cK$, and that $A$ is a self-adjoint extension of $\dot A$. Assume that
\begin{equation}
\text{for all $z \in \bbC \backslash \bbR$, } \,
h \, \bot \, \big\{(A - z I_{\cK})^{-1} \ker \big((\dot A)\high{^*} - z_0 I_{\cK}\big)\big\}.    \lb{5.3a}
\end{equation}
Then,
\begin{equation}
\text{for all $z \in \bbC \backslash \bbR$, } \, h \, \bot \, \ker \big((\dot A)\high{^*} - z I_{\cK}\big).   \lb{5.4a}
\end{equation}
\end{lemma}
%%%%%%%%
\begin{proof}
Let $f_{z_0} \in \ker \big((\dot A)\high{^*} - z_0 I_{\cK}\big)$, then
$\slim_{z \to i \infty} (-z) (A - z I_{\cK})^{-1}f_{z_0} = f_{z_0}$ and hence $h \, \bot \, f_{z_0}$, that is,
$h \, \bot \, \ker \big((\dot A)\high{^*} - z_0 I_{\cK}\big)$. The latter fact together with \eqref{5.3a} imply
\eqref{5.4a} due to \eqref{5.2}.
\end{proof}
%%%%%%%%

Next, given a self-adjoint extension $A$ of $\dot A$ in $\cK$ and a closed, 
linear subspace $\cN$ of $\cK$, $\cN\subseteq \cK$,
the Donoghue-type $m$-operator $M_{A,\cN}^{Do} (z) \in\cB(\cN)$ associated with the pair
 $(A,\cN)$  is defined by
\begin{align}
\begin{split}
M_{A,\cN}^{Do}(z)&=P_\cN (zA + I_\cK)(A - z I_{\cK})^{-1}
P_\cN\big\vert_\cN      \\
&=zI_\cN+(z^2+1)P_\cN(A - z I_{\cK})^{-1}
P_\cN\big\vert_\cN\,, \quad  z\in \bbC\backslash \bbR,     \lb{5.3}
\end{split}
\end{align}
with $I_\cN$ the identity operator in $\cN$ and $P_\cN$ the orthogonal projection in $\cK$
onto $\cN$. In our principal Section \ref{s6}, we will exclusively focus on the particular case 
$\cN = \cN_i = \dim\big((\dot A)\high{^*} - i I_{\cK}\big)$.

We turn to the Nevanlinna--Herglotz property of $M_{A,\cN}^{Do}(\cdot)$ next:

%%%%%%%%
\begin{theorem} \lb{t5.3} 
Assume Hypothesis \ref{h5.1}. 
Let $A$ be a self-adjoint extension of  $\dot A$ with associated orthogonal family of spectral
projections $\{E_A(\lambda)\}_{\lambda\in \bbR}$, and $\cN$ a closed subspace of $\cK$.
Then the Donoghue-type $m$-operator $M_{A,\cN}^{Do}(z)$ is analytic for
$z\in \bbC\backslash\bbR$ and
\begin{align} 
& [\Im(z)]^{-1} \Im\big(M_{A,\cN}^{Do} (z)\big) \geq
2 \Big[\big(|z|^2 + 1\big) + \big[\big(|z|^2 -1\big)^2 + 4 (\Re(z))^2\big]^{1/2}\Big]^{-1} I_{\cN},    \no \\
& \hspace*{9.5cm}  z\in \bbC\backslash \bbR.    \lb{5.4} 
\end{align}
In particular,
\begin{equation}
\big[\Im\big(M_{A,\cN}^{Do} (z)\big)\big]^{-1} \in \cB(\cN), \quad
z\in \bbC\backslash \bbR,     \lb{5.4A}
\end{equation}
and $M_{A,\cN}^{Do}(\cdot)$ is a $\cB(\cN)$-valued Nevanlinna--Herglotz function that 
admits the following representation valid in the strong operator topology of $\cN$,
\begin{equation}
M_{A,\cN}^{Do}(z)=
\int_\bbR
d\Omega_{A,\cN}^{Do}(\lambda) \bigg[\f{1}{\lambda-z} -
\f{\lambda}{\lambda^2 + 1}\bigg], \quad z\in\bbC\backslash\bbR,    \lb{5.5}
\end{equation}
where $($see also \eqref{A.42A}--\eqref{A.42b}$)$
\begin{align}
&\Omega_{A,\cN}^{Do}(\lambda)=(\lambda^2 + 1) (P_\cN E_A(\lambda)P_\cN\big\vert_\cN),
\lb{5.6} \\
&\int_\bbR d\Omega_{A,\cN}^{Do}(\lambda) \, (1+\lambda^2)^{-1}=I_\cN,    \lb{5.7} \\
&\int_\bbR d(\xi,\Omega_{A,\cN}^{Do} (\lambda)\xi)_\cN=\infty \, \text{ for all } \,
 \xi\in \cN\backslash\{0\}.    \lb{5.8}
\end{align}
\end{theorem}
%%%%%%%%

We just note that inequality \eqref{5.4} follows from
\begin{align}
[\Im (z)]^{-1} \Im (M_{A,\cN}^{Do}(z))&=
P_\cN(I_\cK+A^2)^{1/2}\big((A-\Re (z) I_{\cK})^2+
(\Im (z))^2 I_{\cK}\big)^{-1} \no \\
&\quad \times (I_\cK+A^2)^{1/2}P_\cN\big\vert_\cN,
\quad \, z\in \bbC\backslash\bbR,
\lb{5.8a}
\end{align}
the spectral theorem applied to $(I_\cK+A^2)^{1/2}\big((A-\Re (z) I_{\cK})^2+
(\Im (z))^2 I_{\cK}\big)^{-1} (I_\cK+A^2)^{1/2}$, together with 
\begin{align}
& \inf_{\lambda \in \bbR} \bigg(\frac{\lambda^2 + 1}{(\lambda-\Re(z))^2 + (\Im(z))^2}\bigg) 
= \inf_{\lambda \in \bbR} \bigg(\bigg|\f{\lambda - i}{\lambda - z}\bigg|^2\bigg)   \no \\
& \quad = \f{2}{\big(|z|^2 + 1\big) + \Big[\big(|z|^2 - 1\big)^2 + 4 (\Re(z))^2\Big]^{1/2}}, \quad 
z \in \bbC\backslash\bbR.     \lb{5.8b}
\end{align}

Since 
\begin{align}
& \Big[\big(|z|^2 + 1\big) + \big[\big(|z|^2 - 1\big)^2 + 4 (\Re(z))^2\big]^{1/2}\Big] \Big/ 2     \no \\ 
& \quad \leq \Big[\big(|z|^2 + 1\big) + \big(|z|^2 - 1\big) + 2 |\Re(z)|\Big] \Big/ 2   \no \\
& \quad = \max (1,|z|^2)+|\Re(z)|, \quad z \in \bbC\backslash\bbR, 
\end{align}
the lower bound \eqref{5.4} improves the one for 
$[\Im(z)]^{-1} \Im\big(M_{A,\cN}^{Do} (z)\big)$ recorded in \cite{GKMT01} and \cite{GMT98} 
if $\Re(z) \neq 0$\footnote{We note that \cite{GKMT01} and \cite{GMT98} contain a 
typographical error in this context in the sense that $\Im(z)$ must be replaced by $[\Im(z)]^{-1}$ 
in (4.16) of \cite{GKMT01} and (40) of \cite{GMT98}.}. 

Operators of the type $M_{A,\cN}^{Do}(\cdot)$ and some of its variants have attracted 
considerable attention in the literature. The interested reader can find a variety of additional 
results, for instance, in \cite{AB09}, \cite{AP04}, \cite{BL07}, \cite{BM14}--\cite{BR15a}, 
\cite{BGW09}--\cite{Bu97}, \cite{DHMdS09}--\cite{DMT88}, \cite{GKMT01}--\cite{GT00}, 
\cite{HMM13}, \cite{KO77}, \cite{KO78}, \cite{LT77}, \cite{Ma92a}, \cite{Ma92b}, \cite{MN11}, \cite{MN12}, \cite{Ma04}, \cite{Mo09}, \cite{Pa13}, \cite{Po04}, \cite{Re98}, \cite{Ry07}, 
and the references therein. 
We also add that a model operator approach for the pair $(\dot A, A)$ on the basis of the
operator-valued measure $\Omega_{A, \cN_{i}}$ has been developed in detail in \cite{GKMT01}.

In addition, we mention the following well-known fact (cf., e.g., \cite[Lemma~4.5]{GKMT01}, 
\cite[p.~80--81]{Kr71}):

%%%%%%%%
\begin{lemma} \lb{l5.4} 
Assume Hypothesis \ref{h5.1}. Then $\cK$ decomposes into the direct orthogonal sum
\begin{align}
& \cK=\cK_0\oplus \cK_0^\bot, \quad
 \ker \big((\dot A)\high{^*} - z I_{\cK}\big) \subset \cK_0, \quad
 z\in \bbC\backslash\bbR,   \lb{5.9} \\ 
& \cK_0^\bot = \bigcap_{z \in \bbC \backslash \bbR} 
\ker \big((\dot A)\high{^*} - z I_{\cK}\big)^\bot = 
\bigcap_{z \in \bbC \backslash \bbR} \ran \big(\dot A - z I_{\cK}\big),     \lb{5.9a} 
\end{align}
where $\cK_0$ and $ \cK_0^\bot$ are invariant subspaces for
all self-adjoint extensions $A$ of $\dot A$ in $\cK$, that is,
\begin{equation}
(A - z I_{\cK})^{-1}\cK_0\subseteq \cK_0 , \quad
(A - z I_{\cK})^{-1}\cK_0^\bot\subseteq \cK_0^\bot,
\quad z\in \bbC\backslash\bbR.     \lb{5.10}
\end{equation}
In addition,
\begin{equation}
\cK_0=\ol{\linspan \{(A - z I_{\cH})^{-1}u_+  \, \vert\,
u_+\in \cN_i, \, z \in \bbC\backslash \bbR \}}.     \lb{5.10a} 
\end{equation}
Moreover, all self-adjoint extensions  $\dot A$ coincide on $\cK_0^\bot$, that is, if
$A_\alpha$ denotes an arbitrary self-adjoint extension of $\dot A$, then
 \begin{equation}
A_\alpha=A_{0,\alpha} \oplus A_0^\bot \, \text{ in } \, \cK=\cK_0\oplus\cK_0^\bot,    \lb{5.11}
\end{equation}
where
\begin{equation}
A_0^\bot \text{ is independent of the chosen } A_{\alpha},      \lb{5.12}
\end{equation}
and $A_{0,\alpha}$ $($resp., $A_0^\bot$$)$ is self-adjoint in $\cK_0$ $($resp., $\cK_0^\bot$$)$.
\end{lemma}
%%%%%%%%

In this context we note that a densely defined closed symmetric operator $\dot A$
with deficiency indices $(k,k)$, $k\in \bbN\cup\{\infty\}$ is called 
{\it completely non-self-adjoint} (equivalently, {\it simple} or {\it prime}) in $\cK$ if $\cK_0^\bot=\{0\}$ 
in the decomposition \eqref{5.9} (cf.\ \cite[p.~80--81]{Kr71}). 

%%%%%%%%
\begin{remark} \lb{r5.5} 
In addition to Hypothesis \ref{h5.1} assume that $\dot A$ is not completely 
non-self-adjoint in $\cK$. Then in addition to \eqref{5.9}, \eqref{5.11}, and \eqref{5.12} one obtains
\begin{equation}
\dot A={\dot A}_0\oplus A_0^\bot, \quad \cN_{i}=\cN_{0,i}\oplus \{ 0 \}     \lb{5.13}
\end{equation}
with respect to the decomposition $\cK=\cK_0 \oplus \cK_0^\bot$.
In particular, the part $A_0^\bot$ of $\dot A$ in $\cK_0^\bot$ is self-adjoint. Thus, if 
$A = A_0 \oplus A_0^\bot$ is a self-adjoint extension of $\dot A$ in $\cK$, then 
\begin{equation}
M_{A, \cN_i}^{Do}(z) = M_{A_0, \cN_{0,i}}^{Do}(z), \quad z\in \bbC\backslash \bbR.     \lb{5.14}
\end{equation}
This reduces the $A$-dependent spectral properties of the Donoghue-type operator 
$M_{A, \cN_i}^{Do}(\cdot)$ effectively to those of $A_0$. A different manner in which to express 
this fact would be to note that the subspace $\cK_0^\bot$ is ``not detectable'' by 
$M_{A, \cN_i}^{Do}(\cdot)$ (we refer to \cite{BNMW14}) for a systematic investigation of this 
circle of ideas, particularly, in the context of non-self-adjoint operators). \hfill $\diamond$
\end{remark}
%%%%%%%%

We are particularly interested in the question under which conditions on $\dot A$, the spectral information for $A$ contained in its family of spectral projections 
$\{E_A(\lambda)\}_{\lambda \in \bbR}$ is already encoded in the $\cB(\cN_i)$-valued measure 
$\Omega_{A,\cN_i}^{Do}(\cdot)$. In this connection we now mention the following result, denoting 
by $C_b(\bbR)$ the space of scalar-valued bounded continuous functions on $\bbR$:

%%%%%%%%
\begin{theorem} \lb{t5.6}
Let $A$ be a self-adjoint operator on a separable Hilbert space $\cK$ and
$\{E_A(\la)\}_{\la\in\bbR}$ the family of spectral projections associated with $A$.
Suppose that $\cN\subset \cK$ is a closed linear subspace such that
\begin{align} \lb{5.22}
\ol{\linspan \, \{g(A)v \,|\, g\in C_b(\bbR), \, v\in\cN\}} = \cK.
\end{align}
Let $P_\cN$ be the orthogonal projection in $\cK$ onto $\cN$. Then $A$ is
unitarily equivalent to the operator of multiplication by the independent
variable $\la$ in the space $L^2(\bbR;d\Sigma_A(\la);\cN)$. Here the
operator-valued measure $d\Sigma_A(\cdot)$ is given in terms of the Lebesgue--Stieltjes
measure defined by the nondecreasing uniformly bounded family 
$\Sigma_A(\cdot)=P_\cN E_A(\cdot)P_\cN\big\vert_\cN$.
\end{theorem}
%%%%%%%%
\begin{proof}
It suffices to construct a unitary transformation $U:\cK \to
L^2(\bbR;d\Sigma_A(\la);\cN)$ that
satisfies $U A u = \la U u$ for all $u\in\cK$.
First, define $U$ on the set of vectors $\cS=\{g(A)v \,|\, g\in
C_b(\bbR), \; v\in\cN\}\subset\cK$ by
\begin{align}
U[g(A)v]=g(\la)v, \quad g\in C_b(\bbR), \; v\in\cN,
\end{align}
and then extend $U$ by linearity to the span of these vectors, which
by assumption is a dense subset of $\cK$. Applying the above
definition to the function $\la g(\la)$ yields $U A u = \la U u$ for
all $u$ in $\cS$ and hence by linearity also for all $u$ in the dense
subset $\linspan(\cS)$. In addition, the following simple computation
utilizing the spectral theorem for the self-adjoint operator $A$ shows
that $U$ is an isometry on $\cS$ and hence by linearity also on
$\linspan(\cS)$,
\begin{align}
\big(f(A)u,g(A)v\big)_{\cK} &=
\big(u,f(A)^*g(A)v\big)_{\cK} = \big(u,P_\cN f(A)^*g(A)
P_\cN\big\vert_\cN v\big)_{\cN} \no\\
&= \int_{\bbR} \big(u,\ol{f(\la)}g(\la)d\Sigma_A(\la)v\big)_{\cN} \\
&= \big(f(\cdot)u,g(\cdot)v\big)_{L^2(\bbR;d\Sigma_A(\la);\cN)}, \quad
f,g\in C_b(\bbR),
\; u,v\in\cN.   \no
\end{align}
Thus, $U$ can be extended by continuity to the whole Hilbert space $\cK$. Since the
range of $U$ contains the set $\{g(\cdot)v \,|\, g\in C_b(\bbR), \, v\in\cN\}$ which is dense
in $L^2(\bbR;d\Sigma_A(\la);\cN)$ (cf.\ \cite[Appendix~B]{GWZ13b}), it follows that $U$ is a unitary
transformation.
\end{proof}
%%%%%%%%

%%%%%%%%
\begin{remark} \lb{r5.7}
Since $\{(\la-z)^{-1}\,|\, z\in\bbC\bs\bbR\}\subset C_b(\bbR)$, the condition \eqref{5.22}
in Theorem \ref{t5.6} can be replaced by the following stronger, and frequently encountered,
one,
\begin{align}
\ol{\linspan \, \{(A - z I_{\cK})^{-1}v \,|\, z\in \bbC\bs\bbR, \, v\in\cN\}} = \cK.    \lb{5.25} 
\end{align}
\hfill $\diamond$ 
\end{remark}
%%%%%%%%

Combining  Lemma \ref{l5.4}, Remark \ref{r5.5}, Theorem \ref{t5.6}, and Remark \ref{r5.7} 
then yields the following fact:

%%%%%%%%
\begin{corollary} \lb{c5.8}
Assume Hypothesis \ref{h5.1} and suppose that $A$ is a self-adjoint extension of $\dot A$. 
Let $M_{A, \cN_i}^{Do}(\cdot)$ be the 
Donoghue-type $m$-operator associated with the pair $(A, \cN_i)$, with 
$\cN_i = \ker\big((\dot A)\high{^*} - i I_{\cK}\big)$, and denote by $\Omega_{A,\cN_i}^{Do}(\cdot)$ 
the $\cB(\cN_i)$-valued measure in the Nevanlinna--Herglotz representation of 
$M_{A, \cN_i}^{Do}(\cdot)$ $($cf.\ \eqref{5.5}$)$. Then $A$ is unitarily equivalent to the 
operator of multiplication by the independent
variable $\la$ in the space $L^2(\bbR; (\lambda^2 + 1)^{-1}d\Omega_{A,\cN_i}^{Do}(\la);\cN_i)$, 
with $\Omega_{A,\cN_i}^{Do}(\lambda) = (\lambda^2 + 1) P_{\cN_i} E_A(\lambda) P_{\cN_i}\big\vert_{\cN_i}$, $\lambda \in \bbR$, if and only if $\dot A$ is completely non-self-adjoint in $\cK$.
\end{corollary} 
%%%%%%%%
\begin{proof}
If $\dot A$ is completely non-self-adjoint in $\cK$, then $\cK_0 = \cK$, $\cK_0^{\bot} = \{0\}$ 
in \eqref{5.9}, together with \eqref{5.10a}, and \eqref{5.25} with 
$\cN = \cN_i$ yields  $\Sigma_A (\lambda) 
= (\lambda^2 + 1) P_{\cN_i} E_A(\lambda) P_{\cN_i}\big\vert_{\cN_i} 
= \Omega_{A,\cN_i}^{Do}(\lambda)$, $\lambda \in \bbR$, in Theorem \ref{t5.6}. 
Conversely, if $\dot A$ is not completely non-self-adjoint in $\cK$, then the fact \eqref{5.14} 
shows that $\Omega_{A,\cN_i}^{Do}(\cdot)$ cannot describe the nontrivial self-adjoint 
operator $A_0^{\bot}$ in $\cK_0^{\bot} \supsetneq \{0\}$. 
\end{proof}
%%%%%%%%

In other words, $\dot A$ is completely non-self-adjoint in $\cK$, if and only if the entire 
spectral information on $A$ contained in its family of spectral 
projections $E_A(\cdot)$, is already encoded in the $\cB(\cN_i)$-valued measure 
$\Omega_{A,\cN_i}^{Do}(\cdot)$  (including multiplicity properties of the spectrum of $A$).

%%%%%%%%%%%%%%%%%%%%%%%%%%%%%%%%%%%%%%%%
%%%%%%%%%%%%%%%%%%%%%%%%%%%%%%%%%%%%%%%%
\section{Donoghue-type $m$-Functions for Schr\"odinger Operators with
Operator-Valued Potentials and Their Connections to Weyl--Titchmarsh
$m$-Functions} \lb{s6}
%%%%%%%%%%%%%%%%%%%%%%%%%%%%%%%%%%%%%%%%
%%%%%%%%%%%%%%%%%%%%%%%%%%%%%%%%%%%%%%%%

In our principal section we construct Donoghue-type $m$-functions for half-line and
full-line Schr\"odinger operators with operator-valued potentials and establish their
precise connection with the Weyl--Titchmarsh $m$-functions discussed in Sections
\ref{s3} and \ref{s4}.

To avoid overly lengthy expressions involving resolvent operators, we now simplify our
notation a bit and use the symbol $I$ to denote the identity operator in
$L^2((x_0, \pm \infty); dx;\cH)$ and $L^2(\bbR; dx; \cH)$.

The principal hypothesis for this section will be the following:

%%%%%%%%
\begin{hypothesis} \lb{h6.1} ${}$ \\
$(i)$ For half-line Schr\"odinger operators on $[x_0,\infty)$ we assume Hypothesis \ref{h2.7}
with $a=x_0$, $b = \infty$ and assume $\tau = - (d^2/dx^2) I_{\cH} + V(x)$ to be in the limit-point
case at $\infty$. \\
$(ii)$ For half-line Schr\"odinger operators on $(-\infty,x_0]$ we assume Hypothesis \ref{h2.7}
with $a=-\infty$, $b = x_0$ and assume $\tau = - (d^2/dx^2) I_{\cH} + V(x)$ to be in the limit-point
case at $-\infty$. \\
$(iii)$ For Schr\"odinger operators on $\bbR$ we assume Hypothesis \ref{h2.8}.
\end{hypothesis}
%%%%%%%%

\subsection{The half-line case:} 
We start with half-line Schr\"odinger operators $H_{\pm,\min}$ in $L^2((x_0,\pm\infty); dx; \cH)$
and note that for $\{e_j\}_{j \in \cJ}$ a given orthonormal basis in $\cH$ ($\cJ \subseteq \bbN$ an
appropriate index set), and $z \in \bbC \backslash \bbR$,
\begin{equation}
\{\psi_{\pm,\alpha}(z, \, \cdot \,,x_0) e_j\}_{j \in \cJ}      \lb{6.1}
\end{equation}
is a basis in the deficiency subspace $\cN_{\pm,z} = \ker\big(H_{\pm,\min}^* - z I\big)$.
In particular, given $f \in L^2((x_0,\pm\infty); dx; \cH)$, one has
\begin{equation}
f \bot \{\psi_{\pm,\alpha}(z, \, \cdot \,,x_0) e_j\}_{j \in \cJ},     \lb{6.2}
\end{equation}
if and only if
\begin{align}
\begin{split}
0 &= (\psi_{\pm,\alpha}(z, \, \cdot \,,x_0) e_j,f)_{L^2((x_0,\pm\infty); dx; \cH)}
= \pm \int_{x_0}^{\pm\infty} dx \, (\psi_{+,\alpha}(z,x,x_0) e_j,f(x))_{\cH}  \lb{6.3} \\
&= \pm \int_{x_0}^{\pm\infty} dx \, (e_j,\psi_{\pm,\alpha}(z,x,x_0)^* f(x))_{\cH}, \quad j \in \cJ,
\end{split}
\end{align}
and since $j \in \cJ$ is arbitrary,
\begin{align}
\begin{split}
& f \bot \{\psi_{\pm,\alpha}(z, \, \cdot \,,x_0) e_j\}_{j \in \cJ} \, \text{ if and only if } \\
& \quad \pm \int_{x_0}^{\pm\infty} dx \, (h,\psi_{\pm,\alpha}(z,x,x_0)^* f(x))_{\cH} = 0, \quad h \in \cH,     \lb{6.4}
\end{split}
\end{align}
a fact to be exploited below in \eqref{6.5}.

Next, we prove the following generating property of deficiency spaces of $H_{\pm,\min}$:

%%%%%%%
\begin{theorem} \lb{t6.2}
Assume Hypothesis \ref{h6.1}\,$(i)$, respectively, $(ii)$, and suppose that 
$f \in L^2((x_0,\pm\infty); dx; \cH)$ satisfies for all $z \in \bbC \backslash \bbR$, 
$f \bot \ker\big(H_{\pm,\min}^* - z I\big)$. Then $f = 0$. Equivalently, $H_{\pm,\min}$ are 
completely non-self-adjoint in $L^2((x_0,\pm\infty); dx; \cH)$. 
\end{theorem}
%%%%%%%
\begin{proof}
We focus on the right-half line $[x_0,\infty)$ and recall the $\cB(\cH)$-valued Green's 
function $G_{+,\alpha}(z,\, \cdot \, , \, \cdot \,)$ in
\eqref{3.63A} of a self-adjoint extension $H_{+,\alpha}$ of $H_{+,\min}$.

Choosing a test vector $\eta \in C_0^{\infty}((x_0,\infty); \cH)$, $\lambda_j \in \bbR$,
$j=1,2$, $\lambda_1 < \lambda_2$, one computes with the help of Stone's formula
(cf.\ Lemma \ref{l2.4a}),
\begin{align}
& (\eta, E_{H_{+,\alpha}}((\lambda_1,\lambda_2])) f)_{L^2((x_0,\infty); dx; \cH)}    \no \\
& \quad = \lim_{\delta \downarrow 0} \lim_{\varepsilon \downarrow 0} \f{1}{2 \pi i}
\int_{\lambda_1 + \delta}^{\lambda_2 + \delta} d\lambda \,
\big[(\eta, (H_{+,\alpha} - (\lambda + i \varepsilon)I)^{-1} f)_{L^2((x_0,\infty); dx; \cH)} \no \\
& \hspace*{4.2cm}
- (\eta, (H_{+,\alpha} - (\lambda - i \varepsilon)I)^{-1} f)_{L^2((x_0,\infty); dx; \cH)}\big] \no \\
& \quad = \lim_{\delta \downarrow 0} \lim_{\varepsilon \downarrow 0} \f{1}{2 \pi i}
\int_{\lambda_1 + \delta}^{\lambda_2 + \delta} d\lambda \int_{x_0}^{\infty} dx \no \\
& \qquad \times \bigg\{\bigg[\bigg(\eta(x), \psi_{+,\alpha}(\lambda + i \varepsilon,x,x_0)
\int_{x_0}^x dx' \, \phi_{\alpha}(\lambda - i \varepsilon, x', x_0)^* f(x')\bigg)_{\cH}    \no \\
& \hspace*{1.5cm}
+ \underbrace{\int_{x_0}^{\infty} dx' \, (\phi_{\alpha}(\lambda + i \varepsilon,x,x_0)^* \eta(x),
\psi_{+,\alpha}(\lambda - i \varepsilon,x',x_0)^* f(x'))_{\cH}}_{=0 \text{ by \eqref{6.4}}}  \no \\
& \hspace*{1.5cm}
- \bigg(\eta(x), \phi_{\alpha}(\lambda + i \varepsilon,x,x_0) \int_{x_0}^x dx' \,
\psi_{+,\alpha}(\lambda - i \varepsilon,x',x_0)^* f(x')\bigg)_{\cH}\bigg]    \no \\
& \hspace*{1.5cm}
- \bigg[\bigg(\eta(x), \psi_{+,\alpha}(\lambda - i \varepsilon,x,x_0)
\int_{x_0}^x dx' \, \phi_{\alpha}(\lambda + i \varepsilon, x', x_0)^* f(x')\bigg)_{\cH}    \no \\
& \hspace*{2cm}
+ \underbrace{\int_{x_0}^{\infty} dx' \, (\phi_{\alpha}(\lambda - i \varepsilon,x,x_0)^* \eta(x),
\psi_{+,\alpha}(\lambda + i \varepsilon,x',x_0)^* f(x'))_{\cH}}_{= 0 \text{ by \eqref{6.4}}}  \no \\
& \hspace*{2cm}
- \bigg(\eta(x), \phi_{\alpha}(\lambda - i \varepsilon,x,x_0) \int_{x_0}^x dx' \,
\psi_{+,\alpha}(\lambda + i \varepsilon,x',x_0)^* f(x')\bigg)_{\cH}\bigg]\bigg\}.   \lb{6.5} 
\end{align} 
Here we twice employed the orthogonality condition \eqref{6.4} in the terms with underbraces.

Thus, one finally concludes, 
\begin{align} 
& (\eta, E_{H_{+,\alpha}}((\lambda_1,\lambda_2])) f)_{L^2((x_0,\infty); dx; \cH)}    \no \\
& \quad = \lim_{\delta \downarrow 0} \lim_{\varepsilon \downarrow 0} \f{1}{2 \pi i}
\int_{\lambda_1 + \delta}^{\lambda_2 + \delta} d\lambda \int_{x_0}^{\infty} dx \int_{x_0}^x dx' \no \\
& \hspace*{4cm}
\times \big[(\eta(x),[\theta_{\alpha}(\lambda + i \varepsilon,x,x_0)
\phi_{\alpha}(\lambda - i \varepsilon,x',x_0)^*    \no \\
& \hspace*{5cm} - \phi_{\alpha}(\lambda + i \varepsilon,x,x_0)
\theta_{\alpha}(\lambda - i \varepsilon,x,x_0)^*]f(x'))_{\cH}    \no \\
& \hspace*{4.5cm}
- (\eta(x),[\theta_{\alpha}(\lambda - i \varepsilon,x,x_0)
\phi_{\alpha}(\lambda + i \varepsilon,x',x_0)^*    \no \\
& \hspace*{5cm} - \phi_{\alpha}(\lambda - i \varepsilon,x,x_0)
\theta_{\alpha}(\lambda + i \varepsilon,x,x_0)^*]f(x'))_{\cH} \big]   \no \\
& \quad = 0.   \lb{6.6}
\end{align}
Here we used the fact that $\eta$ has compact support, rendering all $x$-integrals
over the bounded set $\supp \, (\eta)$. In addition, we employed the property that for 
fixed $x \in [x_0,\infty)$, $\phi_{\alpha}(z,x,x_0)$ and $\theta_{\alpha}(z,x,x_0)$ are entire 
with respect to $z \in \bbC$, permitting freely the interchange of the $\varepsilon$ limit with 
all integrals and implying the vanishing of the limit $\varepsilon \downarrow 0$ in the last 
step in \eqref{6.6}.

Since $\eta \in C_0^{\infty}((x_0,\infty); \cH)$ and $\lambda_1, \lambda_2 \in \bbR$ were
arbitrary, \eqref{6.6} proves $f=0$.

The fact that $H_{\pm,\min}$ are completely non-self-adjoint in $L^2((x_0,\pm\infty); dx; \cH)$ 
now follows from \eqref{5.9a}.
\end{proof}
%%%%%%%

We note that Theorem \ref{t6.2} in the context of regular (and quasi-regular) half-line differential operators with scalar coefficients has been established by Gilbert \cite[Theorem~3]{Gi72}. The 
corresponding result for $2n \times 2n$ Hamltonian systems, $n \in \bbN$, was established in 
\cite[Proposition~7.4]{DLDS88}, and the case of indefinite Sturm--Liouville operators in the associated Krein space has been treated in \cite[Proposition~4.8]{BT07}. While these proofs 
exhibit certain similarities with that of Theorem \ref{t6.2}, it appears that our approach in the 
case of a regular half-line Schr\"odinger operator with $\cB(\cH)$-valued potential is a 
canonical one.  

For future purpose we recall formulas \eqref{2.59a}--\eqref{2.62}, and now add some 
additional results:

%%%%%%%%
\begin{lemma} \lb{l6.3} 
Assume Hypothesis \ref{h6.1}\,$(i)$, respectively, $(ii)$, and let $z \in \bbC \backslash \bbR$. 
Then, for all $h \in \cH$, and $\rho_{+,\alpha}(\, \cdot \, , x_0)$-a.e.~$\lambda \in \sigma(H_{\pm,\alpha})$, 
\begin{align} 
& \pm \slim_{R \to \infty} \int_{x_0}^{\pm R} dx \, \phi_{\alpha}(\lambda,x,x_0)^* 
\psi_{\pm,\alpha}(z,x,x_0) h = \pm (\lambda -z)^{-1} h,   \lb{6.10} \\
& \pm \slim_{R \to \infty} \int_{x_0}^{\pm R} dx \, \theta_{\alpha}(\lambda,x,x_0)^* 
\psi_{\pm,\alpha}(z,x,x_0) h = \mp (\lambda -z)^{-1} m_{\pm,\alpha}(z,x_0) h,  \lb{6.11} 
\end{align}
where $\slim$ refers to the $L^2(\bbR;d\rho_{+,\alpha}(\,\cdot\,,x_0);\cH)$-limit.
\end{lemma}
%%%%%%%%
\begin{proof}
Without loss of generality, we consider the case of $H_{+,\alpha}$ only. Let 
$u \in C_0^{\infty}((x_0,\infty); \cH) \subset L^2((x_0,\infty); dx; \cH)$ and 
$v = (H_{+,\alpha} - z I)^{-1} u$, then by Theorem \ref{t2.5}, \eqref{2.40}, and \eqref{2.45}, 
\begin{align}
u &= (H_{+,\alpha} - z I) v = \slim_{\mu_2 \uparrow \infty, \mu_1 \downarrow - \infty} 
\int_{\mu_1}^{\mu_2} \phi_{\alpha}(\lambda, \, \cdot \, , x_0) \, d\rho_{+,\alpha}(\lambda,x_0) \, 
\hatt u_{+,\alpha}(\lambda)   \no \\
& = \slim_{\mu_2 \uparrow \infty, \mu_1 \downarrow - \infty} 
\int_{\mu_1}^{\mu_2} (\lambda - z) \phi_{\alpha}(\lambda, \, \cdot \, , x_0) \, d\rho_{+,\alpha} \, (\lambda,x_0) \hatt v_{+,\alpha}(\lambda), 
\end{align}
that is,
\begin{equation}
\hatt v_{+,\alpha}(\lambda) = (\lambda - z)^{-1} \hatt u_{+,\alpha}(\lambda) \, 
\text{ for $\rho_{+,\alpha}(\, \cdot \,,x_0)$-a.e.~$\lambda \in \sigma(H_{+,\alpha})$.}
\end{equation}
Hence,
\begin{align}
v &= (H_{+,\alpha} - z I)^{-1} u   \no \\
& = \slim_{\mu_2 \uparrow \infty, \mu_1 \downarrow - \infty} 
\int_{\mu_1}^{\mu_2} \phi_{\alpha}(\lambda, \, \cdot \, , x_0) \, d\rho_{+,\alpha} 
\, (\lambda,x_0) \hatt u_{+,\alpha}(\lambda) (\lambda - z)^{-1}    \no \\
& = \int_{x_0}^\infty dx' \, G_{+,\alpha}(z,\, \cdot \,,x') u(x').    
\end{align}
Thus one computes, given unitarity of $U_{+,\alpha}$ (cf.\ \eqref{2.40}, \eqref{2.45}), 
\begin{align}
& \big(h, \big((H_{+,\alpha} - z I)^{-1} u\big)(x)\big)_{\cH} 
= \int_{x_0}^{\infty} dx' \, (h, G_{+,\alpha}(z,x,x') u(x'))_{\cH}    \no \\
& \quad =  \int_{x_0}^{\infty} dx' \, (G_{+,\alpha}(z,x,x')^* h, u(x'))_{\cH}    \no \\ 
& \quad = \slim_{\mu_2 \uparrow \infty, \mu_1 \downarrow - \infty} 
\int_{\mu_1}^{\mu_2} \big(\hatt{(G_{+,\alpha}(z,x,\, \cdot \,)^* h)} (\lambda), 
d \rho_{+,\alpha} (\lambda, x_0) \, \hatt u_{+,\alpha}(\lambda)\big)_{\cH}     \no \\
& \quad = \slim_{\mu_2 \uparrow \infty, \mu_1 \downarrow - \infty} 
\int_{\mu_1}^{\mu_2} \big(h,\phi_{\alpha}(\lambda, x , x_0) 
\, d\rho_{+,\alpha}(\lambda,x_0) \, \hatt u_{+,\alpha}(\lambda)\big)_{\cH} (\lambda - z)^{-1}   \no \\
& \quad = \slim_{\mu_2 \uparrow \infty, \mu_1 \downarrow - \infty} 
\int_{\mu_1}^{\mu_2} \big((\lambda - \ol{z})^{-1}\phi_{\alpha}(\lambda, x , x_0)^* h,  
\, d\rho_{+,\alpha}(\lambda,x_0) \, \hatt u_{+,\alpha}(\lambda)\big)_{\cH}.   \no \\
\end{align}
Since $u \in C_0^{\infty}((x_0,\infty); \cH)$ was arbitrary, one concludes that 
\begin{align}
\begin{split} 
\Big(\hatt{G_{+,\alpha}(z,x, \, \cdot \,)^* h}\Big)(\lambda) 
= (\lambda - \ol{z})^{-1}\phi_{\alpha}(\lambda, x , x_0)^* h, \quad h \in \cH, \; 
z \in \bbC \backslash \bbR,&    \lb{6.11A} \\ 
\text{for $\rho_{+,\alpha}(\, \cdot \,,x_0)$-a.e.~$\lambda \in \sigma(H_{+,\alpha})$.}&
\end{split}  
\end{align}
In precisely the same manner one derives,
\begin{align}
\begin{split} 
\Big(\partial_x \hatt{G_{+,\alpha}(z,x, \, \cdot \,)^* h}\Big)(\lambda) 
= (\lambda - \ol{z})^{-1}\phi_{\alpha}'(\lambda, x , x_0)^* h, \quad h \in \cH, \; 
z \in \bbC \backslash \bbR,&    \lb{6.11B} \\ 
\text{for $\rho_{+,\alpha}(\, \cdot \,,x_0)$-a.e.~$\lambda \in \sigma(H_{+,\alpha})$.}&
\end{split}  
\end{align}
Taking $x \downarrow x_0$ in \eqref{6.11A} and \eqref{6.11B}, observing that 
\begin{align}
\begin{split} 
& G_{+,\alpha}(z,x_0,x') = \sin(\alpha) \psi_{+,\alpha}({\ol z},x',x_0),   \\ 
& [\partial_x G_{+,\alpha}(z,x,x')]\big|_{x=x_0} = \cos(\alpha) \psi_{+,\alpha}({\ol z},x',x_0), 
\end{split} 
\end{align} 
and choosing 
$h = \sin(\alpha) g$ in \eqref{6.11A} and $h = \cos(\alpha) g$ in \eqref{6.11B}, 
$g \in \cH$, then yields
\begin{align}
\hatt{\Big(\psi_{+,\alpha} (\ol z, \, \cdot \,, x_0) [\sin(\alpha)]^2 g\Big)}(\lambda) 
=  (\lambda - {\ol z})^{-1} [\sin(\alpha)]^2 g,&   \lb{6.11C} \\
\hatt{\Big(\psi_{+,\alpha} (\ol z, \, \cdot \,, x_0) [\cos(\alpha)]^2 g\Big)}(\lambda) 
= (\lambda - \ol z)^{-1} [\cos(\alpha)]^2 g,&   \lb{6.11D} \\
g \in \cH, \; z \in \bbC \backslash \bbR, \, 
\text{ for $\rho_{+,\alpha}(\, \cdot \,,x_0)$-a.e.~$\lambda \in \sigma(H_{+,\alpha})$.}&  \no 
\end{align}
Adding equations \eqref{6.11C} and \eqref{6.11D} yields relation \eqref{6.10}. 

Finally, changing $\alpha$ into $\alpha - (\pi/2)I_{\cH}$, and noticing
\begin{align}
& \phi_{\alpha - (\pi/2)I_{\cH}} (z,\, \cdot \,, x_0) = \theta_{\alpha} (z,\, \cdot \,, x_0), 
\quad  \theta_{\alpha - (\pi/2)I_{\cH}} (z,\, \cdot \,, x_0) = - \phi_{\alpha} (z,\, \cdot \,, x_0),   \\
& m_{+,\alpha - (\pi/2)I_{\cH}} (z,x_0) = - [m_{+,\alpha} (z,x_0)]^{-1},     \\
& \psi_{+,\alpha - (\pi/2) I_{\cH}} (z, \, \cdot \,,x_0) = - \psi_{+,\alpha} (z, \, \cdot \,,x_0) 
[m_{+,\alpha} (z,x_0)]^{-1},   
\end{align}
yields 
\begin{equation}
\int_{x_0}^{\infty} dx \, \theta_{\alpha}(\lambda,x,x_0)^* \psi_{\pm,\alpha}(z,x,x_0) \wti h
= \mp (\lambda -z)^{-1} m_{\pm,\alpha}(z,x_0) \wti h,
\end{equation} 
with $\wti h = - [m_{+,\alpha} (z,x_0)]^{-1} h$, and hence \eqref{6.11} since $h \in \cH$ was 
arbitrary. 
\end{proof}
%%%%%%%%

By Theorem \ref{t3.3}
\begin{equation}
[\Im(m_{\pm,\alpha}(z,x_0))]^{-1} \in \cB(\cH), \quad z\in\bbC\backslash\bbR,   \lb{6.11b}
\end{equation}
therefore
\begin{align}
& \big(\psi_{\pm,\alpha}(z, \, \cdot \,,x_0) [\pm (\Im(z))^{-1} \Im(m_{\pm,\alpha}(z,x_0))]^{-1/2} e_j,
\psi_{\pm,\alpha}(z, \, \cdot \,,x_0)    \no \\
& \qquad \times [\pm (\Im(z))^{-1}\Im(m_{\pm,\alpha}(z,x_0))]^{-1/2}
e_k\big)_{L^2((x_0,\pm\infty);dx;\cH)}    \no \\
& \quad = \big([\pm \Im(m_{\pm,\alpha}(z,x_0))]^{-1/2} e_j,
\Im(m_{\pm,\alpha}(z,x_0) \no \\
& \qquad \times [\pm \Im(m_{\pm,\alpha}(z,x_0))]^{-1/2} e_k\big)_{\cH}   \no \\
& \quad = (e_j,e_k)_{\cH} = \delta_{j,k}, \quad j,k \in \cJ, \; z\in\bbC\backslash\bbR.  \lb{6.11c}
\end{align}
Thus, one obtains in addition to \eqref{6.1} that
\begin{equation}
\big\{\Psi_{\pm,\alpha,j}(z,\, \cdot \, , x_0) =
\psi_{\pm,\alpha}(z, \, \cdot \,,x_0) [\pm (\Im(z))^{-1} \Im(m_{\pm,\alpha}(z,x_0))]^{-1/2}
e_j\big\}_{j \in \cJ}      \lb{6.12}
\end{equation}
is an orthonormal basis for $\cN_{\pm,z} = \ker\big(H_{\pm,\min}^* - z I\big)$,
$z \in \bbC \backslash \bbR$, and hence (cf.\ the definition of $P_{\cN}$ in Section \ref{s5})
\begin{equation}
P_{\cN_{\pm,i}} = \sum_{j \in \cJ}
\big(\Psi_{\pm,\alpha,j}(i,\, \cdot \, , x_0), \, \cdot \, \big)_{L^2((x_0,\pm\infty); dx; \cH)} 
\Psi_{\pm,\alpha,j}(i,\, \cdot \, , x_0).  \lb{6.13} 
\end{equation}
Consequently (cf.\ \eqref{5.3}), one obtains for the half-line Donoghue-type
$m$-functions,
\begin{align}
\begin{split}
M_{H_{\pm,\alpha}, \cN_{\pm,i}}^{Do} (z,x_0) &= \pm P_{\cN_{\pm,i}} (z H_{\pm,\alpha} + I)
(H_{\pm,\alpha} - z I)^{-1} P_{\cN_{\pm,i}} \big|_{\cN_{\pm,i}},     \lb{6.14} \\
&= \int_\bbR d\Omega_{H_{\pm,\alpha},\cN_{\pm,i}}^{Do}(\lambda,x_0) \bigg[\f{1}{\lambda-z} -
\f{\lambda}{\lambda^2 + 1}\bigg], \quad z\in\bbC\backslash\bbR,
\end{split}
\end{align}
where $\Omega_{H_{\pm,\alpha},\cN_{\pm,i}}^{Do}(\, \cdot\, , x_0)$ satisfies the analogs of 
\eqref{5.6}--\eqref{5.8} (resp., \eqref{A.42A}--\eqref{A.42b}).

Next, we explicitly compute $M_{H_{\pm,\alpha}, \cN_{\pm,i}}^{Do} (\, \cdot \,,x_0)$. 

%%%%%%%%
\begin{theorem} \lb{t6.3}
Assume Hypothesis \ref{h6.1}\,$(i)$, respectively, $(ii)$. Then,
\begin{align}
\begin{split} 
& M_{H_{\pm,\alpha}, \cN_{\pm,i}}^{Do} (z,x_0)
= \pm \sum_{j,k \in \cJ} \big(e_j, m_{\pm,\alpha}^{Do}(z,x_0) e_k\big)_{\cH}     \\
& \quad \times (\Psi_{\pm,\alpha,k}(i, \, \cdot \, ,x_0), \, \cdot \, )_{L^2((x_0,\pm \infty); dx; \cH)}    
\Psi_{\pm,\alpha,j}(i, \, \cdot \, ,x_0) \big|_{\cN_{\pm,i}},
\quad z \in \bbC \backslash \bbR,     \lb{6.15} 
\end{split} 
\end{align}
where the $\cB(\cH)$-valued Nevanlinna--Herglotz functions
$m_{\pm,\alpha}^{Do}(\, \cdot \, , x_0)$ are given by
\begin{align}
m_{\pm,\alpha}^{Do}(z,x_0) &= \pm [\pm \Im(m_{\pm,\alpha}(i,x_0))]^{-1/2}
[m_{\pm,\alpha}(z,x_0) - \Re(m_{\pm,\alpha}(i,x_0))]      \no \\
& \quad \times [\pm \Im(m_{\pm,\alpha}(i,x_0))]^{-1/2}     \lb{6.16} \\
&= d_{\pm, \alpha} \pm \int_\bbR d\omega_{\pm,\alpha}^{Do}(\lambda,x_0)
\bigg[\f{1}{\lambda-z} -
\f{\lambda}{\lambda^2 + 1}\bigg], \quad z\in\bbC\backslash\bbR.     \lb{6.16a}
\end{align}
Here $d_{\pm,\alpha} = \Re(m_{\pm,\alpha}^{Do}(i,x_0)) \in \cB(\cH)$, and
\begin{equation}
\omega_{\pm,\alpha}^{Do}(\, \cdot \,,x_0) = [\pm \Im(m_{\pm,\alpha}(i,x_0))]^{-1/2}
\rho_{\pm,\alpha}(\,\cdot\,,x_0) [\pm \Im(m_{\pm,\alpha}(i,x_0))]^{-1/2}
\end{equation}
satisfy the analogs of \eqref{A.42a}, \eqref{A.42b}.
\end{theorem}
%%%%%%%%
\begin{proof}
We will consider the right half-line $[x_0,\infty)$. To verify \eqref{6.15} it suffices to
insert \eqref{6.13} into \eqref{6.14} and then apply \eqref{2.28}, \eqref{2.29} to compute,
\begin{align}
& \big(\Psi_{+,\alpha,j}(i,\, \cdot \, , x_0), (z H_{+,\alpha} + I)
(H_{+,\alpha} - z I)^{-1} \Psi_{+,\alpha,k}(i,\, \cdot \, , x_0)\big)_{L^2((x_0,\infty); dx; \cH)}    \no \\
& \quad = \big(\hatt e_{j,+,\alpha}, (z \cdot + I_{\cH}) (\cdot - z I_{\cH})^{-1}
\hatt e_{k,+,\alpha}\big)_{L^2(\bbR; d \rho_{+,\alpha}; \cH)}    \no \\
& \quad = \int_{\bbR} d \big(\hatt e_{j,+,\alpha}, \rho_{+,\alpha} (\lambda,x_0)
 \hatt e_{k,+,\alpha}\big)_{\cH} \, \f{z \lambda + 1}{\lambda - z},  \quad j, k \in \cJ,  \lb{6.17}
\end{align}
where
\begin{align}
\hatt e_{j,+,\alpha} (\lambda) &= \int_{x_0}^{\infty} dx \, \phi_{\alpha}(\lambda,x,x_0)^*
\psi_{+,\alpha}(i,x,x_0) [\Im(m_{+,\alpha}(i,x_0))]^{-1/2} e_j     \no \\
&= (\lambda - i)^{-1} [\Im(m_{+,\alpha}(i,x_0))]^{-1/2} e_j, \quad j \in \cJ,    \lb{6.18}
\end{align}
employing \eqref{6.10} (with $z=i$). Thus,
\begin{align}
\eqref{6.17} &= \int_{\bbR} d \big([\Im(m_{+,\alpha}(i,x_0))]^{-1/2} e_j,
\rho_{+,\alpha} (\lambda,x_0) [\Im(m_{+,\alpha}(i,x_0))]^{-1/2} e_k\big)_{\cH}       \no \\
& \hspace*{8.5mm} \times \f{z \lambda + 1}{\lambda - z} \f{1}{\lambda^2 + 1}     \no \\
& = \int_{\bbR} d \big([\Im(m_{+,\alpha}(i,x_0))]^{-1/2} e_j,
\rho_{+,\alpha} (\lambda,x_0) [\Im(m_{+,\alpha}(i,x_0))]^{-1/2} e_k\big)_{\cH}     \no \\
& \hspace*{8.5mm} \times \bigg[\f{1}{\lambda - z} - \f{\lambda}{\lambda^2 + 1}\bigg]   \no \\
& = \big([\Im(m_{+,\alpha}(i,x_0))]^{-1/2} e_j, [m_{+,\alpha}(z,x_0) - \Re(m_{+,\alpha}(i,x_0)]
\no \\
& \hspace*{5mm}
\times [\Im(m_{+,\alpha}(i,x_0))]^{-1/2} e_k\big)_{\cH},     \lb{6.19}
\end{align}
using \eqref{2.25}, \eqref{2.25a} in the final step.
\end{proof}
%%%%%%%%

%%%%%%%
\begin{remark} \lb{r6.4}
Combining Corollary \ref{c5.8} and Theorem \ref{t6.2} proves that the entire spectral information
for $H_{\pm,\alpha}$, contained in the corresponding family of spectral projections
$\{E_{H_{\pm,\alpha}}(\lambda)\}_{\lambda \in \bbR}$ in $L^2((x_0,\pm\infty); dx; \cH)$, is
already encoded in the operator-valued measure
$\{\Omega_{H_{\pm,\alpha},\cN_{\pm,i}}^{Do}(\lambda,x_0)\}_{\lambda \in \bbR}$ in
$\cN_{\pm,i}$ (including multiplicity properties of the spectrum of $H_{\pm,\alpha}$). By the
same token, invoking Theorem \ref{t6.3} shows that the entire spectral information
for $H_{\pm,\alpha}$ is already contained in
$\{\omega_{\pm,\alpha}^{Do}(\lambda,x_0)\}_{\lambda \in \bbR}$ in $\cH$. 
${}$ \hfill $\diamond$
\end{remark}
%%%%%%%

\subsection{The full-line case:} 
In the remainder of this section we turn to Schr\"odinger operators on $\bbR$, assuming
Hypotheis \ref{h2.8}. Decomposing
\begin{equation}
L^2(\bbR; dx; \cH) = L^2((-\infty,x_0); dx; \cH) \oplus L^2((x_0, \infty); dx; \cH),  \lb{6.20}
\end{equation}
and introducing the orthogonal projections $P_{\pm,x_0}$ of $L^2(\bbR; dx; \cH)$ onto
the left/right subspaces $L^2((x_0,\pm\infty); dx; \cH)$, we now define a particular minimal 
operator $H_{\min}$ in $L^2(\bbR; dx; \cH)$ via 
\begin{align}
H_{\min} &:= H_{-,\min} \oplus H_{+,\min}, \quad
H_{\min}^* = H_{-,\min}^* \oplus H_{+,\min}^*,     \lb{6.23} \\
\cN_z &\, = \ker\big(H_{\min}^* - z I \big) =  \ker\big(H_{-,\min}^* - z I \big) \oplus
 \ker\big(H_{+,\min}^* - z I \big)     \no \\
&\, = \cN_{-,z} \oplus \cN_{+,z}, \quad z \in \bbC \backslash \bbR.   \lb{6.23a}
\end{align}
We note that \eqref{6.23} is not the standard minimal operator associated with the differential 
expression $\tau$ on $\bbR$. Usually, one introduces
\begin{align}
& \hatt \oT_{\min} f = \tau f,   \no \\
& f\in \dom\big(\hatt \oT_{\min}\big)=\big\{g\in L^2(\bbR;dx;\cH) \,\big|\,
g\in W^{2,1}_{\rm loc}(\bbR;dx;\cH); \, \supp(g) \, \text{compact};     \no \\
& \hspace*{7.6cm} \tau g\in L^2(\bbR;dx;\cH)\big\}.    \lb{6.24} 
\end{align}
However, due to our limit-point assumption at $\pm \infty$, $ \hatt \oT_{\min}$ is essentially 
self-adjoint and hence (cf.\ \eqref{2.53}),
\begin{equation}
\ol{\hatt \oT_{\min}} = \oT,
\end{equation}
rendering $\hatt \oT_{\min}$ unsuitable as a minimal operator with nonzero deficiency indices. 
Consequently, $\oT$ given by \eqref{2.53}, as well as the Dirichlet extension,  
$\oT_{\rm D} = \oT_{-,{\rm D}} \oplus \oT_{+,{\rm D}}$, where $\oT_{\pm,{\rm D}} = \oT_{\pm, 0}$ 
(i.e., $\alpha = 0$ in \eqref{3.9}, see also our notational conventions following \eqref{3.63A}),  
are particular self-adjoint extensions of $H_{\min}$ in \eqref{6.23}. 

Associated with the operator $H$ in $L^2(\bbR; dx; \cH)$ (cf.\ \eqref{2.53}) we now 
introduce its $2 \times 2$ block operator representation via
\begin{equation}
(H - z I)^{-1} = \begin{pmatrix} P_{-,x_0} (H - zI)^{-1} P_{-,x_0}
& P_{-,x_0} (H - z I)^{-1} P_{+,x_0} \\
P_{+,x_0} (H - z I)^{-1} P_{-,x_0} & P_{+,x_0} (H - z I)^{-1} P_{+,x_0}
\end{pmatrix}.     \lb{6.21}
\end{equation}

Hence (cf.\ \eqref{6.12}),
\begin{align}
\begin{split}
&\big\{\hatt \Psi_{-,\alpha,j}(z,\, \cdot \, , x_0) =
P_{-,x_0} \psi_{-,\alpha}(z, \, \cdot \, ,x_0)[- (\Im(z))^{-1} \Im(m_{-,\alpha}(z,x_0))]^{-1/2} e_j,  \\
& \;\;\, \hatt \Psi_{+,\alpha,j}(z,\, \cdot \, , x_0) =
P_{+,x_0} \psi_{+,\alpha}(z, \, \cdot \, ,x_0)[(\Im(z))^{-1} \Im(m_{+,\alpha}(z,x_0))]^{-1/2}
e_j\big\}_{j \in \cJ}    \lb{6.22}
\end{split}
\end{align}
is an orthonormal basis for $\cN_{z} = \ker\big(H_{\min}^* - z I \big)$,
$z \in \bbC \backslash \bbR$, if $\{e_j\}_{j \in \cJ}$ is an orthonormal basis for $\cH$,
and (cf.\ \eqref{6.13})
\begin{align}
P_{\cN_i} &= P_{\cN_-,i} \oplus P_{\cN_+,i}   \no \\
& = \sum_{j \in \cJ} \Big[
\big(\psi_{-,\alpha}(i, \, \cdot \,,x_0) [- \Im(m_{-,\alpha}(i,x_0))]^{-1/2}e_j,
\, \cdot \, \big)_{L^2((-\infty,x_0); dx; \cH)}     \no \\
& \hspace*{4cm}
\times \psi_{-,\alpha}(i, \, \cdot \,,x_0) [- \Im(m_{-,\alpha}(i,x_0))]^{-1/2}e_j    \lb{6.26} \\
& \hspace*{1.1cm} \oplus
\big(\psi_{+,\alpha}(i, \, \cdot \,,x_0) [\Im(m_{+,\alpha}(i,x_0))]^{-1/2}e_j,
\, \cdot \, \big)_{L^2((x_0,\infty); dx; \cH)}    \no \\
& \hspace*{4cm}
\times \psi_{+,\alpha}(i, \, \cdot \,,x_0) [\Im(m_{+,\alpha}(i,x_0))]^{-1/2}e_j \Big],  \no \\
&= \sum_{j \in \cJ} \big[\big(\hatt\Psi_{-,\alpha,j}(i,\, \cdot \,, x_0), 
\, \cdot \,\big)_{L^2((-\infty,x_0); dx; \cH)}
\hatt \Psi_{-,\alpha,j}(i,\, \cdot \,,x_0)   \no \\
& \hspace*{1.2cm } \oplus \big(\hatt \Psi_{+,\alpha,j}(i,\, \cdot \,,x_0), 
\, \cdot \,\big)_{L^2((x_0,\infty); dx; \cH)} \hatt \Psi_{+,\alpha,j}(i,\, \cdot \,,x_0)\big]
\end{align}
is the orthogonal projection onto $\cN_i$.

Consequently (cf.\ \eqref{5.3}), one obtains for the full-line Donoghue-type $m$-function,
\begin{align}
\begin{split}
M_{H, \cN_i}^{Do} (z) &= P_{\cN_i} (z H + I)
(H - z I)^{-1} P_{\cN_i} \big|_{\cN_i},     \lb{6.27} \\
&= \int_\bbR d\Omega_{H,\cN_i}^{Do}(\lambda) \bigg[\f{1}{\lambda-z} -
\f{\lambda}{\lambda^2 + 1}\bigg], \quad z\in\bbC\backslash\bbR,
\end{split}
\end{align}
where $\Omega_{H,\cN_i}^{Do}(\cdot)$ satisfies the analogs of \eqref{5.6}--\eqref{5.8}
(resp., \eqref{A.42A}--\eqref{A.42b}). With respect to the decomposition \eqref{6.20}, one
can represent $M_{H, \cN_i}^{Do} (\cdot)$ as the $2 \times 2$ block operator,
\begin{align}
& M_{H, \cN_i}^{Do} (\cdot) = \big(M_{H, \cN_i,\ell,\ell'}^{Do} (\cdot)\big)_{0 \leq \ell, \ell' \leq 1}
\no \\
& \quad = z \left(\begin{smallmatrix} P_{\cN_-,i} & 0 \\
0 & P_{\cN_+,i} \end{smallmatrix}\right)      \lb{6.32} \\
& \qquad + (z^2 + 1) \left(\begin{smallmatrix} 
P_{\cN_-,i} P_{-,x_0} (H - zI)^{-1} P_{-,x_0} P_{\cN_-,i} & \;\;\;  
P_{\cN_-,i} P_{-,x_0} (H - zI)^{-1} P_{+,x_0} P_{\cN_+,i}   \\ 
P_{\cN_+,i} P_{+,x_0} (H - zI)^{-1} P_{-,x_0} P_{\cN_-,i} & \;\;\;
P_{\cN_+,i} P_{+,x_0} (H - zI)^{-1} P_{+,x_0} P_{\cN_+,i}  
\end{smallmatrix}\right),     \no 
\end{align}
and hence explicitly obtains,
\begin{align}
& M_{H, \cN_i,0,0}^{Do} (z)  \no \\
& \quad = \sum_{j,k \in \cJ}
\big(\hatt \Psi_{-,\alpha,j}(i,\, \cdot \,,x_0), (z H + I)(H - z I)^{-1}
\hatt \Psi_{-,\alpha,k}(i,\, \cdot \,,x_0)\big)_{L^2(\bbR; dx; \cH)}   \no \\
& \hspace*{1.65cm} \times
\big(\hatt \Psi_{-,\alpha,k}(i,\, \cdot \,,x_0), \, \cdot \, \big)_{L^2(\bbR; dx; \cH)}
\hatt \Psi_{-,\alpha,j}(i,\, \cdot \,,x_0),     \lb{6.33} \\
& M_{H, \cN_i,0,1}^{Do} (z)  \no \\
& \quad = \sum_{j,k \in \cJ}
\big(\hatt \Psi_{-,\alpha,j}(i,\, \cdot \,,x_0), (z H + I)(H - z I)^{-1}
\hatt \Psi_{+,\alpha,k}(i,\, \cdot \,,x_0)\big)_{L^2(\bbR; dx; \cH)}   \no \\
& \hspace*{1.65cm} \times
\big(\hatt \Psi_{+,\alpha,k}(i,\, \cdot \,,x_0), \, \cdot \, \big)_{L^2(\bbR; dx; \cH)}
\hatt \Psi_{-,\alpha,j}(i,\, \cdot \,,x_0),     \lb{6.34} \\
& M_{H, \cN_i,1,0}^{Do} (z)  \no \\
& \quad = \sum_{j,k \in \cJ}
\big(\hatt \Psi_{+,\alpha,j}(i,\, \cdot \,,x_0), (z H + I)(H - z I)^{-1}
\hatt \Psi_{-,\alpha,k}(i,\, \cdot \,,x_0)\big)_{L^2(\bbR; dx; \cH)}   \no \\
& \hspace*{1.65cm} \times
\big(\hatt \Psi_{-,\alpha,k}(i,\, \cdot \,,x_0), \, \cdot \, \big)_{L^2(\bbR; dx; \cH)}
\hatt \Psi_{+,\alpha,j}(i,\, \cdot \,,x_0),     \lb{6.35} \\
& M_{H, \cN_i,1,1}^{Do} (z)  \no \\
& \quad = \sum_{j,k \in \cJ}
\big(\hatt \Psi_{+,\alpha,j}(i,\, \cdot \,,x_0), (z H + I)(H - z I)^{-1}
\hatt \Psi_{+,\alpha,k}(i,\, \cdot \,,x_0)\big)_{L^2(\bbR; dx; \cH)}   \no \\
& \hspace*{1.65cm} \times
\big(\hatt \Psi_{+,\alpha,k}(i,\, \cdot \,,x_0), \, \cdot \, \big)_{L^2(\bbR; dx; \cH)}
\hatt \Psi_{+,\alpha,j}(i,\, \cdot \,,x_0),    \lb{6.36} \\
& \hspace*{7.55cm} z\in\bbC\backslash\bbR.     \no
\end{align}

Taking a closer look at equations \eqref{6.33}--\eqref{6.36} we now state the following
preliminary result:

%%%%%%%
\begin{lemma} \lb{l6.5}
Assume Hypothesis \ref{h2.8}. Then,
\begin{align}
& \big(\hatt \Psi_{\varepsilon,\alpha,j}(i,\, \cdot \,,x_0), (zH + I)(H - z I)^{-1}
\hatt \Psi_{\varepsilon',\alpha,k}(i,\, \cdot \,,x_0)\big)_{L^2(\bbR;dx;\cH)}    \no \\
& \quad = \int_{\bbR} d\big(\hatt e_{\varepsilon,\alpha,j}(\lambda),
\Omega_{\alpha}(\lambda,x_0) \hatt e_{\varepsilon',\alpha,k}(\lambda)\big)_{\cH^2}
\, \f{z \lambda + 1}{\lambda - z}    \no \\
& \quad = \int_{\bbR} d\big(e_{\varepsilon,\alpha,j}(\lambda),
\Omega_{\alpha}(\lambda,x_0) e_{\varepsilon',\alpha,k}(\lambda)\big)_{\cH^2}
\, \f{z \lambda + 1}{(\lambda - z)(\lambda^2 + 1)}    \no \\
& \quad = \big(e_{\varepsilon,\alpha,j}, [M_{\alpha}(z,x_0)
- \Re(M_{\alpha}(i,x_0)] e_{\varepsilon',\alpha,k}\big)_{\cH^2},    \lb{6.37} \\
& \hspace*{2.45cm}
\varepsilon, \varepsilon' \in \{+,-\}, \; j,k \in \cJ, \; z \in\bbC\backslash\bbR,   \no
\end{align}
where
\begin{align}
& \hatt e_{\varepsilon,\alpha,j}(\lambda)
= \big(\hatt e_{\varepsilon,\alpha,j,0}(\lambda),
\hatt e_{\varepsilon,\alpha,j,1}(\lambda)\big)^\top    \no \\
& \quad = \f{1}{\lambda - i}  e_{\varepsilon,\alpha,j}
= \f{1}{\lambda - i}  (e_{\varepsilon,\alpha,j,0}, e_{\varepsilon,\alpha,j,1})^\top   \no \\
& \quad = \f{1}{\lambda - i}  \big(- \varepsilon m_{\varepsilon,\alpha}(i,x_0)
[\varepsilon \Im(m_{\varepsilon,\alpha}(i,x_0))]^{-1/2} e_j,
\varepsilon [\varepsilon \Im(m_{\varepsilon,\alpha}(i,x_0))]^{-1/2} e_j\big)^\top,   \no \\
& \hspace*{7cm} \varepsilon \in \{+,-\}, \; j \in \cJ, \; \lambda \in \bbR.    \lb{6.38}
\end{align}
\end{lemma}
%%%%%%%
\begin{proof}
The first two equalities in \eqref{6.37} follow from \eqref{2.73}, \eqref{2.74} upon introducing
$\hatt e_{\varepsilon,\alpha,j}(\cdot) = \big(\hatt e_{\varepsilon,\alpha,j,0}(\cdot),
\hatt e_{\varepsilon,\alpha,j,1}(\cdot)\big)^\top$, where
\begin{align}
& \hatt e_{\varepsilon,\alpha,j,0}(\lambda) = \varepsilon \int_{x_0}^{\varepsilon \infty} dx \,
\theta_{\alpha}(\lambda,x,x_0)^* \psi_{\varepsilon,\alpha}(i,x,x_0)
[\varepsilon \Im(m_{\varepsilon,\alpha}(i,x_0))]^{-1/2} e_j     \no \\
& \hspace*{1.5cm} = - \varepsilon (\lambda - i)^{-1} m_{\varepsilon,\alpha}(i,x_0)
 [\varepsilon \Im(m_{\varepsilon,\alpha}(i,x_0))]^{-1/2} e_j,    \lb{6.39} \\
& \hatt e_{\varepsilon,\alpha,j,1}(\lambda) = \varepsilon \int_{x_0}^{\varepsilon \infty} dx \,
\phi_{\alpha}(\lambda,x,x_0)^* \psi_{\varepsilon,\alpha}(i,x,x_0)
[\varepsilon \Im(m_{\varepsilon,\alpha}(i,x_0))]^{-1/2} e_j     \no \\
& \hspace*{1.5cm} = \varepsilon (\lambda - i)^{-1}
[\varepsilon \Im(m_{\varepsilon,\alpha}(i,x_0))]^{-1/2} e_j ,    \lb{6.40} \\
& \hspace*{3.3cm} \varepsilon \in \{+,-\}, \; j \in \cJ, \; \lambda \in \bbR,   \no
\end{align}
and we employed \eqref{6.11}, \eqref{6.10} (with $z=i$) to arrive at \eqref{6.39},
\eqref{6.40}. The third equality in \eqref{6.37} follows from \eqref{2.71b}, \eqref{2.71c}.
\end{proof}
%%%%%%%

Next, further reducing the computation \eqref{6.37} to scalar products of the type
$(e_j, \cdots e_k)_{\cH}$, $j,k \in \cH$, naturally leads to a $2 \times 2$ block operator
\begin{equation}
M_{\alpha}^{Do} (\, \cdot \,,x_0) = \big(M_{\alpha,\ell,\ell'}^{Do} (\, \cdot \,,x_0)\big)_{0 
\leq \ell, \ell' \leq 1},
\end{equation}
where
\begin{align}
(e_j, M_{\alpha,0,0}^{Do} (z,x_0) e_k)_{\cH} &= \big(e_{-,\alpha,j}, [M_{\alpha}(z,x_0)
- \Re(M_{\alpha}(i,x_0)] e_{-,\alpha,k}\big)_{\cH^2},    \no \\
(e_j, M_{\alpha,0,1}^{Do} (z,x_0) e_k)_{\cH} &= \big(e_{-,\alpha,j}, [M_{\alpha}(z,x_0)
- \Re(M_{\alpha}(i,x_0)] e_{+,\alpha,k}\big)_{\cH^2},    \no \\
(e_j, M_{\alpha,1,0}^{Do} (z,x_0) e_k)_{\cH} &= \big(e_{+,\alpha,j}, [M_{\alpha}(z,x_0)
- \Re(M_{\alpha}(i,x_0)] e_{-,\alpha,k}\big)_{\cH^2},    \lb{6.41}  \\
(e_j, M_{\alpha,1,1}^{Do} (z,x_0) e_k)_{\cH} &= \big(e_{+,\alpha,j}, [M_{\alpha}(z,x_0)
- \Re(M_{\alpha}(i,x_0)] e_{+,\alpha,k}\big)_{\cH^2},    \no \\
& \hspace*{4.5cm} j,k \in \cJ, \; z \in\bbC\backslash\bbR.    \no
\end{align}

%%%%%%%
\begin{theorem} \lb{t6.6}
Assume Hypothesis \ref{h2.8}.~Then $M_{\alpha}^{Do}(\, \cdot \, ,x_0)$ is a
$\cB\big(\cH^2\big)$-valued \\ Nevanlinna--Herglotz function given by
\begin{align}
M_{\alpha}^{Do} (z,x_0) &= T_{\alpha}^* M_{\alpha}(z,x_0) T_{\alpha}
+ E_{\alpha}     \lb{6.42} \\
&= D_{\alpha} + \int_\bbR d\Omega_{\alpha}^{Do}(\lambda,x_0) \bigg[\f{1}{\lambda-z} -
\f{\lambda}{\lambda^2 + 1}\bigg], \quad z\in\bbC\backslash\bbR,    \lb{6.43} 
\end{align}
where the $2 \times 2$ block operators $T_{\alpha} \in \cB\big(\cH^2\big)$ and
$E_{\alpha} \in \cB\big(\cH^2\big)$ are defined by
\begin{align}
& T_{\alpha} = \begin{pmatrix} m_{-,\alpha}(i,x_0) [-\Im(m_{-,\alpha}(i,x_0))]^{-1/2}
& - m_{+,\alpha}(i,x_0) [\Im(m_{+,\alpha}(i,x_0))]^{-1/2}  \\
- [-\Im(m_{-,\alpha}(i,x_0))]^{-1/2} & [\Im(m_{+,\alpha}(i,x_0))]^{-1/2}
\end{pmatrix},     \lb{6.46} \\
& E_{\alpha} = \begin{pmatrix} 0 & E_{\alpha,0,1}  \\
E_{\alpha,1,0} & 0 \end{pmatrix} = E_{\alpha}^*,   \no \\
& E_{\alpha,0,1} = 2^{-1} [-\Im(m_{-,\alpha}(i,x_0))]^{-1/2}
[m_{-,\alpha}(-i,x_0) - m_{+,\alpha}(i,x_0)]    \no \\
& \hspace*{1.3cm} \times [\Im(m_{+,\alpha}(i,x_0))]^{-1/2},      \lb{6.47} \\
& E_{\alpha,1,0} =  2^{-1} [\Im(m_{+,\alpha}(i,x_0))]^{-1/2}
[m_{-,\alpha}(i,x_0) - m_{+,\alpha}(-i,x_0)]     \no \\
& \hspace*{1.3cm} \times [-\Im(m_{-,\alpha}(i,x_0))]^{-1/2},      \no
\end{align}
and $T_{\alpha}^{-1} \in \cB\big(\cH^2\big)$, with
\begin{align}
& \big(T_{\alpha}^{-1}\big)_{0,0} = [-\Im(m_{-,\alpha}(i,x_0))]^{1/2}
[m_{-,\alpha}(i,x_0) - m_{+,\alpha}(i,x_0)]^{-1},    \lb{6.47a}\\
& \big(T_{\alpha}^{-1}\big)_{0,1} = [-\Im(m_{-,\alpha}(i,x_0))]^{1/2}
[m_{-,\alpha}(i,x_0) - m_{+,\alpha}(i,x_0)]^{-1} m_{+,\alpha}(i,x_0),    \lb{6.48}\\
& \big(T_{\alpha}^{-1}\big)_{1,0} = [\Im(m_{+,\alpha}(i,x_0))]^{1/2}
[m_{-,\alpha}(i,x_0) - m_{+,\alpha}(i,x_0)]^{-1},   \lb{6.49} \\
& \big(T_{\alpha}^{-1}\big)_{1,1} = [\Im(m_{+,\alpha}(i,x_0))]^{1/2}
[m_{-,\alpha}(i,x_0) - m_{+,\alpha}(i,x_0)]^{-1} m_{-,\alpha}(i,x_0).     \lb{6.50}
\end{align}
In addition, $D_{\alpha} = \Re\big(M_{\alpha}^{Do} (i,x_0)\big) \in \cB\big(\cH^2\big)$,
and $\Omega_{\alpha}^{Do}(\, \cdot \, ,x_0)
= T_{\alpha}^* \Omega_{\alpha}(\, \cdot \, ,x_0) T_{\alpha}$ satisfy the analogs of
\eqref{A.42a}, \eqref{A.42b}.
\end{theorem}
%%%%%%%
\begin{proof}
While \eqref{6.43} is clear from \eqref{6.42}, and similarly, \eqref{6.47a}--\eqref{6.50} is clear from
\eqref{6.46}, the main burden of proof consists in verifying \eqref{6.42}, given \eqref{6.46},
\eqref{6.47}. This can be achieved after straightforward, yet tedious computations.
To illustrate the nature of this computations we just focus on the $(0,0)$-entry of the
$2 \times 2$ block operator \eqref{6.42} and consider the term (cf.\ the first equation in
\eqref{6.41}), $(e_{-,\alpha,j}, M_{\alpha}(z,x_0) e_{-,\alpha,k})_{\cH^2}$, temporarily
suppressing $x_0$ and $\alpha$ for simplicity:
\begin{align}
& (e_{-,\alpha,j}, M_{\alpha}(z,x_0) e_{-,\alpha,k})_{\cH^2}
= \bigg(\bigg(\begin{smallmatrix} m_-(i) [- \Im(m_-(i))]^{-1/2} e_j \\
- [- \Im(m_-(i))]^{-1/2} e_j\end{smallmatrix}\bigg),   \no \\
& \qquad \times \bigg(\begin{smallmatrix}
[m_-(z) - m_+(z)]^{-1}
& 2^{-1} [m_-(z) - m_+(z)]^{-1} [m_-(z) + m_+(z)] \\
2^{-1} [m_-(z) + m_+(z)] [m_-(z) - m_+(z)]^{-1}
\end{smallmatrix} \bigg)   \no \\
& \qquad \times \bigg(\bigg(\begin{smallmatrix} m_-(i) [- \Im(m_-(i))]^{-1/2} e_j \\
- [- \Im(m_-(i))]^{-1/2} e_j\end{smallmatrix}\bigg)\bigg)_{\cH^2}    \no \\
& \quad = \big(m_-(i) [- \Im(m_-(i))]^{-1/2} e_j, [m_-(z) - m_+(z)]^{-1}   \no \\
& \hspace*{9mm} \times m_-(i) [- \Im(m_-(i))]^{-1/2} e_k\big)_{\cH}    \no \\
& \qquad - 2^{-1} \big(m_-(i) [- \Im(m_-(i))]^{-1/2} e_j, [m_-(z) - m_+(z)]^{-1} [m_-(z) + m_+(z)]   \no \\
& \hspace*{1.7cm} \times [- \Im(m_-(i))]^{-1/2} e_k\big)_{\cH}    \no \\
& \qquad - 2^{-1} \big([- \Im(m_-(i))]^{-1/2} e_j, [m_-(z) + m_+(z)] [m_-(z) - m_+(z)]^{-1}   \no \\
& \hspace*{1.7cm} \times m_-(i) [- \Im(m_-(i))]^{-1/2} e_k\big)_{\cH}    \no \\
& \qquad + \big([- \Im(m_-(i))]^{-1/2} e_j, m_{\mp}(z)] [m_-(z) - m_+(z)]^{-1} m_{\pm}(z)   \no \\
& \hspace*{1.7cm} \times [- \Im(m_-(i))]^{-1/2} e_k\big)_{\cH}  \no \\
& \quad = \big(e_j,  [- \Im(m_-(i))]^{-1/2} m_-(-i) [m_-(z) - m_+(z)]^{-1}   \no \\
& \hspace*{9mm} \times m_-(i) [- \Im(m_-(i))]^{-1/2} e_k\big)_{\cH}    \no \\
& \qquad - 2^{-1} \big(e_j, [- \Im(m_-(i))]^{-1/2} m_-(-i)
[m_-(z) - m_+(z)]^{-1} [m_-(z) + m_+(z)]   \no \\
& \hspace*{1.7cm} \times [- \Im(m_-(i))]^{-1/2} e_k\big)_{\cH}    \no \\
& \qquad - 2^{-1} \big(e_j, [- \Im(m_-(i))]^{-1/2} [m_-(z) + m_+(z)] [m_-(z) - m_+(z)]^{-1}   \no \\
& \hspace*{1.7cm} \times m_-(i) [- \Im(m_-(i))]^{-1/2} e_k\big)_{\cH}    \no \\
& \qquad + \big(e_j, [- \Im(m_-(i))]^{-1/2} m_{\mp}(z)] [m_-(z) - m_+(z)]^{-1} m_{\pm}(z)   \no \\
& \hspace*{1.7cm} \times [- \Im(m_-(i))]^{-1/2} e_k\big)_{\cH}, \quad z\in\bbC\backslash\bbR.
\lb{6.51}
\end{align}
Explicitly computing $(e_j, [T_{\alpha}^* M_{\alpha}(z,x_0)T_{\alpha}]_{0,0} e_k)_{\cH}$,
given $T_{\alpha}$ in \eqref{6.46} yields the same expression as in \eqref{6.51}. Similarly,
one verifies that
\begin{equation}
(e_{-,\alpha,j}, \Re(M_{\alpha}(i,x_0)) e_{-,\alpha,k})_{\cH^2} = 0,
\end{equation}
verifying the  $(0,0)$-entry of \eqref{6.42}. The remaining three entries are verified analogously.
\end{proof}
%%%%%%%

Combining Lemma \ref{l6.5} and Theorem \ref{t6.6} then yields the following result:

%%%%%%%
\begin{theorem} \lb{t6.7}
Assume Hypothesis \ref{h2.8}.~Then 
$M_{H, \cN_i}^{Do} (\cdot) = \big(M_{H, \cN_i,\ell,\ell'}^{Do} (\cdot)\big)_{0 \leq \ell, \ell' \leq 1}$, 
explicitly given by \eqref{6.27}--\eqref{6.36}, is of the form,
\begin{align}
& M_{H, \cN_i,0,0}^{Do} (z) = \sum_{j,k \in \cJ}
(e_j, M_{\alpha,0,0}^{Do}(z,x_0) e_k)_{\cH}  \no \\
& \hspace*{3.1cm} \times
\big(\hatt \Psi_{-,\alpha,k}(i,\, \cdot \,,x_0), \, \cdot \, \big)_{L^2(\bbR; dx; \cH)}
\hatt \Psi_{-,\alpha,j}(i,\, \cdot \,,x_0),     \lb{6.53} \\
& M_{H, \cN_i,0,1}^{Do} (z) = \sum_{j,k \in \cJ}
(e_j, M_{\alpha,0,1}^{Do}(z,x_0) e_k)_{\cH}   \no \\
& \hspace*{3.1cm} \times
\big(\hatt \Psi_{+,\alpha,k}(i,\, \cdot \,,x_0), \, \cdot \, \big)_{L^2(\bbR; dx; \cH)}
\hatt \Psi_{-,\alpha,j}(i,\, \cdot \,,x_0),     \lb{6.54} \\
& M_{H, \cN_i,1,0}^{Do} (z) = \sum_{j,k \in \cJ}
(e_j, M_{\alpha,1,0}^{Do}(z,x_0) e_k)_{\cH}    \no \\
& \hspace*{3.1cm} \times
\big(\hatt \Psi_{-,\alpha,k}(i,\, \cdot \,,x_0), \, \cdot \, \big)_{L^2(\bbR; dx; \cH)}
\hatt \Psi_{+,\alpha,j}(i,\, \cdot \,,x_0),     \lb{6.55} \\
& M_{H, \cN_i,1,1}^{Do} (z) = \sum_{j,k \in \cJ}
(e_j, M_{\alpha,1,1}^{Do}(z,x_0) e_k)_{\cH}    \no \\
& \hspace*{3.1cm} \times
\big(\hatt \Psi_{+,\alpha,k}(i,\, \cdot \,,x_0), \, \cdot \, \big)_{L^2(\bbR; dx; \cH)}
\hatt \Psi_{+,\alpha,j}(i,\, \cdot \,,x_0),    \lb{6.56} \\
& \hspace*{8.95cm} z\in\bbC\backslash\bbR,     \no
\end{align}
with $M_{\alpha}^{Do}(\, \cdot \,,x_0)$ given by \eqref{6.42}--\eqref{6.47}.
\end{theorem}
%%%%%%%

%%%%%%%
\begin{remark} \lb{r6.8}
Combining Corollary \ref{c5.8} and Theorem \ref{t6.7} proves that the entire spectral information
for $H$, contained in the corresponding family of spectral projections
$\{E_H(\lambda)\}_{\lambda \in \bbR}$ in $L^2(\bbR; dx; \cH)$, is
already encoded in the operator-valued measure
$\{\Omega_{H,\cN_i}^{Do}(\lambda)\}_{\lambda \in \bbR}$ in
$\cN_i$ (including multiplicity properties of the spectrum of $H$). In addition, 
invoking Theorem \ref{t6.6} shows that for any fixed $\alpha = \alpha^* \in \cB(\cH)$, 
$x_0 \in \bbR$, the entire spectral information for $H$ is already contained in
$\{\Omega_{\alpha}^{Do}(\lambda,x_0)\}_{\lambda \in \bbR}$ in $\cH^2$. 
\hfill $\diamond$ 
\end{remark}
%%%%%%%

%%%%%%%%%%%%%%%% Appendices %%%%%%%%%%%%%%%
%%%%%%%%%%%%%%%%%%%%%%%%%%%%%%%%%%%%%%
\appendix
%%%%%%%%%%%%%%%%%%%%%%%%%%%%%%%%%%%%%%
%%%%%%%%%%%%%%%% Appendix A %%%%%%%%%%%%%%%%
\section{Basic Facts on Bounded Operator-Valued Nevanlinna--Herglotz Functions} \lb{sA}
\setcounter{theorem}{0}
\setcounter{equation}{0}
%%%%%%%%%%%%%%%%%%%%%%%%%%%%%%%%%%%%%%%
%%%%%%%%%%%%%%%%%%%%%%%%%%%%%%%%%%%%%%%

We review some basic facts on (bounded) operator-valued Nevanlinna--Herglotz
functions (also called Nevanlinna, Pick, $R$-functions, etc.),
frequently employed in the bulk of this paper. For additional details concerning the material in this appendix we refer to \cite{GWZ13}, \cite{GWZ13b}.

Throughout this appendix, $\cH$ is a separable, complex Hilbert space with inner product denoted by $(\, \cdot \,,\, \cdot \,)_{\cH}$, identity operator abbreviated by $I_{\cH}$. We also denote
$\bbC_{\pm} = \{z \in \bbC \,| \pm \Im(z) > 0\}$.

%%%%%%%%%%
\begin{definition}\label{dA.4}
The map $M: \bbC_+ \rightarrow \cB(\cH)$ is called a bounded
operator-valued Nevanlinna--Herglotz function on $\cH$ (in short, a bounded Nevanlinna--Herglotz operator on $\cH$) if $M$ is analytic on $\bbC_+$ and $\Im (M(z))\geq 0$ for all $z\in \bbC_+$.
\end{definition}
%%%%%%%%%%

Here we follow the standard notation
\begin{equation} \lb{A.37}
\Im (M) = (M-M^*)/(2i),\quad \Re (M) = (M+M^*)/2, \quad M \in \cB(\cH).
\end{equation}

Note that $M$ is a bounded Nevanlinna--Herglotz operator if and only if the scalar-valued functions
$(u,Mu)_\cH$ are Nevanlinna--Herglotz for all $u\in\cH$.

As in the scalar case one usually extends $M$ to $\bbC_-$ by
reflection, that is, by defining
\begin{equation}
M(z)=M(\overline z)^*, \quad z\in \bbC_-.   \lb{A.36}
\end{equation}
Hence $M$ is analytic on $\bbC\backslash\bbR$, but $M\big|_{\bbC_-}$
and $M\big|_{\bbC_+}$, in general, are not analytic
continuations of each other.

In contrast to the scalar case, one cannot generally expect strict
inequality in $\Im(M(\cdot))\geq 0$. However, the kernel of $\Im(M(\cdot))$
has the following simple properties recorded in \cite[Lemma\ 5.3]{GT00} (whose proof was kindly communicated to us by Dirk Buschmann) in the matrix-valued context. Below we indicate that the proof extends to the present infinite-dimensional situation (see also \cite[Proposition\ 1.2\,$(ii)$]{DM97} for additional results of this kind):

%%%%%%%%%%%
\begin{lemma} \lb{lA.5}
Let $M(\cdot)$ be a $\cB(\cH)$-valued Nevanlinna--Herglotz function.
Then the kernel $\cH_0 = \ker(\Im(M(z)))$ is independent of $z\in\bbC \backslash \bbR$.
Consequently, upon decomposing $\cH = \cH_0 \oplus \cH_1$,
$\cH_1 = \cH_0^\bot$, $\Im(M(\cdot))$ takes on the form
\begin{equation}
\Im(M(z))= \begin{pmatrix} 0 & 0 \\ 0 & N_1(z) \end{pmatrix},
\quad z \in \bbC_+,     \lb{A.38}
\end{equation}
where $N_1(\cdot) \in \cB(\cH_1)$ satisfies
\begin{equation}
N_1(z) \geq 0, \quad \ker(N_1) = \{0\}, \quad z\in\bbC_+.    \lb{A.39}
\end{equation}
\end{lemma}
%%%%%%%%%%%%
\begin{proof}
Pick $z_0 \in \bbC \backslash \bbR$, and suppose $f_0 \in \ker(\Im(M(z_0)))$. Introducing
$m(z) = (f_0,M(z) f_0)_{\cH}$, $z \in \bbC \backslash \bbR$, $m(\cdot)$ is a scalar
Nevanlinna--Herglotz function and $m(z_0) \in \bbR$. Hence the Nevanlinna--Herglotz
function $m(z) - m(z_0)$ has a zero at $z=z_0$, and thus must be a real-valued constant,
$m(z) = m(z_0)$, $z \in \bbC \backslash \bbR$. Since $(f_0, M(z)^* f_0)_{\cH} =
\ol{(f_0, M(z) f_0)_{\cH}} = \ol{m(z)} = m(z_0) \in \bbR$, $z \in \bbC \backslash \bbR$, one concludes that $(f_0, \Im(M(z)) f_0)_{\cH} = \pm \big\|[\pm \Im(M(z))]^{1/2} f_0\big\|_{\cH}^2 = 0$,
$z \in \bbC_{\pm}$, that is,
\begin{equation}
f_0 \in \ker\big([\pm \Im(M(z))]^{1/2}\big) = \ker(\Im(M(z))), \quad z \in \bbC_{\pm},
\end{equation}
and hence $\ker(M(z_0) \subseteq \ker(M(z))$, $z \in \bbC \backslash \bbR$. Interchanging
the role of $z_0$ and $z$ finally yields $\ker(M(z_0) = \ker(M(z))$, $z \in \bbC \backslash \bbR$.
\end{proof}
%%%%%%%%%%%%

Next we recall the definition of a bounded operator-valued measure (see, also
\cite[p.\ 319]{Be68}, \cite{MM04}, \cite{PR67}):

%%%%%%%%%%%%%%
\begin{definition} \lb{dA.6}
Let $\cH$ be a separable, complex Hilbert space.
A map $\Sigma:\mathfrak{B}(\bbR) \to\cB(\cH)$, with $\mathfrak{B}(\bbR)$ the
Borel $\sigma$-algebra on $\bbR$, is called a {\it bounded, nonnegative,
operator-valued measure} if the following conditions $(i)$ and $(ii)$ hold: \\
$(i)$ $\Sigma (\emptyset) =0$ and $0 \leq \Sigma(B) \in \cB(\cH)$ for all
$B \in \mathfrak{B}(\bbR)$. \\
$(ii)$ $\Sigma(\cdot)$ is strongly countably additive (i.e., with respect to the
strong operator  \hspace*{5mm} topology in $\cH$), that is,
\begin{align}
& \Sigma(B) = \slim_{N\to \infty} \sum_{j=1}^N \Sigma(B_j)   \lb{A.40} \\
& \quad \text{whenever } \, B=\bigcup_{j\in\bbN} B_j, \, \text{ with } \,
B_k\cap B_{\ell} = \emptyset \, \text{ for } \, k \neq \ell, \;
B_k \in \mathfrak{B}(\bbR), \; k, \ell \in \bbN.    \no
\end{align}
$\Sigma(\cdot)$ is called an {\it $($operator-valued\,$)$ spectral
measure} (or an {\it orthogonal operator-valued measure}) if additionally the following
condition $(iii)$ holds: \\
$(iii)$ $\Sigma(\cdot)$ is projection-valued (i.e., $\Sigma(B)^2 = \Sigma(B)$,
$B \in \mathfrak{B}(\bbR)$) and $\Sigma(\bbR) = I_{\cH}$. \\
$(iv)$ Let $f \in \cH$ and $B \in \mathfrak{B}(\bbR)$. Then the vector-valued
measure $\Sigma(\cdot) f$ has {\it finite variation on $B$}, denoted by
$V(\Sigma f;B)$, if
\begin{equation}
V(\Sigma f; B) = \sup\bigg\{\sum_{j=1}^N \|\Sigma(B_j)f\|_{\cH} \bigg\} < \infty,
\end{equation}
where the supremum is taken over all finite sequences $\{B_j\}_{1\leq j \leq N}$
of pairwise disjoint subsets on $\bbR$ with $B_j \subseteq B$, $1 \leq j \leq N$.
In particular, $\Sigma(\cdot) f$ has {\it finite total variation} if
$V(\Sigma f;\bbR) < \infty$.
\end{definition}
%%%%%%%%%%%%%%

We recall that due to monotonicity considerations, taking the limit
in the strong operator topology in \eqref{A.40} is equivalent to taking the limit with respect to the weak operator topology in $\cH$.

For relevant material in connection with the following result we refer the reader, for instance, to \cite{AL95}, \cite{AN75}, \cite{AN76}, \cite{ABT11},
\cite[Sect.\ VI.5,]{Be68}, \cite[Sect.\ I.4]{Br71}, \cite{Bu97}, \cite{Ca76},
\cite{De62}, \cite{DM91}--\cite{DM97}, \cite{HS98}, \cite{KO77}, \cite{KO78}, \cite{MM02}, \cite{MM04}, \cite{Na74}, \cite{Na77}, \cite{NA75}, \cite{Na87}, \cite{Sh71}, \cite{Ts92}, 
and the detailed bibliography in \cite{GWZ13b}.

%%%%%%%%%%%
\begin{theorem}
\rm {(\cite{AN76}, \cite[Sect.\ I.4]{Br71}, \cite{Sh71}.)} \lb{tA.7}
Let $M$ be a bounded operator-valued Nevanlinna--Herglotz function in $\cH$.
Then the following assertions hold: \\
$(i)$ For each $f \in \cH$, $(f,M(\cdot) f)_{\cH}$ is a $($scalar$)$
Nevanlinna--Herglotz function. \\
$(ii)$ Suppose that $\{e_j\}_{j\in\bbN}$ is a complete orthonormal system in $\cH$
and that for some subset of $\bbR$ having positive Lebesgue measure, and for all
$j\in\bbN$, $(e_j,M(\cdot) e_j)_{\cH}$ has zero normal limits. Then $M\equiv 0$. \\
$(iii)$ There exists a bounded, nonnegative $\cB(\cH)$-valued measure
$\Omega$ on $\bbR$ such that the Nevanlinna representation
\begin{align}
& M(z) = C + D z + \int_{\bbR} d\Omega (\lambda ) \bigg[\f{1}{\lambda-z}
- \f{\lambda}{\lambda ^2 + 1}\bigg],
\quad z\in\bbC_+,    \lb{A.42} \\
& \wti \Omega((-\infty, \lambda]) = \slim_{\varepsilon \downarrow 0}
\int_{-\infty}^{\lambda + \varepsilon} d \Omega (t)  \, (t^2 + 1)^{-1},  \quad
\lambda \in \bbR,   \lb{A.42A} \\
& \wti \Omega(\bbR) = \Im(M(i)) - D
= \int_{\bbR} d\Omega(\lambda) \, (\lambda^2 + 1)^{-1} \in \cB(\cH),   \lb{A.42a} \\
& C=\Re(M(i)),\quad D=\slim_{\eta\uparrow \infty} \,
\frac{1}{i\eta}M(i\eta) \geq 0,      \lb{A.42b}
\end{align}
holds in the strong sense in $\cH$. Here
$\wti\Omega (B) = \int_{B}  \big(1+\lambda^2\big)^{-1}d\Omega(\lambda)$, $B \in \mathfrak{B}(\bbR)$.   \\
$(iv)$ Let $\lambda _1,\lambda_2\in\bbR$, $\lambda_1<\lambda_2$. Then the
Stieltjes inversion formula for $\Omega $ reads
\begin{equation}\lb{A.43}
\Omega ((\lambda_1,\lambda_2]) f =\pi^{-1} \slim_{\delta\downarrow 0}
\slim_{\varepsilon\downarrow 0}
\int^{\lambda_2 + \delta}_{\lambda_1 + \delta}d\lambda \,
\Im (M(\lambda+i\varepsilon)) f, \quad f \in \cH.
\end{equation}
$(v)$ Any isolated poles of $M$ are simple and located on the
real axis, the residues at poles being nonpositive bounded operators in $\cB(\cH)$.  \\
$(vi)$ For all $\lambda \in \bbR$,
\begin{align}
& \slim_{\varepsilon \downarrow 0} \, \varepsilon
\Re(M(\lambda +i\varepsilon ))=0,    \lb{A.45} \\
& \, \Omega (\{\lambda \}) = \slim_{\varepsilon \downarrow 0} \,
\varepsilon \Im (M(\lambda + i \varepsilon ))
=- i \slim_{\varepsilon \downarrow 0} \,
\varepsilon M(\lambda +i\varepsilon).     \lb{A.46}
\end{align}
$(vii)$ If in addition $M(z) \in \cB_{\infty} (\cH)$, $z \in \bbC_+$, then the measure $\Omega$ in \eqref{A.42} is countably additive with respect to the
$\cB(\cH)$-norm, and the Nevanlinna representation \eqref{A.42} and the
Stieltjes inversion formula \eqref{A.43} as well as \eqref{A.45}, \eqref{A.46} hold
with the limits taken with respect to the $\|\cdot\|_{\cB(\cH)}$-norm. \\
$(viii)$ Let $f \in \cH$ and assume in addition that $\Omega(\cdot) f$ is of finite total
variation. Then for a.e.\ $\lambda \in \bbR$, the normal limits $M(\lambda + i0) f$
exist in the strong sense and
\begin{equation}
\slim_{\varepsilon \downarrow 0} M(\lambda +i\varepsilon) f
= M(\lambda +i 0) f = H(\Omega(\cdot) f) (\lambda) + i \pi \Omega'(\lambda) f,
\end{equation}
where $H(\Omega(\cdot) f)$ denotes the $\cH$-valued Hilbert transform
\begin{equation}
H(\Omega(\cdot) f) (\lambda) = \text{p.v.}\int_{- \infty}^{\infty} d \Omega (t) f \,
\f{1}{t - \lambda}
= \slim_{\delta \downarrow 0} \int_{|t-\lambda|\geq \delta} d \Omega (t) f \,
\f{1}{t - \lambda}.
\end{equation}
\end{theorem}
%%%%%%%%%%%%%

As usual, the normal limits in Theorem \ref{tA.7} can be replaced by nontangential ones.The nature of the boundary values of $M(\cdot + i 0)$
when for some $p>0$, $M(z) \in \cB_p(\cH)$, $z \in \bbC_+$, was clarified in detail in \cite{BE67}, \cite{Na89}, \cite{Na90}, \cite{Na94}. We also mention that
Shmul'yan \cite{Sh71} discusses the Nevanlinna representation \eqref{A.42}; moreover, certain special classes of Nevanlinna functions, isolated
by Kac and Krein \cite{KK74} in the scalar context, are studied by Brodskii
\cite[Sect.\ I.4]{Br71} and Shmul'yan \cite{Sh71}.

Our final result of this appendix offers an elementary proof of bounded invertibility of 
$\Im(M(z))$ for all $z \in \bbC_+$ if and only if this property holds for some 
$z_0 \in \bbC_+$:
 
%%%%%%%%%%%%%
\begin{lemma} \lb{lA.8}
Let $M$ be a bounded operator-valued Nevanlinna--Herglotz function in $\cH$. Then 
$[\Im(M(z_0))]^{-1} \in \cB(\cH)$ for some $z_0\in\bbC_+$ $($resp., 
$z_0 \in \bbC_-$$)$ if and only if $[\Im(M(z))]^{-1} \in \cB(\cH)$ for all $z\in\bbC_+$ 
$($resp., $z \in \bbC_-$$)$.
\end{lemma}
%%%%%%%%%%%%%
\begin{proof}
By relation \eqref{A.36}, it suffices to consider $z_0, z \in \bbC_+$, and because of 
Theorem\ \ref{tA.7}\,$(iii)$, we can assume that $M(z)$, $z \in \bbC_+$, has the 
representation \eqref{A.42}. 

Let $x_0,x\in\bbR$ and $y_0,y>0$, then there exists a constant $c\geq1$ such that
\begin{align}
\sup_{\la\in\bbR} \bigg(\f{(\la-x)^2+y^2}{(\la-x_0)^2+y_0^2}\bigg) \leq c,
\end{align}
since the function on the left-hand side is continuous and tends to $1$ as $\la\to\pm\infty$.
If $[\Im(M(x_0+iy_0)]^{-1} \in \cB(\cH)$, there exists $\delta>0$ such that 
$\Im(M(x_0+iy_0))\geq \delta I_\cH$, and hence, using $c\geq1$, $y>0$, and $\Omega\geq0$, 
one obtains
\begin{align}
\delta I_\cH &\leq \Im(M(x_0+iy_0)) = Dy_0 + \int_{\bbR} \f{y_0}{(\la-x_0)^2+y_0^2}\,d\Omega(\la) \no\\
&\leq \f{y_0}{y}\bigg[Dy + c\int_{\bbR} \f{y}{(\la-x)^2+y^2}\,d\Omega(\la)\bigg] \\
&\leq \f{y_0}{y}\bigg[\Im(M(x+iy))+(c-1)\int_{\bbR} \f{y}{(\la-x)^2+y^2}\,d\Omega(\la)\bigg] \leq \f{y_0}{y}\Im(M(x+iy)). \no
\end{align}
Thus, $\Im(M(x+iy))\geq (y/y_0) \delta I_\cH$, and hence $[\Im(M(x+iy))]^{-1} \in \cB(\cH)$.
\end{proof}
%%%%%%%%%%%%%

For a variety of additional spectral results in connection with operator-valued Nevanlinna--Herglotz functions we refer to \cite{BMN02} and \cite[Proposition~1.2]{DM97}. For a systematic treatment of  operator-valued Nevanlinna--Herglotz families we refer to \cite{DHM15}. 

\smallskip

%%%%%%%%%%%%%%%%%%%%%%%%%%%%%%%%%%%%%%%
\noindent {\bf Acknowledgments.}
We are indebted to Jussi Behrndt and Mark Malamud for numerous discussions on this topic. 
S.\,N.~is grateful to the Department of Mathematics of the University of Missouri where part of this work was completed while on a Miller Scholar Fellowship in February--March of 2014.
%%%%%%%%%%%%%%%%%%%%%%%%%%%%%%%%%%%%%%%

%%%%%%%%%%%%%%%%%%%%%%%%%%%%%%%%%%%%%%%

\end{document}